\documentclass[11pt]{article}

\usepackage[latin1]{inputenc}
\usepackage{t1enc}
\usepackage{amsmath, amssymb, amsfonts, amsthm, amsopn,epsfig}
\usepackage[none]{hyphenat}
\usepackage{float}
\usepackage{graphicx}
\usepackage{caption}
\usepackage[toc,page]{appendix}
\usepackage{epstopdf}
\usepackage{etoolbox}
\usepackage[colorlinks=true]{hyperref}
\usepackage{dsfont}
\hypersetup{
	colorlinks=true,
	linkcolor=blue,
	filecolor=magenta,
	urlcolor=cyan,
	citecolor=blue
}
\usepackage{cleveref}
\usepackage[top=21mm, bottom=21mm, left=21mm, right=21mm]{geometry}
\usepackage{comment}

\usepackage{tikz, tikz-3dplot}
\usetikzlibrary{calc, patterns,intersections, scopes}
\usetikzlibrary{positioning,arrows,shapes,decorations.markings,decorations.pathreplacing, decorations.pathmorphing,matrix,patterns}
\tikzstyle{vtx}=[circle,draw=black,fill=black,inner sep=0,minimum size=5pt,text=white,font=\footnotesize]

\theoremstyle{plain}
\newtheorem{theorem}{Theorem}[section]

\newtheorem{corollary}[theorem]{Corollary}
\newtheorem{claim}[theorem]{Claim}
\newtheorem{lemma}[theorem]{Lemma}

\newtheorem{proposition}[theorem]{Proposition}

\theoremstyle{definition}

\newcommand{\cA}{\mathcal{A}}
\newcommand{\cB}{\mathcal{B}}
\newcommand{\cC}{\mathcal{C}}
\newcommand{\cE}{\mathcal{E}}
\newcommand{\cF}{\mathcal{F}}

\newcommand{\cH}{\mathcal{H}}
\newcommand{\cL}{\mathcal{L}}
\newcommand{\cP}{\mathcal{P}}
\newcommand{\cR}{\mathcal{R}}

\newcommand{\cV}{\mathcal{V}}

\newcommand{\bR}{\mathbb{R}}
\newcommand{\bZ}{\mathbb{Z}}

\newcommand{\E}{\mathbb{E}}
\newcommand{\Pb}{\mathbb{P}}

\newcommand{\eps}{\varepsilon}

\newcommand{\hide}[1]{}

\title{
\vspace{-0.8cm}
$C_4$-free subgraphs of high degree with geometric applications}
\author{
Zach Hunter \thanks{Department of Mathematics, ETH Z\"urich, Switzerland. Email: {\tt \{zach.hunter, aleksa.milojevic, benjamin.sudakov\}@math.ethz.ch}. Research supported in part by SNSF grant 200021-228014.}
\and Aleksa Milojevi\'c\footnotemark[1] \and Benny Sudakov \footnotemark[1]
\and 
Istv\'an Tomon\thanks{Ume\r{a} University, \emph{e-mail}: \texttt{istvantomon@gmail.com}, Research supported in part by the Swedish Research Council grant VR 2023-03375.}}
\date{}

\begin{document}

\maketitle

\begin{abstract}
The Zarankiewicz problem, a cornerstone problem in extremal graph theory, asks for the maximum number of edges in an $n$-vertex graph that does not contain the complete bipartite graph $K_{s,s}$. While the problem remains widely open in the case of general graphs, the past two decades have seen significant progress on this problem for various restricted graph classes -- particularly those arising from geometric settings -- leading to a deeper understanding of their structure.

In this paper, we develop a new structural tool for addressing Zarankiewicz-type problems. More specifically, we show that for any positive integer $k$, every graph with average degree $d$ either contains an induced $C_4$-free subgraph with average degree at least $k$, or it contains a $d$-vertex subgraph with $\Omega_k(d^2)$ edges. As an application of this dichotomy, we propose a unified approach to a large number of Zarankiewicz-type problems in geometry, obtaining optimal bounds in each case.
\end{abstract}

\section{Introduction}

In 1953, Zarankiewicz \cite{Zara} asked for the maximum number of edges in an $n$-vertex graph $G$ that contains no copy of $K_{s,s}$, the complete bipartite graph with two parts of size $s$. Just a year later, K\H{o}v\'ari, S\'os, and Tur\'an \cite{KST54} proved that such graphs have at most $O_s(n^{2 - 1/s})$ edges---a bound that remains unimproved to this day, although it is known to be tight only for $s \in \{2, 3\}$. On the other hand, a simple application of the probabilistic deletion method yields a lower bound of $\Omega_s(n^{2 - 2/(s+1)})$ for $s\geq 4$ (see, e.g., \cite{AS}), which, up to polylogarithmic factors, is the best known to date.

Over the past two decades, Zarankiewicz-type problems have been extensively studied in various restricted graph families, motivated by questions from both combinatorial geometry and structural graph theory. In many geometric incidence problems, it is natural to study $K_{s,s}$-free graphs, as this condition often corresponds to an inherent geometric constraint. For instance, the incidence graph of points and lines in $\mathbb{R}^2$ is always $K_{2,2}$-free, while in point-hyperplane incidence graphs, avoiding $K_{s,s}$ reflects a natural non-degeneracy assumption \cite{AS07, MST}.

More generally, refined bounds for Zarankiewicz-type problems have been obtained in a variety of settings: for intersection graphs of geometric objects \cite{FoxPach08,FoxPach10,FoxPach14,KS,TZ}, incidence graphs involving points and geometric shapes \cite{CH23,FPSSZ,KS,MST}, semialgebraic and semilinear graphs \cite{BCSTT,FPSSZ,MST}, graphs of bounded VC-dimension \cite{FPSSZ,JP}, and for visibility graphs \cite{AK25,DM24}. In all these cases, the classical K\H{o}vari-S\'os-Tur\'an bound can be substantially improved. For example, it was shown in \cite{FPSSZ} that a $K_{s,s}$-free graph of VC-dimension $d$ has at most $O_{d, s}(n^{2 - 1/d})$ edges -- a significant improvement over the $O(n^{2 - 1/s})$ bound when $d \ll s$. In some cases, the bound is even linear: for example, \cite{FoxPach08} showed that a $K_{s,s}$-free intersection graph of $n$ curves in the plane has at most $O_s(n)$ edges. Beyond merely improving the bounds, these works also bring together a rich toolkit of techniques such as partitioning methods, variants of $\eps$-nets and separator theorems, that have deepened our understanding of the underlying geometry of these problem.

These results reveal an interesting general phenomenon: in many restricted graph families, the extremal behavior is not governed by $s$, as it is in general graphs, but instead reflects intrinsic structural parameters such as the ambient dimension or the VC-dimension, while $s$ may only appear in the constant factor.

In parallel with this line of work, a similar phenomenon has been observed in structural graph theory. One of the first results in this direction is due to K\"uhn and Osthus \cite{KO04}, who showed that if a $K_{s, s}$-free $n$-vertex graph $G$  does not contain an induced subdivision of a fixed graph $H$, then $G$ has at most $O_{H,s}(n)$ edges. A \textit{subdivision} of the graph $H$ is any graph which can be obtained by replacing the edges of $H$ by disjoint paths. Thus, the answer to Zarankiewicz's problem changes substantially in the family of graphs avoiding induced subdivisions of $H$.

Since then, analogous behavior has been observed in other families of interest, such as graphs avoiding an induced copy of a tree $T$ \cite{HMST,SSS23}, and families defined by forbidding an induced bipartite graph \cite{BBCD,HMST}. 
In fact, the general phenomenon can be described more abstractly, as follows. If $\mathcal{F}$ is a family of graphs closed under taking induced subgraphs and $\mathcal{F}$ excludes at least one bipartite graph $H$, then there exists $\eps = \eps(\mathcal{F}) > 0$ such that every $K_{s,s}$-free $n$-vertex graph in $\mathcal{F}$ has at most $O_{s}(n^{2 - \eps})$ edges. This was implicitly proven by \cite{BBCD, FPSSZ, GH}, and the value of $\eps$ was later improved for various families $\cF$ by \cite{AZ, HMST}.

In this paper, we focus on Zarankiewicz-type problems in geometrically defined graphs, where the extremal number is expected to be linear or nearly linear. We uncover a new strucutural result about graphs avoiding high degree induced $C_4$-free subgraphs, which we use to improve and unify a wide range of Zarankiewicz-type results for geometric graph families, obtaining tight bounds in each case.

\subsection{Intersection graphs of curves}\label{sec:intr_curv}

The \emph{intersection graph} of a family of sets or geometric objects $\mathcal{F}$ is the graph whose vertices correspond to the elements of $\mathcal{F}$, with an edge between any two sets that intersect. A particularly interesting class of intersection graphs are \emph{string graphs}, which are defined as the intersection graphs of families of planar curves (which are sometimes also called strings). Here, a \emph{planar curve} stands for the image of a continuous function $\phi: [0,1] \rightarrow \mathbb{R}^2$.

The study of string graphs dates back to the 1960s, when they were introduced independently by Benzer \cite{Benzer}, in the context of genetic mutations, and by Sinden \cite{Sinden}, in connection with the layout of integrated circuits. Both works posed the basic question of determining which graphs can be represented as string graphs. It is elementary to show that all planar graphs are string graphs, and  Even, Erlich and Tarjan \cite{EET} observed that a subdivision of any non-planar graph is not a string graph. However, giving a complete characterization of string graphs is difficult. The main issue is that the involved curves can be arbitrarily complex, making it hard to show that the recognition problem is even decidable. Remarkably, the decidability was established \cite{PT02, SS04} only in 2002, and the problem was later shown to be NP-complete \cite{SSS03}. Despite -- or perhaps because of -- these results, understanding the structure of string graphs remains an active area of research.

A good example of such a structural result about string graphs is the extension of the classical Lipton--Tarjan separator theorem. This theorem states that any planar graph on $n$ vertices has a balanced separator of size $O(\sqrt{n})$, where a \textit{balanced separator} is a set of vertices whose removal splits the graph into two connected components of roughly equal size. In \cite{FoxPach10}, Fox and Pach showed that every string graph with $m$ edges admits a balanced separator of size at most $O(m^{3/4} \log m)$. This result has several interesting applications; for instance, it implies that a $K_{s,s}$-free string graph on $n$ vertices has at most $2^{O(s)}n$ edges. Since string graphs do not contain subdivisions of non-planar graphs as induced subgraphs, and in particular no subdivision of $K_5$, the theorem of K\"uhn and Osthus \cite{KO04} discussed above directly shows that the number of edges in $K_{s, s}$-free string graphs is $O_s(n)$. However, the point of the work of Fox and Pach, together with the subsequent works, is to pinpoint the correct dependence on $s$.

The dependence on $s$ in this bound was subsequently improved to $s (\log s)^{O(1)}$ by Fox and Pach \cite{FoxPach14}, using a refined bound on the size of balanced separators due to Matou\v{s}ek \cite{Mat14}. Ultimately, the optimal bound of $O(s \log s)$ was established by Lee \cite{Lee17}, alongside optimal bounds on the size of balanced separators in string graphs.

Although the existence of small separators is a powerful tool, with a wide range of theoretical and algorithmic applications, it is only available for a limited selection of graph families. In this paper, we provide an alternative proof of the optimal $O(s \log s)$ bound which relies on our key graph-theoretic tool, Theorem~\ref{thm:main}, and avoids the use of separators altogether.

\begin{theorem}\label{thm:string}
Let $G$ be an $n$-vertex $K_{s,s}$-free string graph. Then $G$ has $O(s(\log s)n)$ edges.
\end{theorem}

For certain restricted families of curves, this bound can be further improved. A family of curves is called \emph{$k$-intersecting} if any two curves intersect in at most $k$ points. We further impose standard non-degeneracy conditions: no curve is self-intersecting, no three curves meet at a common point, and if two curves intersect, one properly passes to the other side of the other (i.e. there are no touching pairs of curves).

Fox and Pach \cite{FoxPach10} showed that a $K_{s, s}$-free intersection graph of an $n$-element $k$-intersecting family of curves has at most $O_k(sn)$ edges. Beyond separators, a key component of their proof is a result of Fox, Pach, and T\'oth \cite{FPT11}, which states that dense intersection graphs of $k$-intersecting families contain bicliques of linear size. Here, a \emph{biclique} is a complete bipartite graph with equal sized vertex classes. We prove a strengthening of this result, carrying it over to the bipartite setting, and give an alternative proof of the bound $O_k(sn)$. More precisely, we assume that $\cA, \cB$ are families of curves which are individually $k$-intersecting, and we make no assumptions on the number of intersections between a curve from $\cA$ and a curve from $\cB$. To state the following result, let us define the notion of the \emph{bipartite intersection graph} of two families $\cA$ and $\cB$, which is the bipartite graph with vertex classes $\cA$ and $\cB$, where there is an edge between $a\in\cA$ and $b\in \cB$ if $a$ and $b$ intersect.

\begin{theorem}\label{thm:stEH_kint}
Let $\mathcal{A}$ and $\mathcal{B}$ be $n$-element $k$-intersecting families of curves. If the bipartite intersection graph of $\cA$ and $\cB$ has at least $cn^2$ edges, then it contains a biclique of size $\Omega_{k,c}(n)$.
\end{theorem}

We note that a similar result to Theorem~\ref{thm:stEH_kint} is proved by Kor\'andi, Pach, and Tomon \cite{KPT}, under the addional assumption that $\cB$ is $1$-intersecting and that all curves are graphs of functions $f:[0, 1]\to \bR$. Also, recently, Fox, Pach, and Suk \cite{FPS23} proved the $k=1$ case of Theorem \ref{thm:stEH_kint} under the stronger assumption that $\mathcal{A}\cup\mathcal{B}$ is 1-intersecting. Therefore, Theorem \ref{thm:stEH_kint} also greatly strengthens both of these results. We can use Theorem~\ref{thm:stEH_kint} and our main graph theoretical lemma to give a simple proof of several results, including the $O_k(sn)$ bound for the number of edges in $K_{s, s}$-free intersection graphs of $k$-intersecting families.

\begin{theorem}\label{thm:curve_k_int}
    Let $G$ be the intersection graph of a $k$-intersecting family of curves. If $G$ is $K_{s,s}$-free, then $G$ has $O_{k}(sn)$ edges.
\end{theorem}

Furthermore, the case when $G$ is the intersection graph of the union of two families of disjoint curves is also interesting -- for example, it includes the intersection graphs of vertical and horiztonal segments in the plane, which have been recently studied in \cite{CKS24, KS}. In this case, we also obtain a tight bound.

\begin{theorem}\label{thm:curve_disjoint}
Let $\mathcal{A}$ and $\mathcal{B}$ be two $n$ element families of pairwise disjoint curves, and let $G$ be the intersection graph of $\mathcal{A}\cup\mathcal{B}$. If $G$ is $K_{s,s}$-free, then $G$ has $O(sn)$ edges.
\end{theorem}

Note that the difference between Theorems~\ref{thm:curve_k_int} and \ref{thm:curve_disjoint} is that the curves $a\in \cA, b\in \cB$ may intersect an arbitrary number of times in Theorem~\ref{thm:curve_disjoint}.
Finally, we can establish a similar result in the case of intersection graphs of convex sets.

\begin{theorem}\label{thm:curve_convex}
 Let $G$ be the intersection graph of a family of convex sets in the plane. If $G$ is $K_{s,s}$-free, then $G$ has $O(sn)$ edges.
\end{theorem}

\subsection{Pseudo-disks}

Given a ground set $X$ and a family of its subsets $\mathcal{F}\subseteq 2^X$, the \emph{incidence graph} of $(X,\mathcal{F})$ is the bipartite graph with vertex classes $X$ and $\mathcal{F}$, with an edge between $x\in X$ and $F\in \mathcal{F}$ if $x\in F$. Recently, Chan and Har-Peled \cite{CH23} presented a number of Zarankiewicz-type results for geometric incidence graphs, including incidence graphs of points and boxes, halfspaces, and pseudo-disks. A family of \emph{pseudo-disks} is a family of simple closed Jordan regions in the plane such that the boundaries of any two intersect in at most two points. Disks form a natural example, and more generally, any family of homothets of a fixed convex shape also constitutes a pseudo-disk family. In \cite{CH23}, it is proved that a $K_{s,s}$-free incidence graph of $n$ points and $n$ $y$-monotone pseudo-disks has at most $O(sn\log\log n)$ edges. They further posed the problem of improving this bound to $O_s(n)$. Here, a pseudodisk is called $y$-monotone if its intersection with any vertical line is either empty or it is an interval.

Soon afterward, Keller and Smorodinsky \cite{KS} improved their results as follows. Firstly, they showed that the number of edges in a $K_{s, s}$-free incidence graph of points and pseudodisks is bounded by $O(s^6 n)$. Furthermore, they extended this result by showing that if $\mathcal{A}$ and $\mathcal{B}$ are two $n$-element families of pseudo-disks, and the bipartite intersection graph $G$ between $\mathcal{A}$ and $\mathcal{B}$ is $K_{s,s}$-free, then $e(G)=O(s^6n)$. In the case of $y$-monotone pseudo-disks, we further improve these result by establishing the optimal dependence on $s$ as well.

\begin{theorem}\label{thm:pseudodisk optimal}
Let $\mathcal{A}$ and $\mathcal{B}$ be two $n$-element families of $y$-monotone pseudodisks. If the biparite intersection graph $G$ between $\mathcal{A}$ and $\mathcal{B}$ is $K_{s,s}$-free, then $G$ has $O(sn)$ edges.
\end{theorem}

\subsection{Semilinear graphs}

A graph $G$ is \emph{semialgebraic} of description complexity $(d,D,m)$ if the vertices of $G$ are points in $\mathbb{R}^d$, and whether $x$ and $y$ are joined by an edge is decided by a Boolean combination of $m$  inequalities $\{f_i(x,y)\leq 0\}$ for $i\in [m]$, where $f_i:\mathbb{R}^d\times \mathbb{R}^d\rightarrow\mathbb{R}$ is a polynomial of total degree at most $D$.
 Semilinear graphs form the special subfamily of semialgebraic graphs in which the defining polynomials are linear functions, that is, $D=1$. The study of Zarankiewicz problems for semilinear graphs was initiated by Basit, Chernikov, Starchenko, Tao, and Tran \cite{BCSTT}, and have recieved significant interest since then \cite{CH23,KS,T24,TZ}. In this paper, we establish optimal bounds for this problem, sharpening all previous results.
 
 Formally, a bipartite graph $\Gamma$ is \emph{semilinear} of dimension $(d_x,d_y)$ and complexity $(h,t)$ if the following holds. The vertex classes of $\Gamma$ are $X\subset \bR^{d_x}$ and $Y\subset \bR^{d_y}$, and there are $ht$ linear functions $f_{i, j}:\bR^{d_x}\times \bR^{d_y}\to \bR$, where $(i,j)\in [h]\times[t]$, such that two vertices $x\in X$ and $y\in Y$ are joined by an edge in $\Gamma$ if and only if 
 $$\bigvee_{j\in [t]}\bigwedge_{i\in [h]} \{f_{i, j}(x, y)\leq 0\}=\mbox{true}.$$
This definition might seem somewhat technical at first, so let us provide some examples. One of the most interesting families of semilinear graphs of bounded dimension and complexity are the incidence graphs of points and axis-parallel boxes in $\mathbb{R}^d$. Indeed, given a point $x=(x_1,\dots,x_d)$ and an axis parallel box $B=[a_1,b_1]\times\dots\times[a_d,b_d]$, we have $x\in B$ if $a_i-x_i\leq 0$ and $x_i-b_i\leq 0$ for every $i\in [d]$. Therefore, we can represent the boxes using $2d$-dimensional vectors, and so incidence graphs of points and boxes are semilinear of dimension $(d,2d)$ and complexity $(2d,1)$. Further examples include interval graphs, shift graphs, circle graphs, and intersection graphs of boxes.

Another semilinear graph of interest is the incidence graph of points and \emph{corners} in $\mathbb{R}^d$, where a corner is a set of the form $[-\infty,b_1]\times \dots\times [-\infty,b_d]$. These graphs are semilinear of dimension $(d,d)$ and complexity $(d,1)$. It is not difficult to show that a semilinear graph of dimension $(d_1,d_2)$ and complexity $(h,t)$ is the union of $t$ incidence graphs of points and corners in $\mathbb{R}^h$. Basit, Chernikov, Starchenko, Tao and Tran \cite{BCSTT} and Tomon and Zakharov \cite{TZ} proved that if the point-corner incidence graph has $n$ points and is $K_{s,s}$-free, then it has $O_{h,s}\big(n (\log n)^h\big)$ edges, which implies a similar bound for the number of edges of $K_{s,s}$-free $n$-vertex semilinear graph of complexity $(h,t)$.

In the case of incidence graphs of points and boxes in $\mathbb{R}^d$, the resulting upper bound $O_{d,s}(n(\log n)^{2d})$ was improved by Chan and Har-Peled \cite{CH23} to $O_{d}(sn(\log n/\log \log n)^{d-1})$, who also established that this bound is the best possible for $s=2$, see also \cite{T24} for a matching lower bound construction. They used this result to show that if an $n$-vertex $K_{s,s}$-free semilinear graph $\Gamma$ of complexity $(h,1)$ can be defined with linear functions $f_{1,1},\dots,f_{h,1}$ such that there are $\delta$ non-parallel halfspaces among the halfspaces $\{y:f_{i,1}(0,y)\leq 0\}$, then $\Gamma$ has $O_{\delta}(sn(\log n/\log \log n)^{\delta-1})$ edges. We greatly improve these results by showing that the exponent of the $\log $ term can be bounded by the dimension alone. 

\begin{theorem}\label{thm:semilinear zarankiewicz}
Let $\Gamma$ be an $n$-vertex $K_{s,s}$-free semilinear graph of dimension $(d_x, d_y)$, where $d_x\geq d_y\geq 1$, and complexity $(h, t)$. Then $$e(\Gamma)\leq O_{d_x, h}\left(t s n \Big(\frac{\log n}{\log\log n}\Big)^{d_y-1}\right).$$
\end{theorem}

The bound in Theorem~\ref{thm:semilinear zarankiewicz} is very close to optimal. Let $m=2n/s$, then as mentioned above, there exist $K_{2,2}$-free $m$-vertex incidence graphs of points and boxes in $\mathbb{R}^d$ with $\Omega_d(m(\log /\log\log m)^{d-1})$ edges. Such graphs are semilinear of dimension $(d,2d)$ and complexity $(2d,1)$. Repeating each point and box $s/2$ times, and perturbing them slightly to avoid repetitions, we get an $n$-vertex $K_{s,s}$-free incidence graph with $\Omega_d\Big(s^2m\big(\frac{\log m}{\log\log m}\big)^{d-1}\Big)=\Omega_d\Big(sn\big(\frac{\log n/s}{\log\log n/s}\big)^{d-1}\Big)$ edges.

\subsection{Polygon visibility graphs}

The study of the following geometric Zarankiewicz problem was recently initiated by Du and McCarty \cite{DM24}. For a (closed) Jordan curve $K$ and a set of points $P\subset K$, the \emph{visibility graph} of $P$ with respect to $K$ is the graph with vertex set $P$, in which $x,y\in P$ are joined by an edge if the straight line segment connecting $x$ and $y$ is completely contained in the region bounded by $K$, which we denote by $K^*$. Du and McCarty \cite{DM24} conjecture that the maximum number of edges in a $K_{s, s}$-free polygon visibility graph is $O_s(n)$. This is partially motivated by the result of Davies, Krawczyk, McCarty, and Walczak \cite{DKMW} that this family is $\chi$-bounded. The notion of $\chi$-boundedness is central in structural graph theory, where a family $\cF$ is said to be $\chi$-bounded if the chromatic number $\chi(G)$ of any graph $G\in \cF$ is bounded by a function of the clique number $\omega(G)$, i.e. if there exists a function $f:\bZ_{>0}\to \bZ_{>0}$ such that $\chi(G)\leq f(\omega(G))$ for all $G\in \cF$.

The question of Du and McCarty is still open, but Ackerman and Keszegh \cite{AK25} established a number of partial results towards it. Say that a Jordan curve is \emph{$x$-monotone}, if every vertical line intersects $K$ in at most 2 points, and say that it is \emph{star-shaped}, if there is a point $c\in K^*$ such that every half-line starting at $c$ intersects $K$ in one point. In \cite{AK25}, it is proved that if $G$ is an $n$-vertex visibility graph with respect to an $x$-monotone or star-shaped curve, and $G$ is $K_{s,s}$-free, then $G$ has at most $O_s(n)$ edges. We improve both of these results as follows.

\begin{theorem}\label{thm:polygon_visibility}
Let $G$ be an $n$-vertex visibility graph with respect to a Jordan curve $K$ that is either $x$-monotone or star shaped. If $G$ contains no $K_{s,s}$, then $G$ has $O(sn)$ edges. 
\end{theorem}

In order to prove this theorem, we establish general results about $K_{s, s}$-free ordered graphs avoiding some fixed ordered matching, which might be of independent interest.

\subsection{Key graph-theoretic tool: finding $C_4$-free subgraphs or dense patches}

Our main theoretical contribution, and the unifying theme of this work, is the following statement, which says that every graph $G$ either contains a $C_4$-free induced subgraph with large average degree or a dense subgraph, whose size is at least the average degree of $G$.

\begin{theorem}\label{thm:main}
For every positive integer $k$ there exists $c=c(k)>0$ such that the following holds. Every graph $G$ of average degree at least $d$ either contains an induced $C_4$-free subgraph with average degree at least $k$, or it contains a subgraph on $d$ vertices with at least $cd^2$ edges.
\end{theorem}

In a typical application, we attempt to bound the average degree of a $K_{s, s}$-free graph from some geometrically defined class by $d\leq O(s)$, which corresponds to bounding the number of edges by $O(sn)$. To do this, the first step is to bound the number of edges of a $C_4$-free graph from this family by $Cn$, for some absolute constant $C$, which corresponds to addressing the case when $s=2$. Then, given a $K_{s, s}$-free graph $G$ from this geometric family, we apply our Theorem~\ref{thm:main} with any $k>C$, and observe that, since $G$ does not contain $C_4$-free subgraphs of average degree at least $k$, it must contain a subgraph on $d$ vertices with $cd^2$ edges. The second step is to show that in such dense graphs one can find bicliques of size $\Omega(d)$. In particular, this means that $s\geq \Omega(d)$, i.e. that $d\leq O(s)$.

Beyond the geometric applications discussed so far, Theorem~\ref{thm:main} has a tight connection with the notion of degree-boundedness in structural graph theory. A hereditary family $\mathcal{F}$ is called \emph{degree-bounded} if every $K_{s,s}$-free $n$-vertex member of $\mathcal{F}$ has $O_{\mathcal{F},s}(n)$ edges, or equivalently, average degree $O_{\mathcal{F},s}(1)$. For example, an equivalent way to phrase the aforementioned theorem of K\"uhn and Osthus \cite{KO04} is that the family of graphs without induced $H$-subdivisions is degree-bounded. As we previously discussed, degree-bounded families appear in a number of places, including the family of graphs avoiding an induced copy of a tree $T$ \cite{HMST,SSS03}, string graphs \cite{FoxPach08,FoxPach10,FoxPach14}, incidence graphs of points and halfspaces in $\mathbb{R}^3$ \cite{CH23}, and visibility graphs with respect to certain Jordan curves \cite{AK25}. 

This has lead to significant interest in understanding general properties of degree-bounded families. The first result in this direction is due to McCarty \cite{MC21}, who showed the following, perhaps surprising result. If $\mathcal{F}$ is a hereditary family for which there exists a positive integer $k$ such that every $C_4$-free member of $\mathcal{F}$ has average degree at most $k$, then $\mathcal{F}$ is degree-bounded. This was later strengthened by Gir\~ao and Hunter \cite{GH} (see also the related independent work by Bourneuf, Buci\'c, Cook and Davies \cite{BBCD}), who showed that if $\mathcal{F}$ is degree-bounded, then there exists $c=c(\mathcal{F})>0$ such that every $K_{s,s}$-free member of $\mathcal{F}$ has average degree at most $O_{\mathcal{F}}(s^c)$. In other words, these works say that (1) in order to show that $\mathcal{F}$ is degree-bounded, it is enough to consider $C_4=K_{2,2}$-free members of $\mathcal{F}$, and (2) if $\mathcal{F}$ is degree-bounded, then it is automatically \emph{polynomially} degree-bounded. 

In light of these results, Theorem \ref{thm:main} is equivalent with the statement that if $\mathcal{F}$ is a degree-bounded family, then there exist $c=c(\mathcal{F})>0$ such that if $G\in \mathcal{F}$ has average degree $d$, then $G$ contains a subgraph with $d$ vertices and edge density at least $c$. Here, the edge-density of a graph $G$ is defined as $e(G)/\binom{|V(G)|}{2}$, i.e. as the number of edges of $G$ divided by maximum number of edges in a graph with the same vertex set.

We now describe from another perspective how to apply Theorem \ref{thm:main} to get sharp bounds on Zarankiewicz problems. Say that a family $\mathcal{F}$ of graphs is \emph{weakly degree-bounded} if there exists $k$ such that every $C_4$-free member of $\mathcal{F}$ has average degree at most $k$. As we pointed out, for hereditary families being weakly degree-bounded is equivalent to being degree-bounded. However, we introduce this weaker notion to emphasize that it is enough to consider $C_4$-free members of the family. Furthermore, say that a family of graphs $\mathcal{F}$ has the \emph{density-EH property} (where EH is an abbreviation for Erd\H{o}s-Hajnal) if for every $c>0$ there exists $\delta=\delta(c)>0$ such that for every $G\in \mathcal{F}$, where $G$ has edge density at least $c$, $G$ contains a biclique of size at least $\delta v(G)$. 

\begin{corollary}\label{thm:master}
Let $\mathcal{F}$ be a hereditary family that is both weakly degree-bounded and has the density EH-property. Then there exists a constant $C=C(\mathcal{F})>0$ such that if $G\in\mathcal{F}$ is $K_{s,s}$-free, then the average degree of $G$ is at most $Cs$.
\end{corollary}
\begin{proof}
As $\mathcal{F}$ is weakly degree-bounded, there exists $k$ such that every $C_4$-free member of $\mathcal{F}$ has average degree less than $k$. Let $G\in\mathcal{F}$ be a $K_{s,s}$-free graph, and let $d$ be the average degree of $G$. If $c=c(k)>0$ is the constant promised by Theorem \ref{thm:main}, then $G$ contains an induced subgraph $H$ on $d$ vertices with at least $cd^2$ edges. But as $\mathcal{F}$ has the density-EH propery,  there exists $\delta=\delta(c)>0$ such that $H$ contains a biclique of size at least $\delta v(H)=\delta d$. Therefore, $s>\delta d$, which gives $d\leq s/\delta$, showing that the theorem holds with $C=1/\delta$.
\end{proof}

Corollary \ref{thm:master} turns out to be highly applicable in geometric settings. One reason for this is that semialgebraic graphs have the density-EH property. It is proved by Fox, Gromov, Lafforgue, Naor and Pach \cite{FGLNP} that for any fixed triple $(d,D,m)$, the family of semialgrabraic-graphs of description complexity $(d,D,m)$ has the density-EH property. This provides one half of the requirements of Theorem \ref{thm:master}. However, these families are not necessarily weakly degree-bounded (e.g. incidence graphs of points and lines), so this property depends more on the specific geometric structure.

\subsection{Graphs without dense subgraphs}

Finally, we highlight a strengthening of Theorem \ref{thm:main}, which can be used to settle a recent conjecture of Fox, Nenadov, and Pham \cite{FNP}. A natural relaxation of $K_{s,s}$-freeness, introduced in~\cite{FNP}, is the notion of $(c, t)$-sparsity. A graph $G$ is said to be \textit{$(c, t)$-sparse} if for all subsets $A, B \subseteq V(G)$ with $|A|, |B| \ge t$, there are at most $e(A, B) \leq (1 - c)|A||B|$ edges between them. It is easy to see that a graph is $K_{s, s}$-free if and only if it is $(1/s^2, s)$-sparse. Another motivation for studying $(c, t)$-sparse graphs is that random graphs $G(n, p)$ are $(c, t)$-sparse for any $c < 1 - p$ and $t = \Omega_{p,c}(\log n)$.

In~\cite{FNP}, the authors showed that if $c > 0$ is a constant and $T$ is a tree, then every $(c, t)$-sparse graph with at least $Ctn$ edges contains an induced copy of $T$, where $C=C(T,c)$ is sufficiently large with respect to $T$ and $c$. They conjectured that the same statement holds if instead of a tree $T$, we seek an induced subdivision of a fixed graph $H$. We prove this conjecture.

\begin{theorem}
For every $c > 0$ and every graph $H$, there exists a constant $C$ such that if $G$ is a $(c, t)$-sparse graph with average degree at least $Ct$, then $G$ contains an induced subdivision of $H$.
\end{theorem}

In particular, this theorem immediately follows from the following strengthening of Theorem \ref{thm:main}, after  recalling that $C_4$-free graphs of sufficiently large average degree contain and induced subdivision of $H$ (K\"uhn and Osthus \cite{KO04}). In Section \ref{sec:graph theory}, we show that the next theorem indeed implies Theorem \ref{thm:main}.

\begin{theorem}\label{thm:main2}
 Let $k\ge 1, c>0$, then there exists $C=C(k,c)$ such that the following holds. Let $G$ be a $(c,t)$-sparse graph $G$ with average degree at least $Ct$. Then $G$ contains a bipartite $C_4$-free induced subgraph of average degree at least $k$. 
\end{theorem}

\paragraph{Organization.} In Section~\ref{sec:short applications}, we begin by presenting how Theorem~\ref{thm:main} can be used to give short proofs of tight bounds for the Zarankiewicz problem for (1) intersection graphs of horizontal and vertical segments in the plane, (2) point-halfspace incidence graphs and (3) string graphs. Then, in Section~\ref{sec:graph theory}, we present the proof of Theorem~\ref{thm:main} and Theorem~\ref{thm:main2}. In Section~\ref{sec:curves}, we discuss the proofs of Theorems~\ref{thm:stEH_kint} and~\ref{thm:pseudodisk optimal}. In Section~\ref{sec:semilinear}, we show the proof of our bounds for the semilinear Zarankiewicz problem, i.e. we prove Theorem~\ref{thm:semilinear zarankiewicz}. We conclude by discussing polygon visibility graphs in Section~\ref{sec:polygon visibility} and proving Theorem~\ref{thm:polygon_visibility}.

\section{Short applications}\label{sec:short applications}

As a warm-up, we present a few simple applications of Theorem \ref{thm:main}, which will be proved in the next section. First, we show that given a collection of $m$ vertical and $n$ horizontal segments on the plane such that the intersection graph $G$ is $K_{s,s}$-free, then $G$ has $O(s(m+n))$ edges. This result may sound innocent at first, and it also follows from any of Theorems \ref{thm:curve_k_int},  \ref{thm:curve_disjoint}, or \ref{thm:curve_convex}, but we were unable to find a completely elementary proof. Also, this statement is quite useful -- for its applications see \cite{AK25, CKS24, KS}.

Then, we provide an alternate proof of the result of Chan and Har-Peled \cite{CH23} on incidence graphs of points and halfspaces in $\mathbb{R}^3$. Finally, we prove the Zarankiewicz-type results concerning intersection graphs of curves which were discussed in Section \ref{sec:intr_curv}.

As alluded in the introduction, in each of the applications, we need to address the $C_4$-free case and the dense case. While there are several ways to resolve the dense case, we found that applying the following result about semialgebraic graphs gives a unified approach, applicable to a wide variety of settings. Formally, a bipartite graph $G$ with vertex classes $X,Y$ is \emph{semialgebraic} of description complexity $(d,D,m)$ if  $X,Y\subset \mathbb{R}^d$, and there exist $m$ polynomials $f_1,\dots,f_m:\mathbb{R}^d\times\mathbb{R}^d\rightarrow\mathbb{R}$, each  of total degree at most $D$, and a Boolean function $\phi:\{\mbox{false,true}\}^m\rightarrow\{\mbox{false,true}\}$ such that for every $x\in X$ and $y\in Y$, $x$ and $y$ are joined by an edge if and only if 
$\phi(\{f_i(x,y)\leq 0\}_{i\in [m]})=\mbox{true}$. The following result first appeared in \cite{FGLNP} by Fox, Gromov, Lafforgue, Naor and Pach, and later quantitatively sharpened by Fox, Pach, Sheffer, Suk and Zahl \cite{FPSSZ}.

\begin{lemma}\label{lemma:semialgrabraic}
For every triple of integers $(d,D,m)$ and $\beta\in (0,1]$, there exists $\alpha=\alpha(d,D,m,\beta)$ such that the following holds. Let $G$ be an $n$-vertex semialgrebraic graph of description complexity $(d,D,m)$. If $e(G)\geq \beta n^2$, then $G$ contains a biclique of size $\alpha n$. 
\end{lemma}

\subsection{Vertical and horizontal segments}

\begin{proposition}\label{prop:vertical vs horizontal}
Let $\cA$ be a family of vertical line segments and let $\cB$ be a family of horizontal line segments in $\bR^2$, such that $|\cA|=|\cB|=n$. If the bipartite intersection graph of $\cA$ and $\cB$ is $K_{s, s}$-free, then it has $O(sn)$ edges.
\end{proposition}
\begin{proof}
After a small perturbation, we may assume that no two elements of $\cA$ intersect, and no two elements of $\cB$ intersect. Let $G$ be the bipartite intersection graph of $\cA$ and $\cB$, and let $d$ be the average degree of $G$. By Theorem~\ref{thm:main}, for any integer $k$, $G$ either contains a $C_4$-free induced subgraph $G'\subseteq G$ with $d(G')\geq k$, or a dense induced subgraph $G'\subseteq G$ on $d$ vertices with $d(G')\geq cd$, for some small constant $c=c(k)>0$.

The family of intersection graphs of horizontal and vertical segments is a subfamily of the family of string graphs. As no string graph contains any subdivision of the 1-subdivision of $K_5$ as an induced subgraph (see e.g. \cite{EET, FoxPach10}), the result of K\"uhn and Osthus \cite{KO04} implies that any $C_4$-free intersection graph of vertical and horizontal segments has average degree at most $O(1)$. Hence, there exists an absolute constant $k$ such that $G$ contains no $C_4$-free induced sugbgraph of average degree at least $k$. 

By Theorem~\ref{thm:main}, there exists $c>0$ such that $G$ contains an induced subgraph $G'$ on $d$ vertices with $d(G')\geq cd$ for some $c>0$. But observe that $G$ and $G'$ are semialgebraic graphs. Indeed, a horizontal segment $s$ with endpoints $(a,h)$ and $(b,h)$ can be represented by the point $(a,b,h)\in \mathbb{R}^3$, while a vertical segment $t$ with endpoints $(v,x)$ and $(v,y)$ by $(v,x,y)\in\mathbb{R}^3$. Then $s$ and $t$ intersect if and only if the four inequalities $a-v\leq 0$, $v-b\leq 0$, $x-h\leq 0$, and $h-y\leq 0$ are all satisfied. Therefore, we can apply Lemma \ref{lemma:semialgrabraic} to conclude that there is some constant $\alpha>0$ such that $G'$ contains an $\alpha d$ sized biclique. Hence, $s\geq \alpha d$, giving the desired result.
\end{proof}

\subsection{Points and halfspaces}

A simple, yet interesting, class of geometric graphs are the incidence graphs of points and halfspaces in $\bR^d$. The Zarankiewicz problem for this class of graphs was studied by Chan and Har-Peled in \cite{CH23}, who showed that a $K_{s, s}$-free incidence graph of points and half-spaces in $\bR^3$ has at most $O(sn)$ edges. In this section, we give a proof of this result using our graph-theoretic lemma.

Chan and Har-Peled also considered the case $d>3$, for which they proved that if $G$ is a $K_{s,s}$-free incidence graph of $n$ points and $n$ halfspaces in $\mathbb{R}^d$, then $G$ has $O_s(n^{2-2/(\lfloor d/2\rfloor+1)})$ edges. A construction provided in their paper shows that for $d=5$, the resulting bound $O_s(n^{4/3})$ is the best possible. Hence, in case $d\geq 5$, the family $\cF_d$ of incidence graphs of points and halfspaces in $\mathbb{R}^d$ is not degree-bounded, in contrast with the case $d\leq 3$. This left the case of dimension 4 open, which we resolve at the end of this section, by showing that $\cF_4$ is not degree-bounded. More precisely, we construct $C_4$-free point-halfspace incidence graphs in $\bR^4$ with at least $\Omega(n\log n/\log\log n)$ egdes. We start with the case $d=3$.

\begin{theorem}\label{thm:point-halfspaces}
Let $\cP$ be a set of $n$ points and $\cH$ be a set of $n$ halfspaces in $\bR^3$. If the incidence graph of $\cP$ and $\cH$ is $K_{s, s}$-free, then there are at most $O(sn)$ incidences.
\end{theorem}

As discussed in the introduction, in order to establish Theorem~\ref{thm:point-halfspaces}, we need to show that the family of point-halfspace incidence graphs is degree-bounded and has the density-EH property. We first discuss the case of $C_4$-free incidence graphs.

\begin{lemma}\label{lemma:point-halfspaces C_4-free}
Let $\cP$ be a set of points and $\cH$ a set of halfspaces in $\bR^3$ such that the incidence graph of $\cP$ and $\cH$ is $C_4$-free. Then the incidence graph of $(\cP, \cH)$ has at most $O(|\cP|+|\cH|)$ edges. 
\end{lemma}
\begin{proof}
A halfspace $h$ with equation $ax+by+cz\geq d$ is \textit{upwards-oriented} if $c > 0$, and if  $c<0$, we call $h$ \textit{downwards-oriented}. By choosing a generic coordinate system, we ensure that all halfspaces are either upwards- or downwards-oriented. Due to symmetry, it suffices to bound the number of incidences of $\cP$ with the set $\cH_u$ of upwards-oriented halfspaces. 

Given a subdivision of a graph $H$, call the vertices corresponding to the vertices of $H$ as \emph{root vertices}. The following claim is the key geometric input of the proof.

\begin{claim}\label{claim:no K_5 subdivision}
If $\mathcal{P}'$ is a set of points and $\mathcal{H}'$ is a set of upwards-oriented halfpsaces, then the incidence graph of $(\mathcal{P}',\mathcal{H}')$ is not isomorphic to any subdivision of $K_5$, where all root vertices are contained in $\mathcal{P}'$.
\end{claim}
\begin{proof}
 Define an auxiliary graph $\Gamma$ on the vertex set $\cP'$, where we connect two points if there is a halfspace in $\cH'$ containing both. Then, $\Gamma$ contains a subdivision of $K_5$, and so it is not planar. In particular, if $\pi:\bR^3\to \bR^2$ denotes the projection map onto the $xy$ plane, then there exists a pair of disjoint edges $p_1q_1, p_2q_2\in E(\Gamma)$ such that $\pi(\overline{p_1q_1})$ and $\pi(\overline{p_2q_2})$ intersect. Hence, there are points $r_1\in \overline{p_1q_1}, r_2\in \overline{p_2q_2}$ for which $\pi(r_1)=\pi(r_2)$. Without loss of generality, we assume that $r_2$ lies above $r_1$. If $h\in \cH'$ denotes the upwards-oriented halfspace containing $p_1, q_1$, then $h$ contains $r_1$, and thus $r_2$ as well. Then, $h$ must contain one of $p_2$ and $q_2$, showing that $h$ is incident to at least three points of $\cP'$. Thus the incidence graph of $(\mathcal{P}',\mathcal{H}')$ is not isomorphic to any subdivision of $K_5$, where all root vertices are contained in $\mathcal{P}'$, otherwise every element of $\mathcal{H}'$ has degree 2.
\end{proof} 

The above claim shows that the incidence graph of $\cP$ and $\cH_u$ contains no induced subdivision of $K_5$ with all root vertices in $\cP$. Using duality, it is easy to conclude that the incidence graph of $\cP$ and $\cH_u$ also contains no induced subdivisions of $K_5$, where all root verties are in $\mathcal{H}_u$. 

Indeed, we may translate the whole configuration so that all points of $\cP$ have positive $z$ coordinates and so that each halfspace $h\in \cH_u$ is represented by the equation $ax+by+cz\geq 1$, without changing the incidence graph. By applying a duality transformation, we can map each point $x\in \mathcal{P}$ to the halfspace given by $\langle x, \cdot\rangle \geq 1$, and each halfspace given by $\langle \cdot, y\rangle \geq 1$ to the point $y$. This duality turns all points of $\cP$ into upwards-oriented halfspaces, and the halfspaces of $\cH_u$ into points, while maintaining the incidence structure. 

From this, we conclude that the incidence graph of $(\mathcal{P},\mathcal{H}_u)$ contains no induced subdivision of $K_9$. Otherwise, there is either an induced subdivsion of $K_5$ with all root vertices in $\cP$, or all root vertices in $\cH_u$. The proof now follows immediately from the result of K\"uhn and Osthus~\cite{KO04}, which states that the family of graph containing no induced subdivision of $K_t$ for any integer $t$ is degree-bounded. 
\end{proof}

Next, we observe that the family of point-halfspace incidence graphs have the density-EH property. Such graphs are semialgebraic of complexity $(3,2,1)$, so this follows immediately from Lemma \ref{lemma:semialgrabraic}.

\begin{lemma}\label{lemma:point-halfspaces dense}
Let $\cP$ be a set of at most $d$ points and $\cH$ a set of at most $d$ halfspaces in $\bR^3$, with at least $cd^2$ incidences. Then the incidence graph of $(P,\mathcal{H})$ contains $K_{s, s}$ for some $s\geq \Omega_c(d)$.
\end{lemma}

\begin{proof}[Proof of Theorem~\ref{thm:point-halfspaces}.]
Let $G$ be the incidence graph of $\cP$ and $\cH$, and  assume for contradiction that the average degree of $G$ is at least $d=d(G)\geq Cs$, for a large constant $C$. By Theorem~\ref{thm:main}, for any integer $k$, $G$ either contains a $C_4$-free induced subgraph $G'\subseteq G$ with $d(G')\geq k$, or a dense induced subgraph $G'\subseteq G$ on $d$ vertices with $d(G')\geq cd$, for some $c=c(k)>0$.

If $k$ is sufficiently large, Lemma~\ref{lemma:point-halfspaces C_4-free} shows that no $C_4$-free incidence graph of points and halfspaces has average degree $k$. Thus, we must have the second outcome, i.e. there exists an induced subgraph $G'\subseteq G$ on $d$ vertices with $e(G')\geq cd^2$. From Lemma~\ref{lemma:point-halfspaces dense}, we get that $G'$ contains a biclique of size $\Omega(d)$, which is a contradiction to the assumption that $G$ is $K_{s, s}$-free and $d(G)\geq Cs$ for an appropriately chosen constant $C$. 
\end{proof}

We conclude the section by showing that the family of point-halfspace incidence graphs in $\bR^4$ is not degree-bounded.

\begin{claim}
For every integer $n$, there exists a $C_4$-free incidence graph of $n$ points and $n$ halfspaces in $\mathbb{R}^4$ with $\Omega(n\log n/\log\log n)$ edges.
\end{claim}
\begin{proof}[Proof sketch.]
This proof is based on the construction of incidence graphs of points and axis-parallel rectangles in $\mathbb{R}^2$ with $\Omega(n\log n/\log\log n)$ edges, which is originally due to Chazelle \cite{Cha90} and was later developed in \cite{BCSTT,CH23,T24}. We argue that the rectangles can be replaced by their inscribed axis-aligned ellipses without losing many incidences. Finally, we use a lifting trick to transform this configuration into an incidence graph of points and halfspaces in $\mathbb{R}^4$.

The first step of our construction is to obtain a collection $\cR$ of axis-parallel rectangles of somewhat large area and comparatively small intersections. To do this, we pick a parameter $1\leq m\leq n$ and set $u=(\log m)^4$ and $k=\log m/\log u$, noting that $u^k=m$. Then, we choose $m$ such that $km=n$. For $i=1,\dots,k$, let us tile the unit square $[0,1]^2$ with $m=u^k$ rectangles of size $u^{-i}\times u^{-(k-i)}$. If we denote by $\cR_i$ the collection of rectangles in this tiling and let $\cR=\bigcup_{i=1}^k \cR_i$, then we have $|\cR|=mk=n$. Furthermore, let $\cE$ be the set of axis-parallel inscribed ellipses of the elements of $\cR$. Note that each element of $\cR$ has area $u^{-k}=1/m$, and thus each element of $\cE$ has area $\pi/4m$. The key property of this collection is that the intersection of any two rectangles of $\cR$ has area at most $u^{-k-1}=1/um$, and the same is true for any two ellipses of $\cE$. 

Let $\cP_0$ be a set of $n$ points chosen from the uniform distribution on $[0,1]^2$. Then each ellipse $E\in \cE$ contains $n\cdot \pi /4m=\pi k/4$ elements of $\cP$ in expectation. Furthermore, note that if two points $p,p'\in\cP_0$ are contained in two different ellipses of $\cE$, then the smallest axis-parallel rectangle $R_{p, p'}$ containing $p$ and $p'$ span a rectangle of area at most $1/um$. Therefore, given $p, p'\in \cP_0$, let us say that $p$ and $p'$ are \emph{close} if the area of the rectangle spanned by $p$ and $p'$ is less than $1/um$. To bound the probability that two random points $p=(p_x, p_y)$ and $p'=(p'_x, p'_y)$ are close, we fix $p$ and bound the area of the set $\{(p_x', p_y')\in \bR^2 : |p_x'-p_x|\cdot |p_y'-p_y|\leq 1/um\}\cap [0, 1]^2$ (geometrically, this is the intersection of the unit square with the region bounded by two hyperbolas). By integrating, one can derive that the area of this region is $O(\log um/um)$, and so $\mathbb{P}(p\text{ and }p'\text{ are close})=O(\log um/um)=O(1/n\log n)$.

Let $\cP$ be a set of points we get by removing a point of each close pair, which ensures that the resulting incidence graphs is $C_4$-free, and adding some points to $\cP$ outside of $[0,1]^2$ to ensure that $\cP$ has $n$ elements. A simple expectation argument now implies that the incidence graph $G$ of $(\cP,\cE)$ has at least $\Omega(k n)=\Omega(n\log n/\log\log n)$ edges, which completes the first step of our construction.

Let us now present the lifting trick which turns this point-ellipse incidence graph into a point-halfspace incidence graph in $\bR^4$. Each axis-parallel ellipse $E\in \cE$ is given by $\{(x,y)\in \bR^2:a(x-x_0)^2+b(y-y_0)^2\leq 1\}$ for some $a,b,x_0,y_0\in \mathbb{R}$. Then, we can map $E$ to the halfspace in $\bR^4$ given by \[\pi(E)=\{(x,y,z,w)\in \bR^4:ax-2ax_0y+bz-2by_0w\leq 1-ax_0^2-by_0^2\}.\] Furthermore, we map every point $p=(t,u)\in \mathbb{R}^2$ to a point $\tau(p)=(t^2,t,u^2,u)\in \bR^4$.
The maps $\pi$ and $\tau$ have the property that $p\in E$ if and only if $\tau(p)\in \pi(E)$. Therefore, the incidence graph of the $n$ points $\tau(\cP)$ and $n$ halfspaces $\pi(\cE)$ is $C_4$-free with $\Omega(n\log n/\log\log n)$ incidences, completing the proof.
\end{proof}

\subsection{Intersection graphs of curves}

In this section, we present the proofs of Theorems \ref{thm:string}, \ref{thm:curve_k_int}, \ref{thm:curve_disjoint}, \ref{thm:curve_convex}. The proofs of Theorems \ref{thm:curve_k_int} and \ref{thm:curve_disjoint} further assume Theorem \ref{thm:stEH_kint}, which is proved in Section \ref{sect:k_int}. We summarize our statements as follows.

\begin{theorem}
Let $\mathcal{C}$ be a family of $n$ curves such that the intersection graph $G$ of $\mathcal{C}$ is $K_{s,s}$-free. 
\begin{enumerate}
    \item[(i)] $G$ has $O(s(\log s)n)$ edges.
    \item[(ii)] If $\mathcal{C}$ is $k$-intersecting, then $G$ has $O_k(sn)$ edges.
    \item[(iii)] If every element of $\mathcal{C}$ is convex, then $G$ has $O(sn)$ edges.
    \item[(iv)] If $\mathcal{C}=\mathcal{A}\cup\mathcal{B}$, where $\mathcal{A}$ and $\mathcal{B}$ are collections of disjoint curves, then $G$ has $O(sn)$ edges.
\end{enumerate}
\end{theorem}

\begin{proof}
String graphs do not contain any induced subdivision of the 1-subdivision of $K_5$. Therefore, by \cite{KO04}, the family of string graphs is weakly degree-bounded. 
\begin{itemize}
    \item[(i)] In particular, there exists an absolute constant $k$ such that any $C_4$-free string graph has average degree less than $k$. Let $c=c(k)>0$ be the constant promised by Theorem \ref{thm:main}, and let $d$ be the average degree of $G$. Then $G$ contains an induced subgraph $H$ on $d$ vertices with at least $cd^2$ edges.   By Corollary 1.2 in \cite{FoxPach12}, there exists $\delta=\delta(c)>0$ such that $H$ contains a biclique of size $\delta d/\log d$. As $G$ is $K_{s,s}$-free, we get $\delta d/\log d<s$, which gives the desired bound $d=O(s\log s)$.

    \item[(ii)] The family $\mathcal{F}_k$ of intersection graphs of $k$-intersecting families of curves has the density-EH property by Theorem \ref{thm:stEH_kint}. Thus Corollary \ref{thm:master} implies the desired result.

    \item[(iii)] The family $\mathcal{F}$ of intersection graphs of convex sets has the density-EH property by Theorem 15 in \cite{FPT10}. Therefore, we are done by Corollary \ref{thm:master}.

    \item[(iv)] Let $\mathcal{F}$ be the family of graphs $G$ that can be realized as the intersection graph of two families of disjoint curves. Then $\mathcal{F}$ is a subfamily of string graphs, so $\mathcal{F}$ is weakly degree-bounded. Moreover, by Theorem \ref{thm:stEH_kint} applied with $k=0$, $\mathcal{F}$ has the density-EH property. Therefore, we are done by Corollary~\ref{thm:master}.\qedhere
\end{itemize}
\end{proof}

\section{Graph-theoretic lemma}\label{sec:graph theory}

In this section, we prove Theorem \ref{thm:main2}, and use it to derive Theorem~\ref{thm:main}. Recall, Theorem~\ref{thm:main2} states that for any integers $k, t\geq 1$ and parameter $c\in (0,1)$, there exists a constant $C=C(k, c)$ such that any $(c, t)$-sparse graph $G$ of average degree $d(G)\geq Ct$ contains a bipartite $C_4$-free induced subgraph of average degree at least $k$.

We recall that a graph $G$ is $(c, t)$-sparse if for any two (not necessarily disjoint) subsets of the vertices $A,B\subset V(G)$ with $|A|, |B|\geq t$, the number $e(A,B)$ of pairs $(a,b)\in A\times B$ joined by an edge is at most $e(A, B)\leq (1-c)|A||B|$. A simple averaging argument shows that in order to check whether a graph is $(c, t)$-sparse, it suffices to verify the above condition for every pair of sets $A,B\subset V(G)$ with $|A|=|B|=t$.

We start by showing that Theorem~\ref{thm:main2} implies Theorem \ref{thm:main}. That is, for every $k\geq 1$ and every graph $G$ of average degree $d$, one can either find an induced $C_4$-free subgraph of average degree $k$ in $G$ or an induced subgraph on $d$ vertices and $\Omega_k(d^2)$ edges.

\begin{proof}[Proof of Theorem \ref{thm:main} assuming Theorem \ref{thm:main2}.] Let $c=1/2$, and let $C=C(k, c)$ be the constant given by Theorem \ref{thm:main2}. We may assume that $C>2$, so that if we set $t=\lfloor d/C\rfloor$, we have $t\leq d/2$. Theorem~\ref{thm:main2} then implies that $G$ either contains a $C_4$-free induced subgraph of average degree at least $k$, or $G$ is not $(c,t)$-sparse. In the first case, we are done, while if $G$ is not $(c,t)$-sparse, then there exist $A,B\subset V(G)$, $|A|=|B|=t$ such that $e(A,B)\geq (1-c)|A||B|=\frac{1}{2}t^2$. Let $U$ be an arbitrary $d$-element subset of $V(G)$ containing $A\cup B$. Then $e(G[U])\geq e(G[A\cup B])\geq \frac{1}{4}t^2\geq \frac{1}{8C^2}d^2$, showing that there exists an induced subgraph of $G$ with $d$ vertices and $\Omega_k(d^2)$ edges.
\end{proof}

The proof of Theorem~\ref{thm:main2} crucially relies on the following technical lemma. Given integers $v\geq h\geq 1$, the \emph{$1$-subdivision of $K_{v}^{(h)}$}, denoted by $H_v^{(h)}$, is the bipartite graph with vertex classes $A$ and $B$, where $|A|=v$, $|B|=\binom{v}{h}$, and for every $S\subset A$ of size $h$ there is a unique vertex $x_S\in B$ whose neighbourhood is exactly $S$. Note that any bipartite graph $H = (A,B; E)$ with $\deg a\le h-1$ for every $a\in A$ is contained in $H_{v}^{(h)}$ as an induced subgraph, where $v=|A|+|B|+h$. To see this, split the part of $H_v^{(h)}$ containing $v$ vertices into three parts $V_A, V_B, V_h$, whose sizes are $|A|, |B|$ and $h$, respectively, and embed the vertices of $B$ into $V_B$ in an arbitrary way. Then, embed the vertices $a\in A$ one by one, by choosing a vertex of $H_{v}^{(h)}$ adjacent to exactly the neighbours of $a$ among in $V_B$, and to an unused vertex in $V_A$ (this is needed to ensure that the vertices $a\in A$ with the exact same neighbourhood in $H$ get embedded to distinct vertices in $H_v^{(h)}$), and potentially to some dummy vertices in $V_h$ (to make its degree to $V_A\cup V_B\cup V_h$ exactly $h$). 

\begin{lemma}\label{lemma:1subdivision}
For any integers $k, v, h\ge 1$ and $c\in (0,1)$, there exists $K=K(k, v, h, c)$ such that the following holds. Every $(c,t)$-sparse graph $G$ of average degree $d(G)\ge Kt$ which does not contain an induced copy of $H_v^{(h)}$ must contain an induced $C_4$-free subgraph of average degree at least $k$.
\end{lemma}

We first prove Lemma~\ref{lemma:1subdivision}, and then show how it can be used to derive Theorem~\ref{thm:main2}. As this proof is the most involved part of this section, we give a brief outline. We proceed in four steps, by finding induced subgraphs $G\supset G_1\supset G_2\supset G_3 \supset G_4$ with increasingly stronger properties. 

In the first step (Lemma~\ref{lemma:step 1}), we pass to an induced subgraph $G_1\subset G$ with a partition $V(G_1)=A_1\cup B_1$ such that $|A_1|\ge |B_1|/2$, $G_1$ is $d$-degenerate for some $d\geq d(G)$, $d_{G_1}(a)\le 4d$ for every $a\in A_1$, and every $a\in A_1$ has at least $d/4$ neighbours in $B_1$.

The main new contribution lies in the second step (Lemma~\ref{lemma:step 2}). By random sampling, we find sets $A_2\subseteq A_1$ and $B_2\subseteq B_1$, forming the graph $G_2=G_1[A_2\cup B_2]$, which does not contain any copy of $K_{2, v+1}$ in which the vertex class with two vertices lies in $A_2$, and the vertex class with $v+1$ vertices lies in $B_2$. Beyond that, we ensure some further useful properties in $G_2$ (e.g., there will be no triangles in $G_2$ containing vertices of both $A_2$ and $B_2$). 

So far, we found a subgraph of $G$ with no copies of $K_{2, v+1}$ as described above. However, the graph $G_2$ may still contain other copies of $K_{2, v+1}$, for example those fully contained in $A_2$. To eliminate all remaining copies of $K_{2, v+1}$, in the third step (Lemma~\ref{lemma:step 3}) we show that there exists a bipartite induced subgraph $G_3\subset G_2$ with high average degree, whose vertex partition $V(G_3)=A_3\cup B_3$ satisfies $A_3\subseteq A_2, B_3\subseteq B_2$. The ideas used in this step come from the work of Kwan, Letzter, Sudakov and Tran \cite{KLST20}, who showed that any triangle-free graph of average degree $d$ contains an induced bipartite subgraph of average degree $\Omega(\log d/\log\log d)$ (see also the nice write-up of Glock \cite{G20}). Note that the result of this step is an induced $K_{v+1, v+1}$-free subgraph $G_3\subset G_2$ of high average degree.

In the fourth step, we use a theorem of Du, Hunter, Gir\~ao, McCarty, and Scott \cite{DGHMS} which states that for every $k$ and $v$, every $K_{v+1, v+1}$-free graph of average degree $k^{C(v+1)^3}$ contains a $C_4$-free subgraph of average degree $k$, where $C$ is some absolute constant (Theorem~\ref{thm:degree-bounded}). A less quantitative version is proved earlier by McCarty \cite{MC21}, and a better dependence on $v$ for fixed $k$ was obtained by Hunter and Gir\~ao \cite{GH}. This theorem can be applied directly to the graph $G_3$, thus guaranteeing an induced subgraph $G_4\subseteq G$, completing the proof.

\begin{lemma}\label{lemma:step 1}
For every graph $G$, there is a real number $d\geq d(G)$ and an induced subgraph $G_1\subseteq G$ with the vertex set $V(G_1)=A_1\cup B_1$ such that $|A_1|\ge |B_1|/2$, $G_1$ is $d$-degenerate, and every $a\in A_1$ has $\deg(a)\leq 4d$ and at least $d/4$ neighbours in $B_1$.
\end{lemma}
\begin{proof}
Let $G_0$ be an induced subgraph of $G$ of maximum average degree $d$. Then $G_0$ is $d$-degenerate, and it has minimum degree $d/2$. Otherwise, removing a vertex of degree less than $d/2$ gives a graph of larger average degree, contradicting the maximality of $G_0$. Let $V(G_0)=A_0\cup B_0$ be a partition that maximizes the number of edges in a cut, and assume that $|A_0|\ge |B_0|$.
    
Let $A'\subset A_0$ be the set of vertices $a$ of degree more than $4d$. Then $2d|A'|\leq e(G_0)=(|A_0|+|B_0|)d/2\leq |A_0|d$, which gives $|A'|\leq |A_0|/2$. Thus, if we set $A_1=A_0\setminus A'$ and $B_1=B_0$, then we have $|A_1|\geq |A_0|/2\geq |B_1|/2$ and $d_{G_0}(a)\leq 4d$ for all $a\in A_1$. 
On the other hand, using that $A_0, B_0$ maximizes the number of edges in a cut, we have $|B_0\cap N(a)|\geq d_{G_0}(a)/2\geq d/4$ for every $a\in A_0$. Otherwise, the cut $(A_0\setminus\{a\},B_0\cup\{a\})$ has more edges than the cut $(A_0,B_0)$. Therefore, the induced subgraph $G_1=G[A_1\cup B_1]$ with partition $A_1\cup B_1$ satisfies all conclusions of the lemma and the proof is complete. 
\end{proof}

In the second step, we use the following lemma, proved by Ding, Gao, Liu, Luan, and Sun \cite{DGLLS}.

\begin{lemma}[Lemma 3.1 in \cite{DGLLS}.]\label{sparsity focusing}
Let $H$ be a bipartite graph. For any $c, \eps>0$ there exists $\beta = \beta(c, H, \eps)$ such that every $(c, t)$-sparse graph $G$ which does not contain an induced copy of $H$ is also $(1-\eps, \beta t)$-sparse. 
\end{lemma}

\begin{lemma}\label{lemma:step 2}
Given integers $\tau, v, h\geq 0$ and a real number $c>0$, there exists a constant $K=K(\tau, v, h, c)$ such that for every $d\geq Kt$ the following statement is true. 
Let $G_1$ be a graph satisfying the conclusions of Lemma~\ref{lemma:step 1}, i.e. a $d$-degenerate $(c, t)$-sparse graph on the vertex set $A_1\cup B_1$ such that $|A_1|\ge |B_1|/2$, and $|N(a)|\le 4d$, $|N(a)\cap B_1|\geq d/4$ for every $a\in A_1$. 

If $G_1$ contains no induced copy of $H_v^{(h)}$, there exist subsets $A_2\subseteq A_1$ and $B_2\subseteq B_1$ such that the graph $G_2=G_1[A_2\cup B_2]$ is $200\tau$-degenerate and $e(A_2, B_2)\geq \tau (|A_2|+|B_2|)$. Additionally, there is no triangle with vertices in both $A_2$ and $B_2$, and there are no copies of $K_{2, v+1}$ with the vertex class of two vertices in $A_2$, and the other vertex class in $B_2$.
\end{lemma}
\begin{proof}
Let $\eps>0$ be a sufficiently small parameter with respect to $\tau, v, h$ ($\eps\ll \tau^{-3v}, v^{-3h}$ suffices). Applying Lemma~\ref{sparsity focusing}, there exists $\beta=\beta(\eps,c,H)$ such that $G_1$ is $(1-\eps, \beta t)$-sparse, and we set $t_0=\beta t$. Finally, we choose $K$ such that $K \geq 4\binom{v}{h}\beta/\eps$. In particular, we hav $\eps d \geq 4\binom{v}{h}\beta t$.

We obtain the sets $A_2, B_2$ by randomly sampling subsets of $A_1, B_1$ and appropriately cleaning them to remove triangles and copies of $K_{2, v+1}$. However, in order to be able to analyze this random process, we need two short preliminary observations. The first observation is a simple consequence of the local sparsity condition, and it is in fact the only place in the proof where this condition is used. For a vertex $a\in A_1$, let $S_a:= \{v\in V(G_1): |N(a)\cap N(v)|> 4\eps d\}$ be the set of vertices which share many common neighbours with $a$. Then, we must have $|S_a|< t_0$, as otherwise, the number of edges between $N(a)$ and $S_a$ is at least $4\eps d|S_a|>\eps |N(a)||S_a|$, since $|N(a)|\leq 4d$. This contradicts the fact that $G_1$ is $(1-\varepsilon, t_0)$-sparse (since $|N(a)|\ge d/4\ge t_0$ for all $a\in A_1$). Thus, as $d/4\geq Kt/4\geq \beta t/\eps$,  we have $|S_a|< t_0=\beta t\leq \eps d$.

The second observation is that very few independet sets $S\subset N(a)$ of size $v$ have at least $2\eps d$ common neighbours in $A_1$.

    \begin{claim}\label{claim:number_of_cliques}
        For any vertex $a\in A_1$, there are at most $\eps^{1/2} \binom{|N(a)|}{v}$ subsets $S\subset N(a)\cap B_1$ of size $v$ which are independent and have at least $2\eps d$ common neighbours in $A_1$.
    \end{claim}
    \begin{proof}
        Let $\cH_a$ be an auxiliary $h$-uniform hypergraph on the vertex set $N(a)\cap B_1$, in which an $h$-element set $e\subset N(a)\cap B_1$ is an edge if $e$ is an independent set and $|\bigcap_{b\in e}N(b)\cap A_1|\geq 2\eps d$. Note that every independent set $S\subset N(a)$ with at least $2\eps d$ common neighbours is a $v$-clique in $\cH_a$, and therefore it suffices to bound the number of $v$-cliques in $\cH_a$. Thus, let us select $S$ uniformly at random among $v$-element subsets of $N(a)$, and bound the probability that it induces a clique in $\cH_a$.

        For an edge $e\in E(\cH_a)$, let $Z_e=(\bigcap_{b\in e}N(b)\cap A_1)\setminus S_a$, so $Z_e$ is the set of  common neighbours of $e$ whose codegree with $a$ is at most $4\eps d$. As $|\bigcap_{b\in e}N(b)\cap A_1|\geq 2\eps d$ and $|S_a|\leq \eps d$, we have $|Z_e|\ge \eps d$ for all $e$. Also, let $Z_e^*$ be the  set of those vertices $x\in Z_e$ which have $e\subsetneq N(x)\cap S$, and let $\cE_e$ be the event that $|Z^*_e|\ge |Z_e|/2$. Observe that $\cE_e$ does not happen if $e\not\subset S$, i.e. $\Pb[\cE_e|e\not\subset S]=0$. 
        
        On the other hand, for $e\in E(\cH_a)$, we claim $\mathbb{P}[\cE_e|e\subset S]\le 64\eps v$. Indeed, conditioning on the event $e\subset S$, we can write $S= e\cup S'$ for a random $v-h$ element set $S'$ chosen from the uniform distribution on $N(a)\cap B_1\setminus e$. For each $x\in Z_e$, we have that $$\mathbb{P}\big[x\in Z_e^*\big|e\subset S\big] = \mathbb{P}\big[N(x)\cap S'\neq \varnothing\big] \le (v-h)\frac{|N(x)\cap N(a)\cap B_1|}{|N(a)\cap B_1|-v}\le v\frac{4\eps d}{d/8}\le 32v\eps,$$ from which $\mathbb{E}\big[|Z_e^*|\big|e\subset S\big]\leq 32v\eps |Z_e|$. Therefore,  we get 
        $$\mathbb{P}\big[\cE_e\big|e\subset S\big]=\mathbb{P}\big[|Z_e^*|\geq |Z_e|/2\big|e\subset S\big]\leq \mathbb{P}\Bigg[Z_e^*\geq \frac{\mathbb{E}[Z_e^*|e\subset S]}{64v\eps}\Bigg|e\subset S\Bigg]\le 64 v\eps,$$
        where the last inequality is due to Markov's inequality. Hence, by the union bound,  $$\mathbb{P}\Big[\bigvee_{e\in E(\cH_a)}\cE_e\Big]\le \sum_{e\in E(\cH_a)} \Pb\big[\cE_e\big|e\subset S\big]\Pb[e\subset S]\leq  64 v \eps \sum_{e\in E(\cH_a)} \Pb[e\subset S]\leq  64 v \eps \cdot \binom{|S|}{h} \le 64 \eps v^{h+1},$$
        where we  used the linearity of expectation to say that $\sum_{e\in E(\cH_a)} \Pb[e\subset S]$ is the expected number of $h$-element subsets of $S$, that is $\binom{|S|}{h}$.
        
        We also claim that if $S$ induces a clique in $\cH_a$, then at least one of the events $\cE_e$ holds for some $e\in \cH_a$. This suffices to complete the proof, since it implies that the probability that $S$ is a clique in $\cH_a$ is at most $64 \eps v^{h+1}\leq \eps^{1/2}$. 
        
        Assume that for some $S$ no event $\cE_e$ holds, then we will find an induced copy of $H_v^{(h)}$ in $G_1$.   As $S$ induces a clique in $\cH_a$, $S$ is an independent set, and we can embed the side of $H_v^{(h)}$ of size $v$ into $S$ in an arbitrary way. Then, for every $h$-element set $e\subset S$, we choose a vertex $x_e\in Z_e\backslash Z_e^*$, that is, a vertex with the property that $N(x_e)\cap S = e$.
        Since $\cE_e$ does not hold, we have $|Z_e\backslash Z_e^*|\geq |Z_e|/2\geq \eps d/2$ for all $e\subset S$, and thus we can fix some set $T_e\subseteq Z_e\backslash Z_e^*$ of size exactly $\eps d/2$. Next, we find a vertex $x_e\in T_e$ for every edge $e$ such that $\{x_e\}_{e\in \binom{S}{h}}$ is an independent set. This gives an induced copy of $H_v^{(h)}$, finishing the proof. 

        To find such an independent set, we begin by eliminating from $T_e$ all vertices $x$ such that $|N(x)\cap T_{e'}|> \eps |T_{e'}|$ for some $e'\subset S$. Since $G_1$ is a $(1-\eps, t_0)$-sparse graph, for each $e'$ there are at most $t_0$ such vertices $x\in T_e$. Let $R_e$ denote the set of remaining vertices in $T_e$ after these deletions, for which we have $|R_e|\geq |T_e|-\binom{v}{h}t_0\geq \eps d/2 - \binom{v}{h}\beta t\geq \eps d/2-\eps d/4=|T_e|/2$. Now, we greedily pick vertices $x_e\in R_e$ one by one, noting that at each step the previously chosen vertices can be adjacent to at most $\binom{v}{h} \cdot \eps |T_{e}|<|T_e|/2\leq |R_e|$ vertices in $R_e$, thus showing that there is always at least one choice for $x_e$ not adjacent to any of the previously chosen vertices. But then $S$ together with $\{v_e : e\in \binom{S}{h}\}$ induces a copy of $H_v^{(h)}$ in $G_1$. This contradicts the assumption that no such copy exists in $G_1$.
    \end{proof}

With this claim, we can now define $A_2, B_2$. First, let $A'\subset A_1, B'\subset B_1$ be random subsets in which each element is included independently with probability $p=30\tau/d$. Since the graph $G_1$ is $d$-degenerate, it has an orientation $X$ in which every vertex has outdegree at most $d$. To obtain $B_2$, we delete from $B'$ all vertices $b$ with $|N^+_X(b)\cap (A'\cup B')|\geq 200\tau$. Similarly, to obtain $A_2$, we delete from $A'$ all vertices $a$ which satisfy one of the following three conditions:
\begin{itemize}
    \item[\textit{(i)}] $|N^+_X(a)\cap (A'\cup B')|\geq 200\tau$,
    \item[\textit{(ii)}] $a$ belongs to some triangle with vertices in $A'\cup B'$,
    \item[\textit{(iii)}] for some $a'\in A'\backslash\{a\}$, the set $N(a)\cap N(a')\cap B'$ contains an independent set of size $v+1$.
\end{itemize}

Observe that the sets $A_2$ and $B_2$ defined as above satisfy all necessary conditions, except perhaps $e(A_2, B_2)\geq \tau (|A_2|+|B_2|)$. Namely, since $X$ is an orientation of $G_2=G_1[A_2\cup B_2]$ in which every vertex has outdegree at most $200\tau$, the graph $G_2$ is $200\tau$-degenerate. Also, since all vertices of $A'$ contained in some triangle were deleted, there do not exist any triangles with vertices in both $A_2$ and $B_2$. Finally, observe that due to the lack of triangles, any copy of $K_{2, v+1}$ with the vertex class with two vertices in $A_2$ and the other vertex class in $B_2$ must be induced. But $N(a)\cap N(a')\cap B_2$ does not contain an independent set of size $v+1$ for any $a, a'\in A_2$, so $G_2$ does not contain any copies of $K_{2, v+1}$. 

To give a lower bound on the number of edges between $A_2$ and $B_2$, we  use the following claim.

\begin{claim}\label{claim:probability of survival}
If $a\in A_1$ and $b\in B_1$ are adjacent vertices such that $|N(a)\cap N(b)|\leq 4\eps d$ (so $b\notin S_a$), then $\Pb[a\in A_2, b\in B_2]\geq p^2/2$. 
\end{claim} 

Note that this claim is sufficient to complete the proof, since for each $a\in A_1$, there are at least $|N(a)\cap B_1|-|S_a|\geq d/4-\eps d\geq d/5$ vertices $b$ for which the claim applies. This allows us to lower bound the expected number of edges between $A_2$ and $B_2$ as follows:
\[\E\big[e(A_2, B_2)\big]\geq |A_1|\cdot \frac{d}{5}\cdot \frac{p^2}{2}\geq \Big(\frac{|A_1|}{3}+\frac{|B_1|}{3}\Big)\cdot \frac{d}{5} \cdot \frac{30\tau}{2d}p\geq \tau \big(|A_1|+|B_1|\big)p\geq \tau \big(\E\big[|A'|\big]+\E\big[|B'|\big]\big),\]
where we used that $|A_1|= |A_1|/3+2|A_1|/3\geq |A_1|/3+|B_1|/3$ since $|A_1|\geq |B_1|/2$.

As $|A_2|\leq |A'|$ and $|B_2|\leq |B'|$, we also have that  $\E\big[e(A_2, B_2)\big]\geq \tau (\E[|A_2|]+ \E[|B_2|])$. Therefore, there exists an outcome in which $e(A_2, B_2)\geq \tau (|A_2|+|B_2|)$, which is precisely what we wanted. Thus, to complete the proof, it remains only to prove Claim~\ref{claim:probability of survival}.

\begin{proof}[Proof of Claim~\ref{claim:probability of survival}.]
For $a, b$ to be in the sets $A_2, B_2$, respectively, it is necessary to have $a\in A'$ and $b\in B'$. These events are independent and each happens with probability $p$. Thus, the event $E$ that both $a\in A'$ and $b\in B'$ has probability $p^2$. In the rest of the proof, we show that, conditioning on the event $E$, we have $\Pb[a\notin A_2|E]\leq 1/3$ and $\Pb[b\notin B_2|E]\leq 1/6$. This is sufficient to complete the proof, since a union bound then implies that $\Pb\big[a\in A_2, b\in B_2\big|E\big]\geq 1/2$, and so we have \[\Pb[a\in A_2, b\in B_2]=\Pb\big[a\in A_2, b\in B_2\big|E\big]\cdot \Pb[E]\geq p^2/2.\]

We begin by showing $\Pb[b\notin B_2|E]\leq 1/6$. Given that $a\in A', b\in B'$, the probability that $b\notin B_2$ is the same as the probability of the event $|N^+_X(b)\cap (A'\cup B'\cup \{a\})|\geq 200\tau$. In other words, $b$ is deleted from $B'$ only if $|(N^+_X(b)\backslash \{a\})\cap (A'\cup B')|\geq 199\tau$. Thus, by Markov's inequality we have 
\begin{align*}
\Pb\big[b\notin B_2\big|E\big]&\leq \Pb\big[\big|(N^+_X(b)\backslash \{a\})\cap (A'\cup B')\big|\geq 199\tau\big]\\
&\leq \frac{1}{199\tau}\E\big[|N^+_X(b)\cap (A'\cup B')|\big]= \frac{p|N^+_X(b)|}{199\tau}=\frac{30\frac{\tau}{d} d}{199\tau}\leq \frac{1}{6}.
\end{align*}

We now show that $\Pb[a\notin A_2|E]\leq 1/3$. By the above argument, the probability that $a$ was deleted from $A'$ due to the condition \textit{(i)} is at most $1/6$.  The probability that $a$ is deleted from $A'$ due to the condition \textit{(ii)} is at most $1/12$. Indeed, given that $a\in A', b\in B'$, the probability $a$ is in a triangle of $A'\cup B'$ which contains $b$ is at most $p|N(a)\cap N(b)|\leq 4\eps d\cdot 30\tau/d\leq 120\eps \tau$. Also, the probability that $a$ is in a triangle which does not include $b$ is at most $e(G_1[N(a)])\cdot p^2\leq \eps |N(a)|^2 \cdot 30^2\tau^2/d^2\leq 1000 \eps \tau^2 \cdot (4d)^2/d^2\leq 2\cdot 10^4 \eps \tau^2$. Note that $e(G_1[N(a)]) \leq \eps |N(a)|^2$ follows from $(1-\eps, \beta t)$-sparsity. In total, the probability that $a'$ is deleted due to \textit{(ii)} is at most $120\eps \tau+2\cdot 10^4 \eps \tau^2\leq 1/12$, using that $\eps$ is sufficiently small compared to $\tau$. 

As the last step, we show that the probability $a$ is deleted due to condition \textit{(iii)} is at most $1/12$. In other words, we want to bound the probability that there is a set $S\subseteq N(a)$ of size $v+1$ such that $S\subseteq B'$ and $\bigcap_{b\in S} N(b)$ contains a vertex $a'\in A'\backslash\{a\}$. To do this, we use a union bound over all independent $(v+1)$-sets $S\subset N(a)$. We call an independent set $S$ \textit{rich} if has at least $2\eps d$ common neighbours in $A_1$ and \textit{poor} otherwise. 

By Claim~\ref{claim:number_of_cliques}, there are at most $\eps^{1/2}\binom{|N(a)|}{v}$ rich sets $S$ containing $b$, since $S\backslash\{b\}$ is an independent set of size $v$ with $2\eps d$ common neighbours. Similarly, there are at most $\eps^{1/2}\binom{|N(a)|}{v}\cdot |N(a)|$ rich sets in $N(a)$ overall, since each such set is obtained by adding an element to a $v$-set with at least $2\eps d$ common neighbours. Thus, the probability that any rich set in included in $B'$ is at most
\begin{align*}
\Pb[B'\text{ contains a rich set }&S|E]\leq \!\!\!\!\sum_{\substack{S\subset N(a)\\ S \text{ rich, }b\notin S} } \!\!\!\!\Pb[S\subseteq B'|E]+\!\!\!\!\sum_{\substack{S\subset N(a)\\ S \text{ rich, }b\in S} }\!\!\!\! \Pb[S\subseteq B'|E]\\
&\leq \eps^{1/2} |N(v)|^{v+1} p^{v+1} + \eps^{1/2} |N(v)|^{v} p^{v}\leq \eps^{1/2} (4d)^{v}\Big(\frac{30\tau}{d}\Big)^v \Big(4d\frac{30\tau}{d}+1\Big)\leq \frac{1}{24},
\end{align*}
since $\eps$ is chosen to be small enough as a function of $\tau$ and $v$.

Furthermore, note that if $S$ is poor, the probability that some $a'\in A_1\backslash\{a\}\cap \bigcap_{b\in S} N(b)$ is selected into $A'$ is at most $2\eps d\cdot p\leq 2\eps\cdot 30\tau\leq \eps^{1/2}$. Thus, using similar estimates as above, 
\begin{multline*}
\Pb\Big[\text{there is a poor $S\subset B$}\text{ and $a'\in A'\backslash\{a\}\cap \bigcap_{b\in S}N(b)$}\Big|E\Big]\\\leq \!\!\!\!\sum_{\substack{S\subset N(a)\\ S \text{ poor, }b\notin S} } \!\!\!\!\Pb[S\subset B'|E] \eps^{1/2}+\!\!\!\!\sum_{\substack{S\subset N(a)\\ S \text{ poor, }b\in S} }\!\!\!\! \Pb[S\subset B'|E] \eps^{1/2}
\leq \eps^{1/2} |N(v)|^{v+1} p^{v+1} + \eps^{1/2} |N(v)|^{v} p^{v}\leq \frac{1}{24}.
\end{multline*}
Thus, the probability that $a$ is deleted from $A'$ due to the condition \textit{(iii)} is at most $2/24=1/12$, completing the proof of Claim~\ref{claim:probability of survival}, and thus completing the second step as a whole.
\end{proof}
\end{proof}

Finally, to perform the third step, we need the following lemma. 

\begin{lemma}\label{lemma:step 3}
Let $\tau$ be a positive integer. Suppose $G_2$ is a $200\tau$-degenerate graph with the vertex set $V(G_2)=A_2\cup B_2$ such that $e(A_2, B_2)\geq \tau (|A_2|+|B_2|)$ and there is no triangle with vertices in both $A_2$ and $B_2$. Then, there are independent sets $A_3\subseteq A_2$, $B_3\subseteq B_2$ such that $G_2[A_3\cup B_3]$ has average degree at least $\Omega(\log \tau/\log\log \tau)$.
\end{lemma}
\begin{proof}
We may assume that each vertex of the bipartite graph between $A_2$ and $B_2$ has degree at least $\tau$. Otherwise, removing vertices with degree less than $\tau$ results in a non-empty subgraph which still satisfies the conditions of the theorem. Furthermore, assume that $|A_2|\geq |B_2|$.

Let $\ell$ be the largest integer satisfying $(200\ell^2)^\ell\leq \tau$, so $\ell=\Theta(\log\tau/\log\log \tau)$. Let $B'\subseteq B_2$, where each vertex of $B_2$ is included independently with probability $p=\frac{1}{200\tau \ell}$. Since $G_2[B_2]$ is a $200\tau$-degenerate graph, there exists an orientation of this graph, denoted by $X$, in which the outdegree of every vertex $b\in B_2$ is at most $200\tau$. Given the set $B'$, we define $B_3$ to be the set of vertices $b\in B'$ satisfying $|B'\cap N^{+}_X(b)|=0$. By construction, $B_3$ is an independent set. Furthermore, associate to every $a\in A_2$ an arbitrary set $N_a\subseteq N(a)\cap B_2$ of size $\tau$. Then, we define $A'$ to be the set of vertices $a\in A_2$ for which $|N_a\cap B_3|= \ell$.

First, we show that $\mathbb{E}\big[|A'|\big]\geq \binom{\tau}{\ell}\cdot p^\ell/4\cdot |A_2|$. For every $a\in A_2$, there are $\binom{\tau}{\ell}$ sets of size $\ell$ in $N_a$, and they are all independent sets since there are no triangles containing vertices of both $A_2$ and $B_2$. For each such set $I$, we will show in the next paragraph that $I=N_a\cap B_3$ with probability at least $p^\ell/4$. The desired inequality follows by the linearity of expectation. For simplicity, we write $q=\binom{\tau}{\ell}\cdot p^\ell$.

For an arbitrary independent set $I\subseteq N_a$ of size $\ell$, we have $I\subseteq B_3$ precisely when $I\subseteq B'$ and $B'\cap \bigcup_{b\in I} N_X^+(b)=0$, where the second condition ensures that no vertex of $I$ is deleted from $B'$. Since $\Big|\bigcup_{b\in I} N_X^+(b)\Big|\leq 200\tau|I|=200\tau \ell$ and each vertex is included in $B'$ randomly and independently with probability $p$, we have that
\[
\mathbb{P}[I\subset B_3]\geq p^\ell (1-p)^{200\tau \ell}\geq p^\ell \Big(1-\frac{1}{200\tau\ell}\Big)^{200\tau \ell}\geq p^\ell/e\geq \frac{1}{3}p^\ell.
\]
Hence,
\[
\mathbb{P}[I = N_a\cap B_3]\geq \mathbb{P}[I \subseteq B_3] - \mathbb{P}[I \subsetneq N_a\cap B_3] \geq \frac{1}{3}p^\ell - \tau p^{\ell+1} \geq \frac{1}{4}p^\ell.
\]
Here, the inequality $\mathbb{P}[I\subsetneq N_a\cap B_3]\leq \mathbb{P}[I\subsetneq N_a\cap B'] \leq \tau p^{\ell+1}$ follows as $|N_a| = \tau$, so by the union bound, the probability that $B'$ contains $I$ and an additional element of $N_a$ is at most $\tau p^{\ell+1}$.

\medskip

The next step is to show $\E\big[e(G_2[A'])\big]\leq 10^3 q\tau \E\big[|A'|\big]$. If $A''$ is the set of vertices $a\in A_2$ with $|N_a\cap B'| \geq \ell$, then $A'\subseteq A''$ since every vertex of $A'$ has $\ell= |N_a\cap B_3|\leq |N_a\cap B'|$. In order to bound $e(G_2[A''])$, we pick two adjacent vertices $a, a'\in A_2$, and observe that they do not have any common neighbours in $B_2$, since $N(b)\cap A_2$ is an independent set for any $b\in B_2$. In particular, $N_a\cap N_{a'}=\varnothing$ and so the events $a\in A''$ and $a'\in A''$ are independent. Thus, the probability that $a, a'\in A''$ is at most $\mathbb{P}[a, a'\in A'']= \mathbb{P}\big[|B'\cap N_a|\ge \ell\big]\cdot\mathbb{P}\big[|B'\cap N_{a'}|\ge \ell\big] \le \binom{\tau}{\ell}^2 p^{2\ell}=q^2$.

Furthermore, since $G_2$ is $200\tau$-degenerate, $A_2$ spans at most $200\tau|A_2|$ edges, and so
\[\E\big[e(G_2[A'])\big]\leq \E\big[e(G_2[A''])\big] \le  q^2 e(G_2[A_2]) \le 200q^2 \tau |A_2|\leq 800q\tau \E[|A'|],\] 
where we used $q|A_2|\leq 4\E[|A'|]$. We are now essentially done. Observe that
\begin{align*}
\E\left[|A'|-\frac{e(G_2[A'])}{10^4q\tau} - 10^4q\tau |B_3|\right]&\geq \frac{1}{2}\E\big[|A'|\big] - 10^4q\tau \E\big[|B_3|\big]\\
&\geq \frac{q}{8}|A_2| - 10^4q\tau p|B_2|\geq \frac{q}{8}|A_2| - \frac{10^4q}{200\ell}|B_2|>0.
\end{align*}

Hence, there is an outcome of the random process in which $10^4q\tau |A'|\geq e(G_2[A'])$ and $|A'|> 10^4q\tau |B_3|$. Furthermore, since $q\tau=p^\ell\binom{\tau}{\ell}\tau\geq \Big(\frac{1}{200\tau \ell}\Big)^\ell \Big(\frac{\tau}{\ell}\Big)^\ell \tau=\frac{\tau}{(200\ell^2)^\ell}\geq 1$, we can find an independent set $A_3\subseteq A'$ of size $|A_3|\geq |A'|/10^4q\tau>|B_3|$. Coupled with the fact that every vertex of $A_3$ has exactly $\ell$ neighbours in $B_3$, we have that the average degree of $G_2[A_3\cup B_3]$ is at least $\frac{2e(A_3, B_3)}{|A_3|+|B_3|}\geq \frac{2\ell|A_3|}{2|A_3|}=\ell$. This completes the proof.
\end{proof}

In the fourth step,  we  use the following result of Du, Hunter, Gir\~ao, McCarty, and Scott \cite{DGHMS} to pass from a $K_{v+1, v+1}$-free graph to a $C_4$-free subgraph of high average degree.

\begin{theorem} \label{thm:degree-bounded}
    There exists an absolute constant $C_0$ such that for all $s,k\ge 2$, the following holds. If $G_3$ is a $K_{s,s}$-free graph such that $d(G_3)\ge k^{C_0s^3}$, then $G_3$ contains a $C_4$-free induced subgraph of average degree at least $k$.
\end{theorem}

Now we are ready to put all the ingredients together and prove Lemma~\ref{lemma:1subdivision}.

\begin{proof}[Proof of Lemma~\ref{lemma:1subdivision}.]
Let $k_4=k$ and choose parameters $k_3, k_2, k_1$, each being sufficiently large as a function of the previous one (and as a function of $v, h$). Our goal is to apply the four steps we described in the introduction, obtaining induced subgraphs $G\supseteq G_1\supseteq G_2\supseteq G_3\supseteq G_4$ with the described properties and such that $G_i$ has average degree at least $k_i$.

Setting $K=k_1$ and starting from a $(c, t)$-sparse graph of average degree $d(G)\geq k_1t$, we apply Lemma~\ref{lemma:step 1} to find an integer $d\geq d(G)\geq k_1t$ and an induced subgraph $G_1\subseteq G$ on the vertex set $V(G_1)=A_1\cup B_1$ such that $|A_1|\geq |B_1|/2$, $G_1$ is $d$-degenerate, $|N(a)|\leq 4d$ and $|N(a)\cap B_1|\geq d/4$ for all $a\in A_1$. Furthermore, since $G$ does not contain an induced copy of $H_v^{(h)}$, the same holds for $G_1$.

As $G_1$ is also $(c, t)$-sparse, we can apply Lemma~\ref{lemma:step 2} with $\tau=k_2$ to find an induced subgraph $G_2\subset G_1$ with the vertex set $V(G_2)=A_2\cup B_2$ such that $G_2$ is $200k_2$-degenerate, $e(A_2, B_2)\geq k_2(|A_2|+|B_2|)$, there are no triangles containing vertices of both $A_2$ and $B_2$, and there are no copies of $K_{2, v+1}$ with the vertex class of size two in $A_2$, and the vertex class of size $v+1$ in $B_2$. 

Then, we apply Lemma~\ref{lemma:step 3} to the graph $G_2$. If $k_2$ is sufficiently large with respect  to $k_3$, then there exist independent sets $A_3\subset A_2$ and $B_3\subset B_2$ such that $G_3=G_2[A_3\cup B_3]$ is a bipartite graph of average degree at least $k_3$. Note that $G_3$ does not contain any copies of $K_{v+1, v+1}$.

Finally, Theorem~\ref{thm:degree-bounded} shows that if $k_3\geq k_4^{C_0(v+1)^3}$, then there is an induced subgraph $G_4\subseteq G_3$ such that $G_4$ is $C_4$-free and has average degree at least $k_4=k$. Note that $G_4$ is also bipartite, so this completes the proof.
\end{proof}

We finish this section by showing how Theorem~\ref{thm:main2} follows directly from our main technical result, Lemma~\ref{lemma:1subdivision}.

\begin{proof}[Proof of Theorem \ref{thm:main2}]
    For a prime power $q$, let $\Gamma_q$ be the incidence graph of points and lines in the finite projective plane over the $q$-element field. Then $\Gamma_q$  is a $(q+1)$-regular bipartite graph with vertex classes of size $q^2+q+1$. 
    
    Set $h=2k$, $v = 10k^2$, and let $K$ be the constant given by Lemma~\ref{lemma:1subdivision}. If $q$ is a prime power between $k$ and $2k-1$, then $H_{v}^{(h)}$ contains $\Gamma_q$ as an induced subgraph. But $\Gamma_q$ is a $C_4$-free graph of average degree at least $k$. Hence, if $G$ contains $H_{v}^{(h)}$ as induced subgraph, it contains $\Gamma_q$ as well, and we are done. Otherwise,  we can apply Lemma \ref{lemma:1subdivision} to find an induced bipartite $C_4$-free subgraph of $G$ with average degree at least $k$, thus completing the proof again.
\end{proof}

\section{Families of curves with bounded complexity}\label{sec:curves}

\subsection{Dense intersection graphs of $k$-intersecting families}\label{sect:k_int}

In this section, we prove Theorem~\ref{thm:stEH_kint}, which states that if $\cA, \cB$ are two $k$-intersecting families of curves of size $n$, such that at least $cn^2$ pairs $(a, b)\in \cA\times \cB$ intersect, then there are subfamilies $\cA'\subseteq \cA, \cB'\subseteq \cB$ such that all strings $a\in \cA', b\in \cB'$ intersect. Also, recall that we say $\cA$ is a $k$-intersecting family of curves if any two (distinct) curves $a, a'\in \cA$ have at most $k$ points of intersection.

We make the following standard non-degeneracy assumptions on our curves $\cA$ and $\cB$: for any $a\in A, b\in B$, the set $a\cap b$ is finite, for any $a_1, a_2\in A, b\in B$, the set $a_1\cap a_2\cap b$ is empty and all curves of $\cA, \cB$ are non-self-intersecting and smooth. Further, we assume that there are no tangencies between the curves, i.e. that when curves $a$ and $b$ cross, the curve $a$ passes from one side of $b$ to the other. These non-degeneracy assumptions can be achieved using standard arguments by perturbing the curves from the family slightly (and potentially slightly increasing $k$). We also give a direction to each curve, so that it has a starting point and an ending point, and we order the points of the curve accordingly. 

Before jumping into the proof, we establish some notation and give a few preliminary observations. Let $G$ be the bipartite intersection graph of $\cA$ and $\cB$, which has vertex classes of size $n$ and at least $cn^2$ edges. Remove all vertices of $G$ with less than $cn/4$ neighbours, and let $G'$ be the resulting graph. Fewer than $c n^2/2$ edges are removed by this, and so $e(G')\geq cn^2/2$. Let $\cA'\subset \cA,\cB'\subset \cB$ be the vertex classes of $G'$, and let $c'=c/4$. Then  $G'$ has minimum degree at least $c'n$, which also implies $|\cA'|,|\cB'|\geq c'n$. Let $n'=c'n$, and let $\cA'',\cB''$ be random $n'$ sized subsets of $\cA',\cB'$, respectively. Then by standard concentration arguments, with high probability, the bipartite intersection graph $G''$ between $\cA''$ and $\cB''$ has minimum degree at least $\frac{1}{2}(n'/n) (c'n)=c'n'/2$.  In what follows, we work with only $(\cA'',\cB'')$, so with slight abuse of notation, we write $\cA$ instead of $\cA''$, $\cB$ instead of $\cB''$, $n$ instead of $n'$, and $c$ instead of $c'/2$. Thus, $|\cA|=|\cB|=n$ and the bipartite intersection graph $G$ between $\cA,\cB$ has minimum degree $cn$.

For each curve $b\in \cB$, we define the set of points 
$$P_b=\bigcup_{a\in \cA} b\cap a.$$
We \textit{color} the point $x\in P_b$ using the color $a$, where $a\in \cA$ is the curve containing $x$. Since each $b\in \cB$ intersects at least $c n$ different curves from $\cA$, we may partition $b$ into $k+1$ closed subcurves $I_1(b), \dots, I_{k+1}(b)$, overlapping only in endpoints, such that the points in $P_b\cap I_j(b)$ receive at least $\frac{c}{2k} n$ different colors for each $1\leq j\leq k+1$. 

We now state one of the key definitions of the proof. For integers $1\leq j\leq k+1 $ and $\ell\geq 1$, we say that an $\ell$-tuple of curves $(b_1, \dots, b_\ell)\in \cB^\ell$ is \textit{avoiding the index $j$} if for all pairs $1\leq r< s\leq \ell$, and all intersection points $x\in b_{r}\cap b_{s}$, we have that $x\notin I_j(b_r)$. 

The following two lemmas are crucial. First, we show that if $\ell$ is a constant, then a positive fraction of all $\ell$-tuples of curves in $\cB$ are avoiding the index $j$, for some $1\leq j\leq k+1$. After this, we  show a general lemma about colorings of finite intervals, which then is applied to the coloring of $P_b$, for various curves $b$.

\begin{lemma}\label{lemma:avoiding ell-tuples}
For every integer $\ell\geq 1$, there is an index $j$ such that the number of $\ell$-tuples $(b_1, \dots, b_\ell)\in \cB^\ell$ avoiding the index $j$ is at least $\Omega_{k, c, \ell}(n^\ell)$.
\end{lemma}
\begin{proof}
Order $\cB$ in an arbitary manner and consider only the $\ell$-tuples whose elements respect this ordering. For any pair of curves $(b, b')\in \cB^2$, such that $b$ comes before $b'$, there are at most $k$ intersection points in $b\cap b'$, and therefore there exists $j\in [k+1]$ such that $I_j(b)$ does not contain any of the points from $b\cap b'$. Consider the complete graph $\Gamma$ on the vertex set $\cB$, and color the edge $bb'$ with color $j$. Observe that the $\ell$-tuple $(b_1, \dots, b_\ell)$ avoids the index $j$ if it induces a monochromatic clique of color $j$ in $\Gamma$. 

 Let $r=r_{k+1}(\ell)$ be the $(k+1)$-color Ramsey number of the clique of size $\ell$, i.e. the smallest number such that any edge-coloring of $K_{r}$ using $k+1$ colors contains a monochromatic $\ell$-clique. Then any subset of $r$ curves of $\cB$ contains a monochromatic copy of $K_\ell$. Moreover, any copy of $K_\ell$ is contained in at most $\binom{n-\ell}{r-\ell}$ sets of size $r$. Hence, the number of monochromatic cliques of size $\ell$ in $\Gamma$ is at least $\binom{n}{r}/\binom{n-\ell}{r-\ell}=\Omega_{k, c, \ell}(n^\ell)$. Picking the color $j\in [k+1]$ which is shared by the largest number of these cliques gives a set of $\Omega_{k, c, \ell}(n^\ell)$ $\ell$-tuples $(b_1, \dots, b_\ell)\in \cB^\ell$ avoiding the index $j$. This completes the proof.
\end{proof}

\begin{lemma}\label{lemma:coloring lemma}
Assume that the elements of $[N]$ are colored with $m$ colors, all colors being used at least once. Say that a pair of colors $(c, d)$ is \textit{good} if there exist $1\leq i< j\leq N$ such that the interval $[i, j]$ receives at least $m/3$ distinct colors, the only element of $[i, j]$ receiving color $c$ is $i$ and the only element of $[i, j]$ receiving color $d$ is $j$. 
Then, there are at least $m^2/9$ good pairs of colors $(c, d)$.
\end{lemma}
\begin{proof}
Let $1\leq t\leq N$ be the smallest integer such that the interval $[1, t]$ receives $2m/3$ different colors. Let $C$ be the set of colors received by this interval, and let $D=[m]\backslash C$ be the set of colors not appearing in this interval.

 We claim that for any color $d\in D$, there are at least $m/3$ other colors $c$ such that the pair $(c, d)$ is good. Fixing color $d$, let $j$ be the smallest element of $[N]$ receiving color $d$, and let $s$ be the maximal element $1\leq s\leq j$ such that the interval $[s,j]$ receives $m/3$ colors. Then, any color $c\in C$ not appearing in the interval $[s, j]$ works. Indeed, if $i$ is the largest element of $[1, j]$ receiving color $c$, then $i< s$. The interval $[i, j]$ then receives at least as many colors at $[s, j]$, i.e. at least $m/3$ colors. Moreover, no element of $[i+1, j]$ receives color $c$, by the maximality of $i$, and no element of $\{i, \dots, j-1\}$ receives color $d$, by minimality of $j$. This gives at least $m/3 \cdot m/3\geq m^2/9$ good pairs $(c, d)$.
\end{proof}

We introduce another essential definition for the proof. Recall that the curves of $\cA, \cB$ are directed. We say that a pair of curves $(a,b)$ intersect \textit{positively} at the point $p$ if $p\in a\cap b$ and the tangent vectors to $a$ and $b$ at $p$ form a positive basis of $\bR^2$. Informally, this means that if one walks along $a$, at the point $p$ the curve $b$ crosses from the right side of $a$ to the left side of $a$.  Also, if $(a, b)$ do not intersect positively at $p$, we say that they intersect negatively, and we note that reversing the direction of either curve switches the sign of their intersection.
Then, for an integer $1\leq j\leq k+1$ and curves $a_1, a_2\in \cA$ and $b\in \cB$, say that the curve $b$ is \textit{$j$-good} with respect to the pair $(a_1, a_2)$ if the following conditions are satisfied: there exists a directed subcurve $I\subseteq I_j(b)$, inheriting the direction from $b$, such that
\begin{itemize}
    \item $I$ starts at $p\in a_1\cap b$, ends at $q\in a_2\cap b$, $(a_1, b)$ intersect positively at $p$, and $(a_2, b)$ intersect positively at $q$
    \item $I$ is disjoint from $a_1\cup a_2$, except at endpoints, and
    \item $I$ intersects at least $\frac{c}{12k} n$ distinct curves from $\cA$.
\end{itemize}
In what follows, we fix a pair of curves $(a_1, a_2)$ and index $j$. Let $\cB_{\rm good}$ be the set of $j$-good curves with respect to $(a_1,a_2)$. If $b\in \cB_{\rm good}$, we denote by $I(b)$ the subcurve $I$, by $p(b)$ the point $p$, and by $q(b)$ the point $q$.

We introduce the final definition needed for the proof. For a parameter $\mu>0$, define a \textit{$(\mu, t)$-ladder} (with respect to the pair of curves $(a_1,a_2)$ and index $j$) as a $(t+1)$-tuple of curves $b_1, b_3, \dots, b_{2t+1}\in \cB_{\rm good}$, together with a $t$-tuple of sets $\cB_2, \cB_4, \dots, \cB_{2t}\subset \cB_{\rm good}$ of size at least $\mu n$, such that the following properties hold.
\begin{itemize}
    \item There is a set $\cA'\subset \cA$ of at least $\mu n$ curves such that for any $a\in \cA'$ and any $b\in \{b_1, b_3, \dots, b_{2t+1}\}$, the subcurve $I(b)$ intersects $a$.
    \item For any  $b_2\in \cB_2, \dots, b_{2t}\in \cB_{2t}$, the intersection points $p(b_1), p(b_2), \dots, p(b_{2t+1})$ appear in order of indices on $a_1$, and also $q(b_1), q(b_2), \dots, q(b_{2t+1})$ appear in order of indices on $a_2$.
    \item For any $b_2\in \cB_2, \dots, b_{2t}\in \cB_{2t}$, the subcurves  $I(b_1),  I(b_2), \dots, I(b_{2t+1})$ are disjoint.
    \item  $a_1$ does not contain any points of $a_1\cap a_2$ between $p(b_1)$ and $p(b_{2t+1})$, and $a_2$ does not contain any points of $a_1\cap a_2$ between $q(b_1)$ and $q(b_{2t+1})$.
\end{itemize}
This definition is illustrated in Figure~\ref{fig:ladder}. 

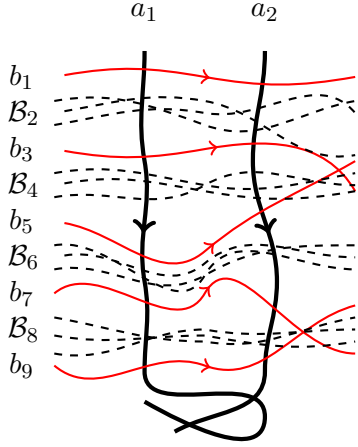
\begin{figure}

\vspace{-0.5cm}
    \centering
        \begin{tikzpicture}[scale=0.8]
            \node[label={[label distance=0]90:$a_1$}] (a1) at (-1, 3) {};
            
            \draw [ultra thick,black,
                postaction={decorate},
                decoration={markings, mark=at position 0.3 with {\arrow[scale=1.2]{>}}}]
            (a1)
            to[out=-90, in=100] ($(a1)+(0, -2)$)
            to[out=-80, in=100] ($(a1)+(0, -4)$)
            to[out=-80, in=90] ($(a1)+(0, -5.5)$)
            to[out=-90, in=90] ($(a1)+(2, -6.25)$)
            to[out=-90, in=-30] ($(a1)+(0, -6)$);
    
            \node[label={[label distance=0]90:$a_2$}] (a2) at (1, 3) {};
            
            \draw [ultra thick,black,
                postaction={decorate},
                decoration={markings, mark=at position 0.42 with {\arrow[scale=1.2]{>}}}]
            (a2)
            to[in=90,out=-90] ($(a2)+(-0.2, -2)$)
            to[in=90,out=-90] ($(a2)+(0.2, -4)$)
            to[in=90,out=-90] ($(a2)+(0, -5.5)$)
            to[in=30,out=-90] ($(a2)+(-1.5, -6.5)$);
    
            \node[label={[label distance=0]180:$b_1$}] (b1) at ($(a1)+(-1.5, -0.6)$) {};
            \draw[thick, red,
                postaction={decorate},
                decoration={markings, mark=at position 0.5 with {\arrow[scale=1.2]{>}}}] (b1) to[in=170, out=10] ($(b1)+(2.5, 0)$) to[in=-170, out=-10] ($(b1)+(5, 0)$);
    
            \node[label={[label distance=0]180:$b_3$}] (b2) at ($(a1)+(-1.5, -1.8)$) {};
            \draw[thick, red,
                postaction={decorate},
                decoration={markings, mark=at position 0.5 with {\arrow[scale=1.2]{>}}}] (b2) to[in=-170, out=-10] ($(b2)+(2.5, 0)$) to[in=120, out=10] ($(b2)+(5, -0.7)$);
    
            \node[label={[label distance=0]180:$b_5$}] (b3) at ($(a1)+(-1.5, -3)$) {};
            \draw[thick, red,
                postaction={decorate},
                decoration={markings, mark=at position 0.5 with {\arrow[scale=1.2]{>}}}] (b3) to[in=-140, out=-10] ($(b3)+(2.5, -0.5)$) to[in=-150, out=40] ($(b3)+(5, 1)$);
            
            \node[label={[label distance=0]180:$b_7$}] (b4) at ($(a1)+(-1.5, -4.2)$) {};
            \draw[thick, red,
                postaction={decorate},
                decoration={markings, mark=at position 0.5 with {\arrow[scale=1.2]{>}}}] ($(b4)+(0, 0)$) to[in=-120, out=40] ($(b4)+(2.5, 0)$) to[in=180, out=60] ($(b4)+(5, -1)$);
    
            \node[label={[label distance=0]180:$b_9$}] (b5) at ($(a1)+(-1.5, -5.4)$) {};
            \draw[thick, red,
                postaction={decorate},
                decoration={markings, mark=at position 0.5 with {\arrow[scale=1.2]{>}}}] ($(b5)+(0, 0)$) to[in=160, out=-40] ($(b5)+(2.5, 0)$) to[in=195, out=-20] ($(b5)+(5, 1)$);

            \node[label={[label distance=-2]180:$\cB_2$}] (b1') at ($(b1)+(0, -0.6)$) {}; 
            \draw[thick, dashed, black] ($(b1')+(0, 0.2)$) to[in=160, out=10] ($(b1')+(2.5, -0.2)$) to[in=-170, out=-20] ($(b1')+(5, 0.2)$);
            \draw[thick, dashed, black] (b1') to[in=170, out=10] ($(b1')+(2.5, 0.2)$) to[in=-170, out=-10] ($(b1')+(5, -0.7)$);
            \draw[thick, dashed, black] ($(b1')+(0, -0.2)$) to[in=170, out=10] ($(b1')+(2.5, 0)$) to[in=130, out=-10] ($(b1')+(5, 0)$);
    
            \node[label={[label distance=-2]180:$\cB_4$}] (b2') at ($(b2)+(0, -0.6)$) {}; 
            \draw[thick, dashed, black] ($(b2')+(0, 0.2)$) to[in=170, out=10] ($(b2')+(2.5, 0)$) to[in=-170, out=-10] ($(b2')+(5, 0.2)$);
            \draw[thick, dashed, black] (b2') to[in=170, out=10] ($(b2')+(2.5, -0.1)$) to[in=-170, out=-10] ($(b2')+(5, -0.1)$);
            \draw[thick, dashed, black] ($(b2')+(0, -0.2)$) to[in=-150, out=10] ($(b2')+(2.5, 0)$) to[in=180, out=30] ($(b2')+(5, 0)$);
    
            \node[label={[label distance=-2]180:$\cB_6$}] (b3') at ($(b3)+(0, -0.6)$) {}; 
            \draw[thick, dashed, black] ($(b3')+(0, 0.2)$) to[in=-120, out=10] ($(b3')+(2.5, -0.3)$) to[in=180, out=60] ($(b3')+(5, -0.2)$);
            \draw[thick, dashed, black] (b3') to[in=-140, out=10] ($(b3')+(2.5, -0.1)$) to[in=180, out=40] ($(b3')+(5, 0)$);
            \draw[thick, dashed, black] ($(b3')+(0, -0.2)$) to[in=-150, out=10] ($(b3')+(2.5, -0.3)$) to[in=180, out=30] ($(b3')+(5, 0.2)$);
    
            \node[label={[label distance=-2]180:$\cB_8$}] (b4') at ($(b4)+(0, -0.6)$) {}; 
            \draw[thick, dashed, black] ($(b4')+(0, 0.2)$) to[in=180, out=-10] ($(b4')+(2.5, -0.2)$) to[in=-170, out=-10] ($(b4')+(5, 0.2)$);
            \draw[thick, dashed, black] ($(b4')+(0, -0.1)$) to[in=180, out=-10] ($(b4')+(2.5, -0.1)$) to[in=-170, out=-10] ($(b4')+(5, -0.2)$);
            \draw[thick, dashed, black] ($(b4')+(0, -0.4)$) to[in=180, out=-10] ($(b4')+(2.5, 0)$) to[in=180, out=-10] ($(b4')+(5, 0)$);
        \end{tikzpicture}
\vspace{-0.4cm}
    \caption{Illustration of a ladder}
    \label{fig:ladder}
\vspace{-0.5cm}
\end{figure}

\begin{lemma}\label{lemma:ladder}
For every $t\geq 1$, there exists $\mu=\mu(t,k,c)>0$ such that the following holds. After possibly switching the directions of some curves in $\cA\cup \cB$, there exists a $(\mu, t)$-ladder with respect to some pair of curves $(a_1,a_2)\in\mathcal{A}^2$ an index $j$. 
\end{lemma}
\begin{proof}
To simplify notation, $\Omega(\cdot)$ hides any constant which may depend on $t,k$ and $c$, but no other parameter. Let $\ell_5=\lceil 12k(t+1)/c\rceil$, $\ell_4=2\ell_5+1, \ell_3=\ell_4^4, \ell_2=(k+1)^2\ell_3$ and $\ell_1=(12k/c)^2\ell_2$. By Lemma \ref{lemma:avoiding ell-tuples}, there exists an index $j\in [k+1]$ with the property that at least $\Omega(n^{\ell_1})$ $\ell_1$-tuples of curves from $\cB$ avoid index $j$. We fix such a choice of $j$. The rest of the proof is divided into five steps.

\medskip\noindent
\textbf{Step 1: Finding tuples of $j$-good curves which avoid index $j$.}
Pick a random pair of curves $a_1, a_2\in \cA$. Fixing a curve $b\in \cB$, the probability that $b$ is $j$-good with respect to $(a_1, a_2)$, i.e. $\Pb[b\in \cB_{\rm good}]$ can be lower bounded as follows.

Recall that the elements of $P_b=\bigcup_{a\in \cA}b\cap a$ are colored by the curve $a$ containing them. Also, $I_j(b)\cap P_b$ receives at least $\frac{c}{2k} n$ different colors. The intersection points in $I_j(b)\cap P_b$ are either positive or negative. By the Pigeonhole principle, one of these two groups receives at least $\frac{c}{4k}n$ different colors. If it is the negative group, we reverse the direction of $b$ -- which switches the positive and the negative intersections, meaning that we can assume that the positive intersections contribute at least $\frac{c}{4k}n$ different colors.

Let $N=|P_b\cap I_j(b)|$ and let $m\geq \frac{c}{4k} n$ be the number of colors received by these $N$ points. By Lemma~\ref{lemma:coloring lemma}, there exist at least $m^2/9$ pairs $a_1, a_2\in \cA$, for which there are intersection points $p\in a_1\cap b, q\in a_2\cap b$ such that the colors $a_1, a_2$ do not appear between $p$ and $q$, and such that at least $m/3$ colors appear between $p$ and $q$. In other words, this means that there are at least $m^2/9$ pairs $a_1, a_2\in \cA$ such that $b$ is $j$-good with respect to $(a_1, a_2)$. Hence, for a fixed curve $b\in \cB$, we have \[\Pb\big[b\in \cB_{\rm good}\big]\geq \frac{m^2/9}{n^2}\geq \Big(\frac{c}{6k}\Big)^2.\]

We denote this probability by $\tau=(c/6k)^2$, and we denote by $\cL_1$ the set of $\ell_1$-tuples $(b_1, \dots, b_{\ell_1})\in \cB^{\ell_1}$ avoiding the index $j$. By linearity of expectation, $\E\big[|\{b_1, \dots, b_{\ell_1}\}\cap \cB_{\rm good}|\big]\geq \tau \ell_1$. Since $\ell_2=\tau \ell_1/4$, we must have 
\[\Pb\Big[\big|\{b_1, \dots, b_{\ell_1}\}\cap \cB_{\rm good}\big|\geq \ell_2\Big]\geq \tau /2.\]
Indeed, if this was not true, we would get a contradiction from \[\E\big[|\{b_1, \dots, b_{\ell_1}\}\cap \cB_{\rm good}|\big]\leq \ell_1\cdot \Pb\Big[\big|\{b_1, \dots, b_{\ell_1}\}\cap \cB_{\rm good}\big|\geq \ell_2\Big]+\ell_2\cdot 1<  \ell_1\tau/2+ \tau\ell_1/4<\tau\ell_1.\]

This implies the existence of a pair $a_1, a_2\in \cA$ for which there are at least $\Omega(\tau n^{\ell_1})=\Omega(n^{\ell_1})$ $\ell_1$-tuples of $\cL_1$ with at least $\ell_2$ curves belonging to $\cB_{\rm good}$. At the same time, any $\ell_2$-tuple of curves from $\cB_{\rm good}$ is contained in at most $n^{\ell_1-\ell_2}$ $\ell_1$-tuples of $\cL_1$, and therefore there must be at least $\frac{\Omega(n^{\ell_1})}{n^{\ell_1-\ell_2}}=\Omega(n^{\ell_2})$ $\ell_2$-tuples of curves from $\cB_{\rm good}$ avoiding index $j$. We  denote this set of $\ell_2$-tuples by $\cL_2$. 

\medskip\noindent
\textbf{Step 2: Avoiding intersections of $a_1$ and $a_2$.} The curve $a_1$ contains at most $k$ points of intersection with $a_2$, which partition $a_1$ into at most $k+1$ subcurves. We denote the subcurves by $I_1, \dots, I_{r+1}$, where $r=|a_1\cap a_2|$. Similarly, we can partition $a_2$ into $r+1$ subcurves $J_1, \dots, J_{r+1}$. For each $(b_1, \dots, b_{\ell_2})\in \cL_2$, there are intervals $I_s$ and $J_{s'}$ such that at least $\ell_2/(r+1)^2\geq \ell_3$ curves $b\in \{b_1, \dots, b_{\ell_2}\}$ satisfy $p(b)\in I_s$ and $q(b)\in J_{s'}$. In other words, for each $(b_1, \dots, b_{\ell_2})\in \cL_2$,  there is a subtuple $(b_{i_1}, \dots, b_{i_{\ell_3}})$ with the property that $a_1$ does not intersect $a_2$ between in the interval containing $\{p(b_{i_1}), \dots, p(b_{i_{\ell_3}})\}$ and similarly for $a_2$. If we denote the set of such $\ell_3$-tuples by $\cL_3$, the usual counting argument gives $|\cL_3|\geq |\cL_2|/n^{\ell_2-\ell_3}=\Omega(n^{\ell_3})$.

\medskip\noindent
\textbf{Step 3: Ordering the tuples.} For each $(b_1, \dots, b_{\ell_3})\in \cL_3$,  consider the order in which the points $p(b_1), \dots, p(b_{\ell_3})$ appear on $a_1$. By the Erd\H{o}s-Szekeres theorem, there exists a set of $r=\ell_2^{1/2}$ curves among them, say $b_{i_1}, \dots, b_{i_{r}}$, such that $p(b_{i_1}), ..., p(b_{i_r})$ appear on $a_1$ in the increasing or decreasing order. Similarly, among those $r$ curves, there is a set of $r^{1/2}=\ell_3^{1/4}=\ell_4$ curves which appear on $a_2$ in either increasing or decreasing order. 

Note that if $(b_{1}, \dots, b_{\ell_4})$ is any sequence of $j$-good curves, then it is impossible that $p(b_{1}), ..., p(b_{\ell_4})$ appear on $a_1$ in an increasing order, but $q(b_{1}), ..., q(b_{\ell_4})$ appear on $a_2$ in decreasing order, or vice versa. Indeed, it is impossible even to have three curves, say $b_1, b_2, b_3$, which all intersect $a_1, a_2$ positively and such that $p(b_1), p(b_2), p(b_3)$ appear in the increasing order on $a_1$, while $q(b_1), q(b_2), q(b_3)$ appear in the decreasing order on $a_2$. To see why, consider the Jordan curve $C$ formed by concatenating the following four arcs: subcurve $a_1$ between $p(b_1)$ and $p(b_3)$, subcurve of $b_3$ from $p(b_3)$ to $q(b_3)$, subcurve of $a_2$ from $q(b_3)$ to $q(b_1)$ and subcurve of $b_1$ from $q(b_1)$ to $p(b_1)$. The curve $C$ splits the plane into two regions. When walking along the curve $a_1$ between $p(b_1)$ and $p(b_3)$, one of these regions, which we call $L$, is on the left side of $a_1$, and the other one, which we call $R$, is on the right side. Due to the positive intersection property, $L$ is also on the left side of $b_3$ when we walk along it from $p(b_3)$ to $q(b_3)$, and thus $L$ is on the left side of $a_2$ when traversed from $q(b_3)$ to $q(b_1)$. However, the curve $b_2$ intersects the boundary of $L$ exactly twice, once on $a_1$ and once on $a_2$, and due to the positive intersection property, it is directed from $R$ to $L$ at both of these intersection points. But this is impossible, and thus we derive a contradiction. 

Therefore, we conclude that for each $(b_1, \dots, b_{\ell_3})\in \cL_3$, there are $\ell_4$ curves $b_{i_1}, \dots, b_{i_{\ell_4}}$ such that $p(b_{i_1}), ..., p(b_{i_{\ell_4}})$ appear on $a_1$ in the increasing or decreasing order and $q(b_{i_1}), ..., q(b_{i_{\ell_4}})$ appear on $a_2$ in the same order. If for at least half of $\ell_3$-tuples of $\cL_3$,  we have that $p(b_{i_1}), ..., p(b_{i_{\ell_4}})$ appear on $a_1$ in the decreasing order, we reverse the directions of all curves in $\cA$ and $\cB$, and we switch the roles of $a_1$ and $a_2$. This does not affect the positive intersection condition, and thus does not change the set $\cB_{\rm good}$, but it ensures that for at least half of $\ell_3$-tuples $(b_1, \dots, b_{\ell_3})\in \cL_3$, there are $\ell_4$ curves $b_{i_1}, \dots, b_{i_{\ell_4}}$ such that $p(b_{i_1}), ..., p(b_{i_{\ell_4}})$ appear on $a_1$ in the increasing order. 

We denote the set of such $\ell_4$-tuples by $\cL_4$, and note that each of them still avoids the index $j$ and has elements from $\cB_{\rm good}$. Finally, since each $\ell_4$-tuple is contained in at most $n^{\ell_3-\ell_4}$ $\ell_3$-tuples of $\cL_3$, we have $|\cL_4|\geq |\cL_3|/n^{\ell_3-\ell_4}=\Omega(n^{\ell_4})$.

\medskip\noindent
\textbf{Step 4: Fixing the rungs of the ladder.} For each $\ell_4=(2\ell_5+1)$-tuple $(b_1, \dots, b_{2\ell_5+1})\in \cL_4$, consider only the odd-indexed elements, i.e. the  $(\ell_5+1)$-tuple of curves $(b_1, b_3, \dots, b_{2\ell_5+1})$. Since there are at most $n^{\ell_5+1}$ possibilities for this $(\ell_5+1)$-tuple,  some $(\ell_5+1)$-tuple appears in at least $|\cL_4|/n^{\ell_5+1}=\Omega(n^\ell_5)$ pieces of $(2\ell_5+1)$-tuples of $\cL_4$. Fix the most popular $(\ell_5+1)$-tuple $(b_1,b_3,\dots,b_{2\ell_5+1})$, and denote by $\cL_5$ elements of $\cL_4$ which have this $(\ell_5+1)$-sequence as their odd-indexed elements. For $i\in \{2, 4, \dots, 2\ell_5\}$, let $\cB_i$ be the set of curves which  appear as $i$-th $b_i$ in some $(2\ell_5+1)$-tuple $(b_1,\dots,b_{2\ell_5+1})\in \cL_5$. Then, we have $|\cB_2||\cB_4|\dots|\cB_{2\ell_5}|\geq |\cL_5|=\Omega(n^{\ell_5})$, which implies that  $|\cB_i|\geq \Omega(n)$ for every $i\in\{2,4,\dots,2\ell_5\}$. Also, for any $b_2\in \cB_2, \dots, b_{2\ell_5}\in \cB_{2\ell_5}$, the subcurves $I(b_1), I(b_2), \dots, I(b_{2\ell_5+1})$ are disjoint. Note that we do not claim that \textit{every} $(2\ell_5+1)$-tuple $(b_1, \dots, b_{2\ell_5+1})$ with $b_i\in \cB_i$ is in $\cL_4$.

\medskip\noindent
\textbf{Step 5: Finding $\cA'$.}
For $i\in \{1, 3, \dots, 2\ell_5+1\}$, let  $\cA_i$ be the set of curves intersecting $b_i$ in $I(b_i)$, then $|\cA_i|\geq \frac{c}{12k}n$. The goal is to find a $(t+1)$-tuple $(b_{i_1}, \dots, b_{i_{t+1}})$ for which $\big|\bigcap_{s=1}^{t+1} \cA_{i_s}\big|\geq \Omega(n)$. For a curve $a \in \cA$, let $d_a$ denote the number of sets $\cA_i$ containing $a$. Then $\sum_{a \in \cA} d_a = \sum_{i=1}^{\ell_5+1} |\cA_{2i-i}| \ge \ell_5\cdot \frac{c}{12k} n$. As $d_a\leq \ell_5$, there is a set $\cA^*\subset \cA$ of at least $\frac{c}{24k}n$  curves  that are contained in at least $\frac{c}{24k}\ell_5\geq t+1$ sets $\cA_i$. There are $\binom{\ell_5}{t+1}$ different $(t+1)$-tuples, so there is a $(t+1)$-tuple  $(b_{i_1}, \dots, b_{i_{t+1}})$ and a set $\cA'\subset \cA^*$ of size at least $|\cA'|\geq |\cA^*|/\binom{\ell_5}{t+1}\geq cn/(24k\binom{\ell_5}{t+1})$ such that $\cA'\subset \bigcap_{s=1}^{t+1} \cA_{i_s}$. But then the $(t+1)$-tuple of curves $b_{i_1},\dots,b_{i_{t+1}}$ together with the $t$-tuple of sets $\cB_{i_1+1},\dots,\cB_{i_{t}+1}$ is a $(\mu,t)$-ladder for some $\mu=\mu(t,k,c)>0$.
\end{proof}

\begin{proof}[Proof of Theorem \ref{thm:stEH_kint}]
Let $t=20k$, then by Lemma \ref{lemma:ladder}, there exists $\mu=\mu(k,c)$ and a $(\mu,t)$-ladder with respect to some pair $(a_1,a_2)$ and index $j$ consisting of curves $b_1,b_3, \dots, b_{2t+1}\in\cB$ and sets $\cB_2,\cB_4, \dots, \cB_{2t}\subset \cB$, each of size at least $\mu n$. In this proof, we only consider the subcurve $I(b)$ for each curve $b\in \{b_1,b_3,\dots,b_{2t+1}\}\cup \cB_2\cup \cB_4\cup\dots\cup \cB_{2t}$. Therefore, we replace each of these curves by this subcurve.  Then the curves $b_1,\dots, b_{2t+1}$ are disjoint for any choice $b_2\in \cB_2,\dots,b_{2t}\in \cB_{2t}$. Furthermore, there exists  a set $\cA'\subset \cA$ of size at least $\mu n$ such that every $a\in \cA'$ intersects all curves $b_1,b_3 \dots, b_{2t+1}$. 

For  $i\in \{1, \dots, t\}$, let $S_i$  be the Jordan region whose boundary consists of $b_{2i-1}$, the subcurve of $a_2$ from $q(b_{2i-1})$ to $q(b_{2i+1})$, $b_{2i+1}$, and the subcurve of $a_1$ from $p(b_{2i+1})$ to $p(b_{2i-1})$. Note that these four subcurves do not intersect each other, so their concatenation is indeed a simple closed Jordan curve. This is illustrated in Figure \ref{fig:regions Si}. 

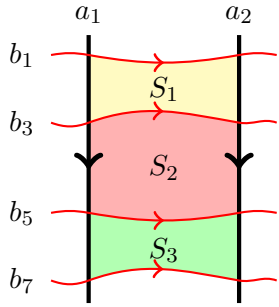
\begin{figure}
    \centering
    \begin{tikzpicture}[scale=1]
        \node[label={[label distance=-7]90:$a_1$}] (a1) at (-1, 3) {};
        \node[label={[label distance=-7]90:$a_2$}] (a2) at (1, 3) {};
        \node[label={[label distance=5]0:$b_1$}] (b1) at ($(a1)+(-1.5, -0.4)$) {};
        \node[label={[label distance=5]0:$b_3$}] (b2) at ($(a1)+(-1.5, -1.3)$) {};
        \node[label={[label distance=5]0:$b_5$}] (b3) at ($(a1)+(-1.5, -2.5)$) {};
        \node[label={[label distance=5]0:$b_7$}] (b4) at ($(a1)+(-1.5, -3.4)$) {};

        \fill[yellow!60, opacity=0.5]
            ($(b1)+(1.5, 0)$) to[out=-10, in=-170] ($(b1)+(3.5, 0)$) to[out=-90, in=90] ($(b2)+(3.5, 0)$) to[out=170, in=20] ($(b2)+(1.5, 0)$) -- cycle;
        \fill[red!60, opacity=0.5]
            ($(b2)+(1.5, 0)$)  
            to[out=20, in=170] ($(b2)+(3.5, 0)$) 
            to[out=-90, in=90] ($(b3)+(3.5, 0)$) 
            to[out=-170, in=-10] ($(b3)+(1.5, 0)$) --  cycle;
        \fill[green!60, opacity=0.5]
            ($(b3)+(1.5, 0)$)  
            to[out=-10, in=-170] ($(b3)+(3.5, 0)$) 
            to[out=90, in=-90] ($(b4)+(3.5, 0)$)
            to[out=170, in=20] ($(b4)+(1.5, 0)$)  
            --  cycle;

        \draw [ultra thick,black,
                postaction={decorate},
                decoration={markings, mark=at position 0.5 with {\arrow[scale=1.5]{>}}}]
        (a1)
        to[out=-90, in=90] ($(a1)+(0, -2)$)
        to[out=-90, in=90] ($(a1)+(0, -3.7)$);
        
        \draw [ultra thick,black,
                postaction={decorate},
                decoration={markings, mark=at position 0.5 with {\arrow[scale=1.5]{>}}}]
        (a2)
        to[in=90,out=-90] ($(a2)+(0, -2)$)
        to[in=90,out=-90] ($(a2)+(0, -3.7)$);
        
        \draw[thick, red,
                postaction={decorate},
                decoration={markings, mark=at position 0.5 with {\arrow[scale=1.5]{>}}}]
                ($(b1)+(1, 0)$) to[out=10, in=170] ($(b1)+(1.5, 0)$) to[out=-10, in=-170] ($(b1)+(3.5, 0)$) to[out=10, in=170] ($(b1)+(4, 0)$);
        \draw[thick, red,
                postaction={decorate},
                decoration={markings, mark=at position 0.5 with {\arrow[scale=1.5]{>}}}]
                ($(b2)+(1, 0)$) to[out=-20, in=-160] ($(b2)+(1.5, 0)$) to[out=20, in=170] ($(b2)+(3.5, 0)$) to[out=-10, in=160] ($(b2)+(4, 0)$);
        \draw[thick, red,
                postaction={decorate},
                decoration={markings, mark=at position 0.5 with {\arrow[scale=1.5]{>}}}]
                ($(b3)+(1, 0)$) to[out=10, in=170] ($(b3)+(1.5, 0)$) to[out=-10, in=-170] ($(b3)+(3.5, 0)$) to[out=10, in=170] ($(b3)+(4, 0)$);
        \draw[thick, red,
                postaction={decorate},
                decoration={markings, mark=at position 0.5 with {\arrow[scale=1.5]{>}}}]
                ($(b4)+(1, 0)$) to[out=-20, in=-160] ($(b4)+(1.5, 0)$) to[out=20, in=170] ($(b4)+(3.5, 0)$) to[out=-10, in=160] ($(b4)+(4, 0)$);

        \node[] at (0, 2.15) {$S_1$};
        \node[] at (0, 1.1) {$S_2$};
        \node[] at (0, 0) {$S_3$};
    \end{tikzpicture}

\vspace{-0.2cm}
    \caption{Illustration of the regions $S_i$ defined in the proof of Theorem~\ref{thm:stEH_kint}.}
    \label{fig:regions Si}
\vspace{-0.3cm}
\end{figure}

Let $\cA_i$ be the set of curves $a\in \cA'$ which either have an endpoint in $S_i$, or intersects the boundary of $S_i$  on $a_1$ or $a_2$. Since each $a\in \cA'$ intersects $a_1, a_2$ in at most $k$ points, and has two endpoints, it can appear in at most $2k+2$ different sets $\cA_i$. Therefore, we have $\sum_{i=1}^t |\cA_i|\leq (2k+2)|\cA'|$. Thus, there exists an index $i\in \{2, 4, \dots, t-2\}$ such that $|\cA_i|+ |\cA_{i+1}|\leq \frac{2k+2}{t/2-1} |\cA'|\leq \frac{1}{3}|\cA'|$. Fix such an index $i$, and let $\cA''=\cA'\backslash (\cA_i\cup \cA_{i+1})$.

The set $\cA''$ has size at least $\frac{2}{3}|\cA''|\geq \frac{2}{3}\mu n$. Moreover, a curve $a\in \cA''$ has no endpoints in $S_i\cup S_{i+1}$, which is a Jordan region bounded by $b_{2i-1}, b_{2i+3}$, the subcurve of $a_1$ from $p(b_{2i-1})$ to $p(b_{2i+3})$ and the subcurve of $a_2$ from $q(b_{2i-1})$ to $q(b_{2i+3})$. Also, $a$ does not intersect the parts of the boundary of $S_i\cup S_{i+1}$ on $a_1$ and $a_2$. Let $x_1,\dots,x_s$ be the intersection points of $a$ with $b_{2i-1}\cup b_{2i+1}\cup b_{2i+3}$, where the indices follow the ordering of $a$. Then $x_1,x_s\in b_{2i-1}\cup b_{2i+3}$ and there is some $1<p<s$ such that $x_p\in b_{2i+1}$. Hence, there exists an index $1<q<s$ such that $x_q\in b_{2i+1}$ and $x_{q+1}\in b_{2i-1}\cup b_{2i+3}$. Without loss of generality, assume that $x_{q+1}\in b_{2i-1}$, the other case can be handled similarly. Then the subcurve of $a$ between $x_q$ and $x_{q+1}$ is contained in $S_i$, and it intersects every element of $\cB_{2i}$. Here, we are using that $a_1$ intersects the elements of $\cB_{2i}$ positively, and thus the curves of $\cB_{2i}$ pass through $S_i$, joining $a_1$ and $a_2$.

In conclusion, there is a set $\cA'''\subset \cA''$ of size at least $|\cA''|/2\geq \mu n/3$ such that that either every element of $\cA'''$ intersects every element of $\cB_{2i}$, or every element of $\cA'''$ intersects every element of $\cB_{2i+2}$. This finishes the proof.
\end{proof}

\subsection{Intersection graphs of pseudodisks}

In this section, we prove Theorem \ref{thm:pseudodisk optimal}. Our strategy is to show that the family $\cF_{\text{pdisk}}$ of bipartite intersection graphs of two families of $y$-monotone pseudodisks is degree-bounded and has the density-EH property. The degree-boundedness already follows from the result of Keller and Smorodinsky \cite{KS}, who show that the $K_{s,s}$-free members of $\cF_{\text{pdisk}}$  have average degree $O(s^6)$. In particular, a result of Keszegh \cite{Keszegh} implies that every member of $\cF_{\text{pdisk}}$   avoids an induced subdivision of $K_5$, where all root vertices are in one of the parts, thus they avoid induced subdivisions of $K_9$, so the theorem of K\"uhn and Osthus \cite{KO04} is also applicable to this end. Therefore, it remains to show that $\cF_{\text{pdisk}}$ has the density-EH property. In order to prove this, we use a combination of Theorem \ref{thm:stEH_kint} and the following cutting lemma of \cite{CCH12}. To be precise, the following lemma is stated for certain ``well-behaved'' families of shapes in $\mathbb{R}^d$, but $y$-monotone pseudodisk families fall into this category, see \cite{CH23}.

\begin{lemma}[Cutting lemma for pseudodisks \cite{CCH12}]\label{lemma:cutting pseudodisks}
Let $\cA$ be a family of $y$-monotone pseudodisks in the plane of size $n$ and let $1\leq r\leq n$. Then there exists a decomposition of the plane into $O(r^2)$ connected regions such that any region intersects the boundaries of  at most $n/r$  of pseudodisks in $\cA$.
\end{lemma}

\begin{lemma}\label{lemma:point-pseudodisk dense}
Let $\eps>0$, then there exists $c>0$ such that the following holds. Let $\cA$ and $\cB$ be two $n$-element families of $y$-monotone pseudodisks. Let $G$ be the bipartite intersection graph between $\cA$ and $\cB$. If $G$ has at least $\eps n^2$ edges, then $G$ contains a biclique of size at least $cn$.
\end{lemma}
\begin{proof}
Let $G_0$ be the subgraph of $G$ composed of those edges $\{A,B\}$ with $A\in \cA$ and $B\in \cB$ for which the boundaries of $A$ and $B$ intersect. Let $G_1$ be the subgraph of those edges for which $\{A,B\}$ for which $A\subset B$. Finally, let $G_2$ be the rest of the edges, that it, those edges $\{A,B\}$ where $B\subset A$. As $G$ is the union of $G_0,G_1,G_2$, we have $e(G_i)\geq \eps n^2/3$ for some $i\in \{0,1,2\}$.
   
First, consider the case $e(G_0)\geq \eps n^2/3$. Let $\cA'$ be the family of boundaries of the elements of $\cA$, and define $\cB'$ analogously. Then $G_0$ is the bipartite intersection graph between $\cA'$ and $\cB'$, and $\cA'$ and $\cB'$ are 2-intersecting families curves. Therefore, Theorem \ref{thm:stEH_kint} implies the existence of some $c>0$ such that $G_0$ contains a biclique of size at least $cn$.

Now consider the case $e(G_1)\geq \eps n^2/3$, and note that the case $e(G_2)\geq \eps n^2/3$ is essentially identical. Replace every element $A\in \cA$ by an arbitary point $p\in A$, and let $\cP$ be the set of resulting points. Then $G_1$ is the incidence graph of $P$ and $\cB$. First, delete every element of $\cP$ that is contained in fewer than $\eps n/6$ elements of $\cB$. This removes at most $\eps n^2/6$ incidences, let $\cP'$ denote the set of remaining points. Note that $|\cP'|\geq \eps n/6$.

Apply Lemma~\ref{lemma:cutting pseudodisks} to the family  $\cB$ with the parameter $r=\eps /12$. This produces a set of at most $\ell\leq O(\eps^{-2})$ regions in the plane, denoted by $\Delta_1, \dots, \Delta_\ell$, where the region $\Delta_i$ intersects at most $\eps n/12$ boundaries of pseudodisks in $\cB$. By the Pigeonhole principle, there is a region $\Delta_i$ containing a set $\cP''$ of at least $|\cP'|/\ell\geq \Omega(\eps^3 n)$ points of $\cP'$. Choose any point $p\in \cP'\cap \Delta_i$, and oberve that $p$  is contained in at least $\eps n/2$ pseudodisks, while at most $\eps n/12$ pseudodisk boundaries intersect $\Delta_i$. Hence, at least $\eps n/12$ pseudodisks of $\cB$ contain $\Delta_i$ completely. This gives a set of $\eps n/12$ pseudodisks, which all contain $\cP''$. Thus we found a complete bipartite graph in $G_1$ with vertex classes of size $\eps n/12$ and $\Omega(\eps^3 n)$. This finishes the proof. 
\end{proof}

\section{Semilinear graphs}\label{sec:semilinear}

In this section, we prove Theorem \ref{thm:semilinear zarankiewicz}. We take a geometric approach by representing semilinear graphs as an incidence graph between points and polytopes. We prove the following proposition, from which Theorem~\ref{thm:semilinear zarankiewicz} follows almost immediately.

\begin{proposition}\label{prop:point-polytope incidence graphs}
Let $\cH$ be a collection of $h$ halfspaces in $\bR^d$ and let ${\rm POL}(\cH)$ be the collection of polytopes formed as intersections of translates of members of $\cH$. If $X\subseteq {\rm POL}(\cH)$ and $Y\subseteq \bR^d$ are sets of size $n$ whose incidence graph is $K_{s, s}$-free, then there are at most $O_{d, h}\Big(sn \Big(\frac{\log n}{\log\log n}\Big)^{d-1}\Big)$ incidences.
\end{proposition}

\begin{proof}[Proof of Theorem \ref{thm:semilinear zarankiewicz} assuming Proposition \ref{prop:point-polytope incidence graphs}] Let $\Gamma$ be an $n$-vertex semilinear graph of dimension $(d_x,d_y)$ and complexity $(h,t)$ which is $K_{s,s}$-free. Let $X\subset \mathbb{R}^{d_x}$ and $Y\subset \mathbb{R}^{d_y}$ be the vertex classes of $\Gamma$. By the definition of semilinear graphs, there exist linear functions $f_{i,j}:\mathbb{R}^{d_x}\times \mathbb{R}^{d_y}\rightarrow\mathbb{R}$ such that $x\in X$ and $y\in Y$ are joined by an edge if and only if there exists $j\in [t]$ such that $$\bigwedge_{i\in [h]}\{f_{i,j}(x,y)\leq 0\}=\mbox{true}.$$
Observe that there exist $j=j_0\in [t]$ such that at least $1/t$-fraction of the edges $\{x,y\}$ satisfy the previous boolean expression. 

Let $\mathcal{H}$ be the collection of $h$ halfspaces in $\mathbb{R}^{d_y}$ defined by $\{y\in \mathbb{R}^{d_y}:f_{i,j_0}(0,y)\leq 0\}$. Then, to every $x\in X$ we can associate a polytope $P_x\subseteq \bR^{d_y}$ containing those points $y\in \bR^{d_y}$ which satisfy the inequalities $f_{i,j_0}(x, y)\leq 0$ for all $i=1, \dots, h$. Note that the facets of this polytope are parallel to the halfspaces of $\cH$, and so $P_x\in {\rm POL}(\cH)$. 

Let $G$ be the incidence graph of $Y$ and the collection of polytopes $P_x$ for each $x\in X$. Then $G$ is a subgraph of $\Gamma$ with $e(G)\geq e(\Gamma)/t$. Since $G$ is also $K_{s,s}$-free, Proposition \ref{prop:point-polytope incidence graphs} implies that $e(G)=O_{d_y, h}\Big(sn \Big(\frac{\log n}{\log\log n}\Big)^{d_y-1}\Big)$. This gives the required upper bound $e(\Gamma)=O_{d_y, h}\Big(tsn \Big(\frac{\log n}{\log\log n}\Big)^{d_y-1}\Big)$.
\end{proof}

If $\cH$ is the collection of halfspaces whose defining hyperplanes are orthogonal to the coordinate axes, then ${\rm POL}(\cH)$ is simply the collection of axis-aligned boxes. In this case, Proposition~\ref{prop:point-polytope incidence graphs} is equivalent to the theorem of Chan and Har-Peled \cite{CH23} which solves the Zarankiewicz problem for point-box incidence graphs. The main idea of our proof is to show that an incidence graph between $Y\subseteq \bR^d$ and $X\subseteq {\rm POL} (\cH)$ can be written as a union of constantly many point-box incidence graphs. The key result which allows us to do this is Lemma~\ref{lemma:polytope decomposition}, which shows that a polytope in ${\rm POL}(\cH)$ is the union  of $O_{d,|\cH|}(1)$ parallelotopes, whose face hyperplane directions are determined by $\cH$. Before we state this lemma, we briefly recall some basic terminology related to polytopes. 

A \emph{polytope} $Q$ is an intersection of finitely many halfspaces in $\bR^d$. If such an intersection is bounded, it is a convex hull of finitely many points in $\bR^d$. We may assume that our polytopes are bounded, as we can intersect our configuration with a large box without changing the incidence structure, and adding only constantly many new hyperplanes to $\cH$.  Furthermore, a \emph{parallelotope} in $\bR^d$ is a polytope which can be obtained from a unit cube by a nondegenerate affine map. Equivalently, $P\subseteq \bR^d$ is a parallelotope if it has $2^d$ vertices of the form $\{x+\sum_{i\in I} v_i | I\subseteq [d]\}$, for some $x, v_1, \dots, v_d\in \bR^d$, where $v_1,\dots,v_d$ are linearly independent.

To simplify  terminology, the last coordinate direction $x_d$ is called the \textit{vertical} direction. By extension, a hyperplane which contains a line parallel to the coordinate axis $x_d$ is also called \textit{vertical}. Note that for $d\geq 3$, there are infinitely many non-parallel vertical hyperplanes in $\bR^d$.

If $Q$ is a polytope, the upper envelope of $Q$ consists of those boundary points $x$ that are the highest points of the intersection of $Q$ and a vertical line, while the lower envelope contains those points that are the lowest points. In case a facet of $H$ is vertical, it is contained neither in the upper nor in the lower envelope.

\begin{lemma}\label{lemma:polytope decomposition}
For every collection $\cH$ of $h$ halfspaces in $\bR^d$, there is a collection $\cH'$ of at most $O_{d, h}(1)$ halfspaces, such that the following holds. For any polytope $Q\in {\rm POL}(\cH)$, there is a collection $\cP$ of $O_{d, h}(1)$  parallelotopes such that $Q=\bigcup_{P\in \cP} P$ and $\cP\subseteq {\rm POL}(\cH')$.
\end{lemma}

Lemma~~\ref{lemma:polytope decomposition} is proved algorithmically, by constructing the required decomposition explicitly. Our algorithm is recursive, and its base case are planar polygons, whose decomposition into parallelograms can be constructed by hand. The main geometric idea behind this algorithm is to use vertical decomposition. This well-known decomposition technique is used in a variety of contexts across discrete and computational geometry. We start by describing our algorithm in the case of planar polygons.

\begin{lemma}\label{lemma:plane decomposition}
Let $Q$ be a convex polygon with $h$ sides. Then $Q$ is the union of at most $4h$ parallelograms such that each side of every parallelogram is either vertical or parallel to a side of $Q$.
\end{lemma}
\begin{proof}
Draw a vertical line from every vertex of $Q$. This set of vertical lines cuts the polygon $Q$ into a collection of cells, which we denote by $\cC_Q$. Note that each member of $\cC_Q$ is either a triangle or a trapezoid, see Figure~\ref{fig:vertical decomposition}. 

\begin{figure}
\centering
\begin{minipage}{.3\textwidth}
  \centering
  \begin{tikzpicture}
    \draw (-1, 0) -- (1, 0.3) -- (1.4, 1) -- (0.3, 2) -- (-1.4, 2.2) -- cycle;
    \draw[dashed] (-1, 0) -- (-1, 2.15);
    \draw[dashed] (0.3, 2) -- (0.3, 0.2);
    \draw[dashed] (1, 0.3) -- (1, 1.37);
\end{tikzpicture}
  \captionof{figure}{Cutting a pentagon using vertical lines}
  \label{fig:vertical decomposition}
\end{minipage}
\hspace{0.5cm}
\begin{minipage}{.3\textwidth}
  \centering
  \begin{tikzpicture}[scale=2]

      \coordinate (A) at (-1, 0) {};
      \coordinate (B) at (-0.5, 1) {};
      \coordinate (C) at (1, 0) {};

      \coordinate (D) at ($(C)!0.5!(B)$) {};
      \coordinate (E) at ($(A)!0.5!(C)$) {};
      \coordinate (F) at ($(A)!0.5!(B)$) {};

      \node[vtx, label=below:$a$] at (A) {};
      \node[vtx, label=$b$] at (B) {};
      \node[vtx, label=below:$c$] at (C) {};
      \node[vtx, label=$d$] at (D) {};
      \node[vtx, label=below:$e$] at (E) {};
      \node[vtx, label=$f$] at (F) {};
      
      \coordinate (sa) at ($0.03*(D)-0.03*(A)$) {};
      \coordinate (sb) at ($0.03*(E)-0.03*(B)$) {};
      \coordinate (sc) at ($0.03*(F)-0.03*(C)$) {};
      \coordinate (sd) at ($0.03*(D)-0.03*(E)$) {};
      \coordinate (se) at ($0.03*(E)-0.03*(F)$) {};
      \coordinate (sf) at ($0.03*(F)-0.03*(D)$) {};
    
        \draw[pattern= horizontal lines,pattern color=red, xscale=10, yscale=10] (D) -- (E) -- (A) -- (F) -- cycle;
        \draw[pattern=vertical lines,pattern color=green, xscale=2] (D) -- (C) -- (E)  -- (F) -- cycle;
        \draw[pattern=dots,pattern color=blue, xscale=2] (D) -- (E) -- (F) -- (B) -- cycle;

        \draw (A) -- (B) -- (C) -- cycle;
        \draw[dashed] (D) -- (E) -- (F) -- cycle;  
  \end{tikzpicture}
  \captionof{figure}{Covering a triangle with three parallelograms}
  \label{fig:decomposing triangles}
\end{minipage}
\hspace{0.5cm}
\begin{minipage}{.3\textwidth}
  \centering
  \begin{tikzpicture}
    \draw (0, 0) -- (1, 1) -- (1, 2) -- (0, 2.5) -- cycle;
    \draw[dashed] (1, 2) -- (0, 1);
    \draw[pattern=dots, pattern color=blue]  (0, 0) -- (1, 1) -- (1, 2) -- (0, 1) -- cycle;
    \draw[pattern=horizontal lines, pattern color=red] (1, 2) -- (0, 1) -- (0, 2.5) -- cycle;
\end{tikzpicture}
  \captionof{figure}{Decomposition of a trapezoid into a parallelogram and a triangle}
  \label{fig:decomposing trapezoids}
\end{minipage}
\end{figure}

Both triangles and trapezoids are easily decomposed into parallelograms. A triangle with vertices $a, b, c$ can be covered by three parallelograms as follows: denote the midpoints of segments $bc, ac, ab$ by $d, e, f$, respectively, and note that the triangle $abc$ is a union of parallelograms $aedf,  bdef$ and $cdfe$, as illustrated in Figure~\ref{fig:decomposing triangles}. Furthermore, each trapezoid can be partitioned into a triangle and a parallelogram by translating its bottom edge upwards until it meets the top edge, as illustrated in Figure~\ref{fig:decomposing trapezoids}. Then, the triangle can be covered by three parallelograms in the way previously described.

Observe that each side of the triangles and trapezoids in $\mathcal{C}_Q$ are either vertical or equal to one of the sides of $Q$. Furthermore, each element $C\in \mathcal{C}_Q$ is covered by at most 4 parallelograms, whose sides are parallel to some side of $C$. This finishes the proof.
\end{proof}

\medskip\noindent
\textbf{Covering polytopes in higher dimensions.} To cover polytopes in higher dimensions with a union of parallelotopes, we follow a four-step algorithm, which we outline now.

\begin{itemize}
    \item In Step 1, we decompose the given polytope $Q$ into a set of cells $\cP_1$, where each cell of $\cP_1$ has only one face in its upper and lower envelope. We call such cells \textit{slanted prisms}. 
    \item In Step 2, we decompose each slanted prism of $\cP_1$ into a collection $\cP_2$ of smaller slanted prisms, where the projection of each prism in $\cP_2$ to the horizontal hyperplane is a $(d-1)$-parallelotope.
    \item In Step 3, we decompose each slanted prism with a parallelotope base from $\cP_2$ into two slanted prisms with parallelotope bases, with the additional property that some edges of the top and the bottom face are parallel. We call the arising collections of prisms $\cP_3$.
    \item In Step 4, we recursively partition one of the vertical facets $F$ of each prism $P\in \cP_3$, and extend the resulting $(d-1)$-parallelotopes covering $F$ to $d$-parallelotopes covering $P$. This results in a collection $\cP_4$ of parallepipeds covering $Q$, achieving our goal.
\end{itemize}

We now give a detailed description of each of the steps.

\medskip\noindent
\textit{Step 1.} Let $Q$ be a polytope with at most $h$ facets. For each $(d-2)$-dimensional face $f$, let $H_f$ be the vertical hyperplane containing $f$. Then the hyperplanes $H_f$ divide $Q$ into cells. Here, each cell has the property that its upper and lower envelope are fully contained in a facet of $Q$, while all other facets of the cell are vertical. We call a polytope with this property a \textit{slanted prism}. Hence, we got a decomposition of $Q$ into a union of slanted prisms, whose collection we denote by $\cP_1$. 

 \begin{figure}[ht]
 \begin{minipage}{0.5\textwidth}
    \centering
\tdplotsetmaincoords{107}{53}
\begin{tikzpicture} [scale=1.4, tdplot_main_coords,
full/.style={black,very thick},
dots/.style={dotted,black,thick},
]

\coordinate (Al) at (0, 0, -1);
\coordinate (Bl) at (1, 0, -1);
\coordinate (Cl) at (1, 1, -1);
\coordinate (Dl) at (0, 1, -1);
\coordinate (B1l) at (2, 0, -1);
\coordinate (D1l) at (0, 2, -1);

\coordinate (A) at (0, 0, -0.5);
\coordinate (B) at (1, 0, 0);
\coordinate (C) at (1, 1, 0.5);
\coordinate (D) at (0, 1, 0);

\coordinate (A') at (0, 0, 2);
\coordinate (B') at (1, 0, 1.7);
\coordinate (C') at (1, 1, 1.5);
\coordinate (D') at (0, 1, 1.8);

\node[vtx] at (A) {};
\node[vtx] at (B) {};
\node[vtx] at (C) {};
\node[vtx] at (D) {};
\node[vtx] at (A') {};
\node[vtx] at (B') {};
\node[vtx] at (C') {};
\node[vtx] at (D') {};

\node[vtx] at (Al) {};
\node[vtx] at (Bl) {};
\node[vtx] at (Cl) {};
\node[vtx] at (Dl) {};
\node[vtx] at (B1l) {};
\node[vtx] at (D1l) {};

\node[label=left:$C$] at ($0.5*(A)+0.5*(A')$) {}; 
\node[label=left:$C'$] at (Al) {}; 

\coordinate (B1) at (2, 0, 0.5);
\node[vtx] at (B1) {};
\coordinate (D1) at (0, 2, 0.5);
\node[vtx] at (D1) {};
\coordinate (B1') at (2, 0, 1.4);
\node[vtx] at (B1') {};
\coordinate (D1') at (0, 2, 1.6);
\node[vtx] at (D1') {};

\draw[full] (A') -- (B1') -- (D1') -- (A');
\draw[full] (A) -- (A');
\draw[dots] (B1) -- (B1');
\draw[dots] (A) -- (B1) -- (D1);
\draw[full] (A) -- (D1);
\draw[full] (D1) -- (D1');

\draw[thick] (B') -- (C') -- (D');
\draw[dashed, thick] (A) -- (B) -- (C) -- (D);
\draw[dashed, thick] (B) -- (B');
\draw[dashed, thick] (C) -- (C');
\draw[dashed, thick] (D) -- (D');

\draw[very thick, ->] (1.5, 1.5, 1) -- (1.5, 1.5, -1);
\draw[full] (Al) -- (D1l) -- (B1l) -- cycle;
\draw[dashed, thick] (Bl) -- (Cl) -- (Dl);
\filldraw[opacity=0.1] (Al) -- (Bl) -- (Cl) -- (Dl) -- cycle;

\filldraw[opacity=0.1] (A) -- (A') -- (B') -- (B) -- cycle;
\filldraw[opacity=0.1] (A) -- (A') -- (D') -- (D) -- cycle;
\filldraw[opacity=0.1] (C) -- (C') -- (B') -- (B) -- cycle;
\filldraw[opacity=0.1] (C) -- (C') -- (D') -- (D) -- cycle;
\filldraw[opacity=0.1] (A) -- (B) -- (C) -- (D) -- cycle;
\filldraw[opacity=0.1] (A') -- (B') -- (C') -- (D') -- cycle;
\end{tikzpicture}

 \end{minipage}
 \begin{minipage}{0.5\textwidth}
    \captionof{figure}{Step 2: Illustration of a slanted prism with a triangular base, out of which a slanted prism with a parallelogram base has been cut out.}
    \label{fig:cutting a slanted prism}
 \end{minipage}
\end{figure}
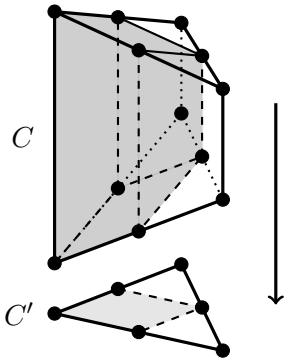

\noindent
\textit{Step 2.} Pick a slanted prism $C\in \cP_1$ and consider its projection along the $x_d$ coordinate. Note that the projections of the upper and lower face of $C$ coincide. The resulting shape is a $(d-1)$-dimensional polytope $C'$, which can be thought of as a ``shadow'' of $C$. We apply our algorithm recursively to cover $C'$ with parallelotopes, giving us a collection $\cP_{C'}$. Each of these $(d-1)$-dimensional parallelotopes lies in the horizontal hyperplane and has $(d-2)$-dimensional facets. For each facet $f$, take the vertical hyperplane $H_f$ containing $f$, then the hyperplanes $H_f$  further subdivide $C$. Each of the resulting cells of $C$ is a slanted prism with a parallelotope base: both the upper and the lower envelope of the subcell are parallelotopes, whose projections along the $x_d$ coordinate coincide. See Figure~\ref{fig:cutting a slanted prism} for an illustration. The conclusion of Step 2 is that we obtain a partition of $Q$ into slanted prisms with a parallelotope base, whose collection we denote by $\cP_2$.

\bigskip
\noindent
\textit{Step 3:} Let $P\in \cP_2$ be a slanted prism with a parallelotope base, and  denote the upper and the lower facet of $P$ by $F^+$ and $F^-$. Since $F^+$ and $F^-$ are parallelotopes with identical projections onto the horizontal hyperplane, their vertices are in one-to-one correspondance, defined by those pairs of vertices which have the same projection. Pick the closest pair of such vertices, and denote them by $x^-\in F^-$ and $x^+\in F^+$  (note that we may have $x^-=x^+$). Since both $F^+$ and $F^-$ are $(d-1)$-dimensional parallelotopes, there are vectors $v_1^-, \dots, v_{d-1}^-, v_1^+, \dots, v_{d-1}^+\in \bR^d$ such that the vertices of $F^-$ are $\{x^- +\sum_{i\in I} v_i^-:I\subseteq[d-1]\}$, and the vertices of the $F^+$ are $\{x^+ + \sum_{i\in I} v_i^+:I\subseteq[d-1]\}$. Moreover, since $F^+$ and $F^-$ have the same projections on the horizontal hyperplane, the vectors $v_i^-$ and $v_i^+$ can be chosen to be nonvertical and such that they differ only in the $x_d$ coordinate. We say that $P$ has \textit{parallel edges} if $v_i^-=v_i^+$ for some $i\in [d-1]$. The goal of Step 3 is to partition each $P\in \cP_2$ into two slanted prisms with parallel edges. 

\bigskip
If $v_i^-=v_i^+$ for some $i\in [d-1]$, we are done, so assume that $v_i^-\neq v_i^+$ for all $i\in [d-1]$. Let $H$ be the hyperplane through $x^+$, spanned by the vectors $v_1^-, v_2^+, \dots, v_{d-1}^+$. Let $H^+$ be the set of points on or above $H$, and let $H^{-}$ be the set of points on or below $H$. If we define $P^+=P\cap H^+$ and $P^-=P\cap H^-$, we claim that both $P^+$ and $P^-$ are slanted prisms with parallel edges.

\begin{figure}[ht]
\begin{minipage}{0.5\textwidth}
    \centering
    \tdplotsetmaincoords{120}{55}
    \begin{tikzpicture} [scale=2.5, tdplot_main_coords,
    full/.style={black,very thick},
    dots/.style={dotted,black,thick},
    ]
        \coordinate (A) at (0, 0, -1.3);
        \coordinate (B) at (1, 0, -0.55);
        \coordinate (C) at (1, 1, 0.2);
        \coordinate (D) at (0, 1, -0.55);
        
        \coordinate (A') at (0, 0, 2);
        \coordinate (B') at (1, 0, 1.7);
        \coordinate (C') at (1, 1, 1.5);
        \coordinate (D') at (0, 1, 1.8);
        
        \coordinate (A1) at (0, 0, 0.95);
        \coordinate (D1) at (0, 1, 0.75);
        
        \draw[dashed, very thick, blue] (A) -- (B);
        \draw[dots] (B) -- (C);
        \draw[dots] (B) -- (B');
        \draw[dashed, very thick, blue] (A1) -- (B');
        
        \filldraw[opacity=0.1] (A) -- (A') -- (B') -- (B) -- cycle;
        \filldraw[opacity=0.1] (A) -- (A') -- (D') -- (D) -- cycle;
        \filldraw[opacity=0.1] (C) -- (C') -- (B') -- (B) -- cycle;
        \filldraw[opacity=0.1] (C) -- (C') -- (D') -- (D) -- cycle;
        \filldraw[opacity=0.1] (A) -- (B) -- (C) -- (D) -- cycle;
        \filldraw[opacity=0.1] (A') -- (B') -- (C') -- (D') -- cycle;
        \filldraw[opacity=0.3, green] (A1) -- (B') -- (C') -- (D1) -- cycle;
        
        \node[vtx, label=left:$x^-\!\!+\!v_1^-\!\!+\!v_2^-$] at (A) {};
        \node[vtx, label=left:$x^-\!\!+\!v_2^-\!\!$] at (B) {};
        \node[vtx, label=right:$x^-$] at (C) {};
        \node[vtx, label=right:$x^-\!\!+\!v_1^-$] at (D) {};
        \node[vtx, label=left:$x^+ \!\! + \! v_1^+\!\! +\! v_2^+$] at (A') {};
        \node[vtx, label=above:$x^+\!\!+ \!v_2^+$] at (B') {};
        \node[vtx, label=right:$x^+$] at (C') {};
        \node[vtx, label=above:$x^+\!\!+\!v_1^+\!\!$] at (D') {};
        
        \node[vtx, label=left:$x^+\!\!+\!v_1^-\!\!+\!v_2^+$] at (A1) {};
        \node[vtx, label=right:$x^+\!\!+\!v_1^-$] at (D1) {};
        
        \node[label=30:$\,\,H$] at (A1) {};
        
        \draw[full] (A') -- (B');
        \draw[full] (C') -- (D');
        \draw[red, very thick] (A') -- (D');
        \draw[red, very thick] (C') -- (B');
        
        \draw[very thick, blue] (C) -- (D);
        \draw[full] (D) -- (A);
        \draw[full] (A) -- (A');
        \draw[full] (C) -- (C');
        \draw[full] (D) -- (D');
        
        \draw[very thick, blue] (C') -- (D1);
        \draw[red, very thick] (A1) -- (D1);
    \end{tikzpicture}
    
\end{minipage}
\begin{minipage}{0.5\textwidth}
    \captionof{figure}{Cutting the slanted prism using a hyperplane $H$ (highlighted in green), which is spanned by $v_1^-, v_2^+$ and passes through $x^+$. The top cell contains a pair of parallel edges, which a highlighted in red. The parallel edges of the bottom cell are highlighted in blue.}
    \label{fig:final step}
\end{minipage}
\end{figure}
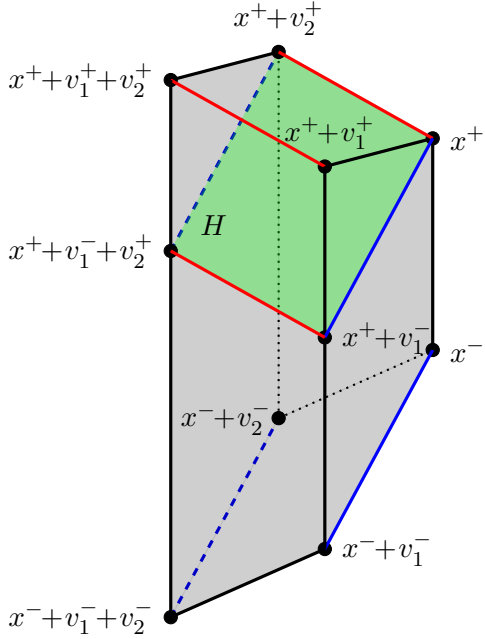

\begin{lemma}
The intesection $H\cap P$ is a parallelotope, and it lies (not necessarily strictly) below $F^+$ and above $F^-$.
\end{lemma}
\begin{proof}
We use the fact that the vertices of $F^+$ and $F^-$ are paired. Pick a pair of  vertices $y^-=x^-+\sum_{i\in I} v_i^-$ and $y^+=x^++\sum_{i\in I} v_i^+$, for some $I\subseteq [d-1]$, and  consider the vertical line $\ell$ passing through these two vertices. We show that $H$ intersects $\ell$ above $y^-$ and below $y^+$. If $1\notin I$, then $y^+$ lies on $H$. Hence, in this case $H$ intersects $\ell$ above $y^-$, namely at $y^+$.

If $1\in I$, then $H$ intersects $\ell$ in $x^+ + v_1^- +\sum_{i\in I\backslash \{1\}} v_i^+$, since $v_1^-$ and $v_1^+$ only differ in the last coordiante. The point $x^+ + v_1^- +\sum_{i\in I\backslash \{1\}} v_i^+$ is above $y^-$ since $x^+ + \sum_{i\in I\backslash \{1\}} v_i^+$ is above $x^- + \sum_{i\in I\backslash \{1\}} v_i^-$. Also, $x^+ + v_1^- +\sum_{i\in I\backslash \{1\}} v_i^+$ is below $y^+ $ if and only if $v_1^-$ is below $v_1^+$, which is true since $x^- x^+$ is the closest pair. More precisely, if $v_1^-$ was strictly above $v_1^+$, then the vertical segment between $x^-+v_1^-$ and $x^+ + v_1^+$ would be shorter than the segment between $x^-$ and $x^+$.

The above discussion shows that $H$ intersects the vertical faces of $P$ between $F^+$ and $F^-$, as claimed. Moreover, this means that $H\cap P$ also projects to the same parallelotope in the horizontal plane as $F^+, F^-$, so $H\cap P$ also must be a parallelotope.
\end{proof}

By the previous lemma, the cells $P^-$ and $P^+$ are slanted prisms with parallelotope bases, since their upper and lower faces are parallelotopes. Observe that both $P^-$ and $P^+$ have parallel edges: in $P^-$, both the bottom face $F^-$ and the top face $H\cap P$ have an edge parallel to $v_1^-$, and in $P^+$ both the bottom face $H\cap P$ and the top face $F^+$ have an edge parallel to $v_2^+$. Thus, Step 3 is completed by putting, for each $P\in \cP_2$, the two cells $P^-$ and $P^+$ into $\cP_3$ 

\bigskip
\noindent
\textit{Step 4.} In the final step of the algorithm, we take a slanted prism with parallel edges $P\in \cP_3$, and we cover it with parallelotopes. We may assume that $v_{d-1}^-=v_{d-1}^+$, and denote this vector simply by $v_{d-1}$. Consider the projection of the polytope $P$ along $v_{d-1}$. The polytope then projects to the vertical facet $F$ of $P$ with vertices \[\left\{x^- +\sum_{i\in I}v_i^- : I\subseteq[d-2] \right\}\cup \left\{x^+ + \sum_{i\in I}v_i^+ : I\subseteq [d-2]\right\},\]
i.e. all vertices which do not include $v_{d-1}^-, v_{d-1}^+$. Since $F$ is $(d-1)$-dimensional, we apply our algorithm to cover it with $(d-1)$-dimensional parallelotopes, whose collection we denote by $\cP_P$. From each parallelotope in $\cP_P$, we construct a $d$-parallelotope, by adding to all of its vertices the vector $v_{d-1}$, let $\cP_P'$ be the collection of these parallelotopes. Clearly, $\cP_P'$ is a covering of $P$. Finally, let $\cP_4=\bigcup_{P\in \cP_3}\cP_P'$.

 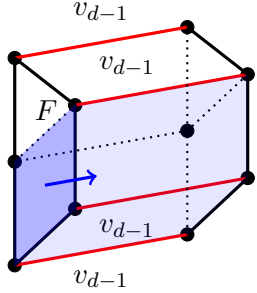
\begin{figure}[ht]
 \begin{minipage}{0.5\textwidth}
   \centering
\tdplotsetmaincoords{80}{35}
\begin{tikzpicture} [scale=1.4, tdplot_main_coords,
full/.style={black,very thick},
dots/.style={dotted,black,thick},
]

\coordinate (A) at (0, 0, -0.4);
\coordinate (D) at (0, 1, 0);
\coordinate (A') at (0, 0, 1.6);
\coordinate (D') at (0, 1, 1);

\coordinate (B) at (2, 0, 0.1);
\coordinate (C) at (2, 1, 0.5);
\coordinate (B') at (2, 0, 2.1);
\coordinate (C') at (2, 1, 1.5);

\coordinate (S) at (0, 0, 0.6);
\coordinate (T) at (2, 0, 1.1);

\node[vtx] at (A) {};
\node[vtx] at (B) {};
\node[vtx] at (C) {};
\node[vtx] at (D) {};
\node[vtx] at (A') {};
\node[vtx] at (B') {};
\node[vtx] at (C') {};
\node[vtx] at (D') {};
\node[vtx] at (S) {};
\node[vtx] at (T) {};

\node[label=right:$F$] at ($0.5*(S)+0.5*(A')$) {};
\node[label=below:$v_{d-1}$] at ($0.5*(A)+0.5*(B)$) {};
\node[label=below:$v_{d-1}$] at ($0.3*(C)+0.7*(D)$) {};
\node[label=above:$v_{d-1}$] at ($0.5*(A')+0.5*(B')$) {};
\node[label=above:$v_{d-1}$] at ($0.3*(C')+0.7*(D')$) {};

\draw[full, red] (A) -- (B);
\draw[full, red] (A') -- (B');
\draw[full, red] (C) -- (D);
\draw[full, red] (C') -- (D');

\draw[full] (A) -- (A') -- (D') -- (D) -- cycle;
\draw[full] (B) -- (C) -- (C') -- (B');
\draw[dots] (B) -- (B');
\draw[dots] (C') -- (T) -- (S) -- (D');

\draw[full, blue, ->] ($0.5*(D) + 0.5*(S)$) -- ($0.35*(D) + 0.35*(S) + 0.15*(C) + 0.15*(T)$);

\filldraw[opacity=0.3, blue] (A) -- (D) -- (D') -- (S) -- cycle;
\filldraw[opacity=0.1, blue] (D) -- (C) -- (C') -- (D') -- cycle;
\filldraw[opacity=0.1, blue] (A) -- (B) -- (C) -- (D) -- cycle;

\end{tikzpicture}
\end{minipage}
 \begin{minipage}{0.5\textwidth}
    \captionof{figure}{Decomposing a slanted prism in case $v_{d-1}^-=v_{d-1}^+$, by decomposing the vertical side $F$. The blue parallelogram participating in the decomposition of $F$ is extended to a parallelopiped in the direction of the blue vector.}
    \label{fig:recursively decomposing one side}
 \end{minipage}
\end{figure}
\medskip

\bigskip\noindent
\textbf{Analysis of the algorithm.} We now show that the proposed algorithm indeed produces a collection of parallelotopes $\cP=\cP_4$ which satisfies the conditions of Lemma~\ref{lemma:polytope decomposition}. First, we observe that the elements of $\cP$ cover $Q$ trivially.
Next, we show that $|\cP|=O_{d, h}(1)$.

\begin{lemma}\label{lemma:few parallelotopes}
The collection $\cP$ contains at most $O_{d, h}(1)$ parallelotopes.
\end{lemma}
\begin{proof}
We show, by induction on $d$, that there exists a function $f(d, h)$ such that, for any set of $h$ halfspaces $\cH$ in $\bR^d$ and any polytope $Q\in {\rm POL}(\cH)$ we have $|\cP|\leq f(d, h)$. For the base case $d=2$,  Lemma~\ref{lemma:plane decomposition} implies that $f(2, h)=4h$ suffices.

Assume that $d\geq 3$. First, observe that  $|\cP_1|\leq O_{h}(1)$. Indeed, the number of hyperplanes $H_f$ dividing $Q$ is the number of $(d-2)$-dimensional facets, which is at most $\binom{h}{2}$. Hence, these hyperplanes cut $Q$ into $O_{d,h}(1)$ cells.

To analyze Step 2, take any cell $C\in \cP_1$ and consider how many cells it contributes to $\cP_2$. Recall that in Step 2, we project $C$ to the horizontal hyperplane and decompose its shadow. The $(d-1)$-dimensional shadow of $C$, denoted by $C'$, has at most $O_{d,h}(1)$ facets and therefore the inductive hypothesis implies that the number of parallelotopes in the decomposition of $C'$ is at most $f(d-1, O_{d,h}(1))$. This shows that $|\cP_2|\leq f(d-1, O_{d,h}(1)) |\cP_1|\leq O_{d,h}(1)$.

In Step 3, we  subdivide each prism $P\in \cP_2$ into at most $2$ prisms, so we have $|\cP_3|\leq 2|\cP_2|$. Finally, Step 4, we apply induction on a $(d-1)$-dimensional vertical face $F$ of the prism $P$ which has parallel edges. The face $F$ has at most $2^{d-1}$ vertices and at most $2(d-1)$ facets, since $P$ is a slanted prism with a parallelotope base. Therefore, $F$ is covered by  at most $f(d, 2(d-1))$ parallelotopes and we conclude $|\cP|=|\cP_4|\leq f(d, 2(d-1)) \cdot |\cP_3|= O_{d,h}(1)$. This finishes the proof of the lemma.
\end{proof}

The last ingredient needed for the proof of the Lemma~\ref{lemma:polytope decomposition} is to show that the set of directions of facets of parallelotopes is contained in a set of size $O_{d,h}(1)$, independently of which element $Q\in {\rm POL}(\cH)$ we picked.

\begin{lemma}\label{lemma:bounding the directions}
For every collection $\cH$ of $h$ halfspaces in $\bR^d$, there exists a collection $\cH'$ of at most $O_{d, h}(1)$ halfspaces, such that the following holds. For any  $Q\in {\rm POL}(\cH)$, the collection of parallelotopes in $\cP$ produced by the above algorithm are intersections of translates of halfspaces of $\cH'$.    
\end{lemma}
\begin{proof}
We describe the elements of $\cH'$ by induction on $d$, noting that the base case $d=2$ is addressed in Lemma~\ref{lemma:plane decomposition}. In the inductive step, we consider the four steps of the algorithm separately.

\medskip
\noindent
\textit{Step 1.} First, let $\cH_1$ be the collection of halfspaces bounded by the vertical hyperplanes through the $(d-2)$-flats formed by the intersections $h_1\cap h_2$ for distinct $h_1, h_2\in \cH$. The observation is that each cell $C\in \cP_1$ is an intersection of translates of halfspaces of $\cH_1$, by construction. Furthermore, note that $\cH_1$ contains at most $2\binom{h}{2}=O_h(1)$ halfspaces, since each vertical hyperplane bounds two halfspaces.

\medskip
\noindent
\textit{Step 2.} Now consider all $(d-2)$-flats of the form $h_1\cap h_2$ for some $h_1, h_2\in \cH_1$, and denote by $\cF$ the set of $(d-1)$-dimensional halfspaces bounded by the projections of these flats to the horizontal hyperplane. The elements of $\cF$ are codimension 1 subspaces of the horizontal hyperplane, and thus we can apply the inductive hypothesis. Thus, if a cell $C\in \cP_1$ projects to a cell $C'\in {\rm POL}(\cF)$, then $C'$ can be covered with some parallelotopes  belonging to ${\rm POL}(\cF')$, for some $\cF'$ depending only on $\cF$ and having size $O_{d, h}(1)$. Let $\cH_2$ be the union of $\cH_1$ and the set of all halfspaces we get by extending the halfspaces of $\cF'$ in the vertical direction. Note that each prism in $\cP_2$ is an intersection of translates of halfspaces in $\cH_2$.

\medskip
\noindent
\textit{Step 3.} Let $\cV$ be the set of direction vectors of lines formed by the intersection of $d-1$ hyperplanes of $\cH_2$. Note that each $1$-dimensional face of a cell in $\cP_2$ is parallel to some $v\in \cV$. Now, consider the hyperplanes spanned by $d-1$ linearly independent  vectors from $\cV$ and let $\cH_3$ be the collection of halfspaces determined by these hyperplanes. The size of $\cH_3$ can be bounded as follows: first note that $|\cV|\leq \binom{|\cH_2|}{d-1}=O_{d, h}(1)$ and $|\cH_3|\leq 2\binom{|\cV|}{d-1}=O_{d, h}(1)$. Furthermore, each cell of $\cP_3$ is an intersection of translates of halfspaces from $\cH_3$.

\medskip
\noindent
\textit{Step 4.} For each vertical hyperplane in $h\in \cH_3$, let $\cF(h)$ be the set of intersections of all other hyperplanes in $\cH_3$ with $h$. This collection of $(d-2)$-dimensional hyperplanes bounds each vertical face $F$ of a cell $P\in \cP_3$. Therefore, there exists a collection of halfspaces $\cF'(h)$, of size $O_{d, h}(1)$, such that each parallelotope in the subdivision of $F$ belongs to ${\rm POL}(\cF'(h))$. Now, for each vector $v$ from the set of directions $\cV$ defined in Step 3, consider the hyperplane determined by $v$ and each element of $\cF'(h)$, and add the corresponding halfspaces to $\cH_4$. By construction, each parallelotope from $\cP=\cP_4$ is bounded by translates of halfspaces from $\cH_4$, and so it suffices to set $\cH'=\cH_4$. The size of $\cH_4$ is bounded by $|\cH_4|\leq \sum_{h\in \cH_3} |\cF'(h)| |\cV|=|\cH_3| \cdot |\cV| \cdot O_{d, h}(1) = O_{d, h}(1)$.
\end{proof}

Lemma~\ref{lemma:polytope decomposition} is now a direct corollary of the previous  results. It only remains to prove Proposition~\ref{prop:point-polytope incidence graphs}.

\begin{proof}[Proof of Proposition~\ref{prop:point-polytope incidence graphs}.]
Let $G$ be the incidence graph of the set of $n$ polytopes $X\subset \mbox{POL}(\cH)$ and set of $n$ points $Y\subset \mathbb{R}^d$.

Apply Lemma~\ref{lemma:polytope decomposition} to the set of halfspaces $\cH$, thus obtaining the set of halfspaces $\cH'$. Each polytope $Q\in X\subseteq {\rm POL}(\cH)$ is the union of a collection of parallelotopes $\cP_Q\subseteq {\rm POL}(\cH')$ of size  $|\cP_Q|\leq A$ for some $A=O_{d,h}(1)$. Say that the \emph{type} of a parallelotope  $T\in {\rm POL}(\cH')$ is the set of $d$ linear hyperplanes parallel to some face of $T$. Then there are at most $\binom{|\cH'|}{d}\leq B$ different types, where $B=O_{d,h}(1)$. Pick a type $\alpha$ randomly from the uniform distribution, and for every $Q\in X$, pick an element of $T_Q\in \cP_Q$ independently from the uniform distribution. Let $X'$ be the collection of parallelotopes $T_Q$ that have type $\alpha$, and let $G'$ be the indicedence graph of $(X',Y)$. Observe that $G'$ is isomorphic to a subgraph of $G$, so if $G$ is $K_{s,s}$-free, then $G'$ is $K_{s,s}$-free as well.

Given an edge $\{y,Q\}$ of $G$, there is at least one $T\in \cP_Q$ such that $y\in T$. Therefore, the probability that $\{y,T_Q\}$ is an edge of $G'$ is at least $1/AB$. Hence, the expected number of edges of $G'$ is at least $e(G)/AB$. Thus, there is a choice of $\alpha$ and $\{T_Q:Q\in Y\}$ such that $G'$ has at least $e(G)/AB$ edges, fix such a choice. After an appropriate affine transformation, each $\alpha$-type parallelotope can be made an axis-parallel box. Therefore, $G'$ is an incidence graph of  $n$ points and at most $n$ axis-parallel boxes in $\mathbb{R}^d$. By the result of Chan and Har-Peled \cite{CH23}, if $G'$ is $K_{s,s}$-free, then $e(G')=O_d(s(\log n/\log \log n)^{d-1})$. As $AB=O_{d,h}(1)$, we get $e(G)=O_{d,h}(s(\log n/\log \log n)^{d-1})$.
\end{proof}

\section{Polygon visibility graphs and ordered graphs}\label{sec:polygon visibility}

In this section, we prove Theorem \ref{thm:polygon_visibility}. We prepare the proof with a number of  general results, starting with a result about graphs of bounded VC-dimension. Then, we will present some results about the degree-boundedness of ordered graphs avoiding some fixed ordered matching, which might be of independent interest. 

Given a graph $G$, its \emph{VC-dimension} is the largest positive integer $d$ with the following property. There is a $d$-element subset $U\subset V(G)$ such that for every $T\subset U$, there is some vertex $v\in V(G)$ with $N(v)\cap U=T$. Lov\'asz and Szegedy \cite{LSz} proved that graphs of bounded VC-dimension admit a so called \emph{ultra-strong regularity lemma}. This lemma states that if $G$ has VC-dimension at most $d$ and given $\eps\in (0,1/2)$, then there is a partition of $V(G)$ into sets $V_1,\dots,V_t$ such that (i) $t=(1/\eps)^{O(d^2)}$, (ii) $|V_1|,\dots,|V_t|\in \{\lfloor|V(G)|/t\rfloor,\lceil |V(G)|/t\rceil\}$, (iii) all but $\eps t^2$ of the pairs $(i,j)$ with $1\leq i<j\leq t$ satisfy that either $e(G[V_i,V_j])<\eps |V_i||V_j|$ or $e(G[V_i,V_j])\geq (1-\eps)|V_i||V_j|$. We only use the following immediate qualitative consequence of this lemma.

\begin{lemma}\label{lemma:VC}
Let $d$ be an integer and $\beta\in (0,1]$, then for every $\gamma\in (0,1/2)$ there exists $\alpha=\alpha(d,\beta,\gamma)$ such that the following holds. Let $G$ be an $n$-vertex graph of VC-dimension at most $d$. If $e(G)\leq (1-\beta) \binom{n}{2}$, then there exist disjoint $A,B\subset V(G)$ such that $|A|=|B|\geq \alpha n$, and the bipartite graph between $G[A,B]$ has maximum degree $\gamma |A|$.
\end{lemma}

\begin{proof}
Let $\eps=\min\{\gamma/4,\beta/4\}$, and let $V_1,\dots,V_t$ be a partition of $V(G)$ given by the ultra-strong regularity lemma. Say that a pair $(i,j)$ with $1\leq i<j\leq t$ is \emph{bad} if $e(G[V_i,V_j])> \eps|V_i||V_j|$. Then the number of non-edges of $G$ contained in some bad pair is at most $\eps (n/t)^2\cdot t^2+(n/t)^2\cdot \eps t^2=2\eps n^2$. Hence, as $2\eps<\beta/2$, there must exist a good pair $(i,j)$. Then there are at most $\eps |V_i||V_j|\leq (\gamma/4)|V_i||V_j|$ edges in the bipartite graph $H=G[V_i,V_j]$. 

Let $A, B$ be the sets of $|V_i|/2$ vertices of lowest degree from both $V_i$ and $V_j$, and note that all of them must have degree at most $(\gamma/2)|V_i|$. Otherwise, the $|V_i|/2$ vertices of higest degree would each have degree larger than $(\gamma/2)|V_i|$, giving more than $(\gamma/4)|V_i||V_j|$ edges in $H$, contradiction.
\end{proof}

We now discuss some results about ordered graphs avoiding a fixed ordered matching. An \emph{ordered graph} is a graph with a linear ordering on its vertex set, denoted by $<$. An ordered graph $G$ is \emph{ordered bipartite} if we can write $V(G)=A\cup B$ such that $A<B$ (that is, $a<b$ for every $a\in A$ and $b\in B$), and every edge of $G$ has one endpoint in $A$ and in $B$. The celebrated Marcus-Tardos theorem \cite{MT04} is equivalent to the statement that if $M$ is an ordered bipartite matching and $G$ is ordered bipartite with both vertex classes of size $n$ not containing $M$, then $G$ has at most $O_M(n)$ edges. Here, the condition that $G$ is ordered bipartite can be relaxed as follows. If $G$ is an $n$ vertex ordered graph not containing $M$ as an ordered subgraph, then $G$ has $O_M(n)$ edges, see \cite{Tardos}. We prove that the following strengthening also holds.

\begin{theorem}\label{thm:matching_degree_bounded}
Let $M$ be an ordered bipartite matching. If $G$ is an ordered graph not containg an induced copy of $M$, and $G$ contains no $K_{s,s}$, then $G$ has $O_{M,s}(n)$ edges.
\end{theorem}

\begin{proof}
Let $V(M)=\{1,\dots,k\}$ with the natural ordering. For an integer $\ell$, let $M^{(\ell)}$ be an arbitrary ordered bipartite matching satisfying the following. The vertex set of $M^{(\ell)}$ can be partitioned into intervals $I_1,\dots,I_{k}$ of size $\ell$ such that $I_a<I_b$ if $a<b$, and there is a perfect matching between $I_a$ and $I_b$ if $ab\in E(M)$, otherwise there is no edge between $I_a$ and $I_b$. For every $\ell$, $M^{(\ell)}$ is clearly an ordered bipartite matching, so there exists $c_{\ell}$ such that every $n$-vertex ordered graph which does not contain $M^{(\ell)}$ as an ordered subgraph has at most $c_{\ell}n$ edges, by \cite{Tardos}.

Choose $\ell$ sufficiently large such that every bipartite graph with vertex classes of size $\ell$ and no copy of $K_{s,s}$ has at most $\ell^2/k^2$ edges. Such a choice is possible by the K\H{o}v\'ari-S\'os-Tur\'an theorem. We show that if $G$ does not contain $K_{s,s}$, but it contains $M^{(\ell)}$, then $G$ contains an induced copy of $M$. If this is true, then every $K_{s,s}$ and induced $M$-free ordered graph has at most $c_{\ell}n$ edges, finishing the proof.

If $G$ contains $M^{(\ell)}$ as an ordered subgraph, we can find $\ell$-element sets $J_1<\dots<J_k$ in $V(G)$ such that there is a complete matching $N_{ab}$ between $J_a$ and $J_b$ if $ab\in E(M)$. Note that if $G$ contains no $K_{s,s}$, then the number of edges between $J_i$ and $J_j$ for every $1\leq i<j\leq k$ is at most  $\ell^2/k^2$.  For every $ab\in E(M)$,  pick an edge $x_ax_b\in E(N_{ab})$ randomly among the $\ell$ edges from the uniform distribution. We show that with positive probability, $\{x_1,\dots,x_k\}$ induces of copy of $M$. For every $1\leq i<j\leq k$ for which $ij\not\in E(M)$, 
$$\mathbb{P}(x_ix_j\in E(G))=\frac{e(G[J_i,J_j])}{\ell^2}\leq \frac{1}{k^2}.$$
Hence, by the union bound,
$$\mathbb{P}(\exists ij\not\in E(M): x_ix_j\in E(G))\leq \frac{\binom{k}{2}}{k^2}<1.$$
Thus, there is a choice of $\{x_1,\dots,x_k\}$ inducing a copy of $M$ in $G$.
\end{proof}

Unfortunately, if $M$ has at least 2 edges, it is not true that $K_{s,s}$-free ordered graphs avoiding an induced copy of $M$ must have $O_M(sn)$ edges, i.e. the hidden constant in $O_{M, s}(n)$ may not be linear in $s$, which can be seen by considering, say $s=n^{1/2}$, and very dense random ordered graphs. The reason for this is that the family of ordered graphs avoiding an induced copy of $M$ does not have the density-EH property. 

However, if the complement of $G$ avoids bipartite-induced copies of some ordered bipartite matching, then we can establish the density-EH property. Given an ordered bipartite graph $H$, say that an ordered graph $G$ contains a \emph{bipartite-induced} copy of $H$ if there exists $A,B\subset V(G)$ such that $A<B$, and the bipartite subgraph of $G$ between $A$ and $B$ is isomorphic to $H$ (note that we do not require $A$ and $B$ to be independent sets). A similar result is established in \cite{KPT}, but we present a proof for completeness as well. 

\begin{theorem}\label{thm:matching_EH}
Let $G$ be an ordered graph on $n$ vertices, and let $c>0$. Let $M$ be an ordered bipartite matching, then there exists $\eps=\eps(c,M)>0$ such that the following holds. If $G$ has at least $cn^2$ edges, and the complement of $G$ contains no bipartite-induced copy of $M$, then $G$ contains a biclique of size at least $\eps n$.
\end{theorem}
\begin{proof}
We divide the proof into three simple claims.

\begin{claim}
There exist $A,B\subset V(G)$ such that $A<B$, $|A|=|B|\geq \delta n$ for some $\delta=\delta(c)>0$, and there are at least $c^2|A||B|$ edges between $A$ and $B$.
\end{claim}

\begin{proof}
 Without loss of generality, let $\{1,\dots,n\}$ be the vertex set of $G$, endowed with the natural ordering. Let $A_0=\{1,\dots,k\}$ and $B_0=\{k+1,\dots,n\}$, where $k$ is randomly chosen from the uniform distribution. Say that an edge $f=\{x,y\}$, with $x<y$,  is \emph{cut} if $x\in A_0$ and $y\in B_0$, and let $X$ be the number of cut edges.  Then $\mathbb{P}(f\mbox{ is cut})=\frac{y-x}{n}$, so $$\mathbb{E}(X)=\sum_{\{x,y\}\in E(G)}\frac{y-x}{n}\geq \frac{cn^2}{2}\cdot \frac{c}{2}=\frac{c^2n^2}{4}.$$
Here, the first inequality follows as there are at most $cn^2/2$ edges $\{x,y\}$ with $y-x\leq cn/2$. Hence, there is a choice for $k$ such that $X\geq c^2n^2/4$. This is only possible if $\min\{k,n-k\}\geq c^2n/4$ also holds. Let $\delta=c^2/4$, and let $A\subset A_0$ and $B\subset B_0$ be random $\delta n$ size subsets chosen from the uniform distribution. Then the expected number of edges between $A$ and $B$ is at least $\frac{|A||B|}{|A_0||B_0|}\cdot c^2n^2/4\geq c^2|A||B|$.
\end{proof}

Let $H_0$ be the bipartite subgraph between $A$ and $B$, and let $H$ be the bipartite complement of $H_0$. Then $H$ has at most $(1-c^2)|A||B|$ edges, and $H$ contains no bipartite-induced copy of $M$.

\begin{claim}
For every $\lambda>0$, there exist $A_1\subset A$ and $B_1\subset B$ such that $|A_1|=|B_1|\geq \alpha |A|$ for some $\alpha=\alpha(M,c,\lambda)$, and the maximum degree of $H[A_1\cup B_1]$ is at most $\lambda |A_1|$.
\end{claim}

\begin{proof}
It is known (and easy to prove) that ordered graphs with no bipartite-induced copy of a given ordered bipartite graph $\Gamma$ have VC-dimension $O_\Gamma(1)$. Therefore,  Lemma \ref{lemma:VC} gives the desired result.
\end{proof}

Let $k$ be the number of edges of $M$, and set $\lambda=1/(8k^3)$. Let $A_1,B_1$ be the sets guaranteed by the previous claim with respect to $\lambda$, so $|A_1|=|B_1|\geq \alpha |A|\geq \alpha\delta n$ for some $\alpha=\alpha(M,c,\lambda)=\alpha(M,c)$.

\begin{claim}\label{claim:finding an empty biclique}
There exists $A_2\subset A_1$ and $B_2\subset B_1$ such that $|A_2|=|B_2|\geq \lambda |A_1|$, and there are no edges between $A_2$ and $B_2$ in $H$.
\end{claim}

\begin{proof}
Assume that there exist no such $A_2,B_2$, then we show that $H$ contains a bipartite induced copy of $M$. Let $V(M)=\{1,\dots,2k\}$, and let $I_1<\dots<I_k$ be a partition of $A_1$ into intervals of size $|A_1|/k$, and $I_{k+1}<\dots<I_{2k}$ be a partition of $B_1$ into intervals of size $|B_1|/k$. For every $ab\in E(M)$, there is a matching $N_{ab}$ of size $|I_a|/2=|A_1|/2k$ between $I_a$ and $I_b$. Otherwise, by K\"onig's theorem, there exist $X\subset I_a$ and $Y\subset I_b$ such that $|X|=|Y|=|A_1|/(2k)>\lambda|A_1|$ and with no edges between $X$ and $Y$. Let $V(N_{ab})=J_a\cup J_b$ with $J_a\subset I_a$ and $J_b\subset I_b$. For every $ab\in E(M)$, pick a random edge $x_ax_b$ of $N_{ab}$ independently from the uniform distribution. We show that $\{x_1,\dots,x_{2k}\}$ induces a copy of $M$ with positive probability. For every $ij\not\in E(M)$, $$\mathbb{P}(x_ix_j\in E(H))= \frac{e(H[J_i,J_j])}{|J_i||J_j|}\leq \frac{|J_i|(\lambda |A_1|)}{|J_i|(|A_1|/2k)}=\frac{1}{4k^2},$$
where  in the first inequality we used that the maximum degree of $H[A_1\cup B_1]$ is at most $\lambda n$. Hence, by the union bound, $$\mathbb{P}(\exists ij\not\in E(M): x_ix_j\in E(G))\leq \frac{\binom{2k}{2}}{4k^2}<1.$$
Thus, there is a choice of $\{x_1,\dots,x_{2k}\}$ inducing a copy of $M$ in $H$.
\end{proof}

As $\lambda|A_1|\geq \lambda \alpha |A|\geq \lambda\alpha\delta n$, the sets $A_2$ and $B_2$ from Claim~\ref{claim:finding an empty biclique} span a biclique of size at least $\lambda\alpha\delta n$ in $G$. Here, $\eps:=\lambda\alpha\delta>0$ only depends on $M$ and $c$, as promised.
\end{proof}

Given a Jordan curve $K$, say that two sets $A,B\subset K$ are \emph{separated} if $K$ can be cut into two subcurves such that one of the subcurves contains $A$, the other $B$. The following lemma sharpens Corollary 3.5 in \cite{AK25}.

\begin{figure}
\begin{center}
\begin{tikzpicture}[scale=0.8]
\node[vtx,label=below:$1$] (x1) at (-2.5,0) {};
\node[vtx,label=below:$2$] (x2) at (-1.5,0) {};
\node[vtx,label=below:$3$] (x3) at (-0.5,0) {};
\node[vtx,label=below:$4$] (x4) at (0.5,0) {};
\node[vtx,label=below:$5$] (x5) at (1.5,0) {};
\node[vtx,label=below:$6$] (x6) at (2.5,0) {};

\draw[dotted] (0,-1.5) -- (0,1.5); 
\draw (x2) to[out=60, in=120] (x4);
\draw (x2) to[out=60, in=120] (x6);
\draw (x1) to[out=60, in=120] (x5);
\draw (x3) to[out=60, in=120] (x5);
\draw[dashed] (x2) to[out=-60, in=-120] (x5);

\end{tikzpicture}
\caption{A double cherry. The dashed arc indicates a non-edge.}
\label{fig:double_cherry}
\end{center}
\end{figure}
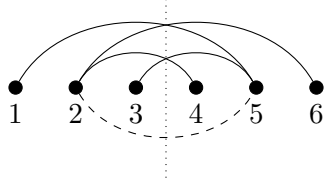

\begin{lemma}\label{lemma:visibility1}
Let $K$ be a Jordan curve, and let $A,B\subset K$ be finite sets such that $A$ and $B$ are separated. If $G$ is the visibility graph of $A\cup B$ with respect to $K$, and $G$ is $K_{s,s}$-free, then there are $O(s(|A|+|B|))$ edges between $A$ and $B$.
\end{lemma}

\begin{proof}
Order $A\cup B$ in clockwise order on the curve $K$ such that the points of $A$ are before the points of $B$, then this defines an ordering $<$ of $G$ with $A<B$. 

Say that an ordered bipartite graph $H$ is a \emph{double cherry} if it has the following form. The vertex classes of $H$ are $\{1,2,3\}$ and $\{4,5,6\}$, the pairs  $\{2,4\},\{2,6\},\{1,5\},\{3,5\}$ are edges, and $\{2,5\}$ is not an edge. The rest of the pairs between the two vertex classes can be either edges or non-edges. See Figure \ref{fig:double_cherry} for an illustration. It is proved in \cite{AK25} that a curve visibility graph does not contain a bipartite-induced copy of a double cherry.

 Define a collection $\mathcal{F}$ of horizontal and vertical segments in the plane as follows. Let $a_1<\dots<a_m$ be the elements of $A$, and $b_1<\dots<b_n$ be the elements of $B$. For every $k\in [m]$, add to $\mathcal{F}$ the closed vertical segment with endpoints $(k,\ell)$ and $(k,\ell')$, where $b_{\ell}\in B$ is the first neighbour of $a_k$, and $b_{\ell'}\in B$ is the last neighbour with respect to $<$. If $a_k$ has does not have neighbours in $B$, do not add anything. Also, for every $\ell\in [n]$, add to $\mathcal{F}$ the open horizontal segment with endpoints $(k,\ell)$ and $(k',\ell)$, where $a_{k}\in A$ is the first neighbour of $b_{\ell}$, and $a_{k'}\in A$ is the last neighbour with respect to $<$. Then $\mathcal{F}$ contains at most $|A|+|B|$ segments. 
 
 Let $Q$ be the intersection graph between the horizontal and vertical segments. Then $Q$ contains at least $e(G[A, B])$ edges. Indeed, if $a_ib_j$ is an edge of $G$, then there exist $a_{i_1},a_{i_2}\in A$ and $b_{j_1},b_{j_2}\in B$ with $i_1\leq i\leq i_2$ and $j_1\leq j\leq j_2$ such that $a_{i_1}$ is the first neighbour of $b_j$ and $a_{j_2}$ is the last, and similarly for $b_{j_1}$ and $b_{j_2}$. In this case, the edge $a_ib_j$ corresponds to the intersection of the horizontal segment with endpoints $(i_1,j), (i_2,j)$ and the vertical segment with endpoints $(i,j_1), (i,j_2)$, both of which belong to $\cF$. Note that these two segments intersect in the point $(i,j)$, and therefore every intersection of $Q$ corresponds to an edge of $G[A, B]$.
 
\begin{claim}
If a horizontal and a vertical segment in $\mathcal{F}$ intersect in some point $(i,j)$, then $(a_i,b_j)$ is an edge of $G$.
\end{claim}

\begin{proof}
Let $(i,j_1)$ and $(i,j_2)$ be the endpoints of the horizontal segment, and $(i_1,j)$ and $(i_2,j)$ be the endpoints of the vertical segment. By the definition of the segments, we know $a_{i_1}b_j,a_{i_2}b_j,a_ib_{j_1},a_{i}b_{j_2}$ are edges of $G$. If $a_ib_j$ is not an edge of $G$, then we must have $i_1<i<i_2$ and $j_1<j<j_2$. But then $G$ contains a bipartite-induced copy of a double cherry with vertices $a_{i_1}, a_i, a_{i_2}, b_{j_1}, b_j, b_{j_2}$, which is impossible.
\end{proof}

From the previous claim, we deduce that if $G$ is $K_{s,s}$-free, then $Q$ is also $K_{s,s}$-free. But then, $Q$ has $O(s(|A|+|B|))$ edges by Proposition \ref{prop:vertical vs horizontal}, which gives that $e(G[A,B])\leq e(Q)\leq O(s(|A|+|B|))$. 
\end{proof}

\begin{theorem}
Let $G$ be an $n$-vertex visibility graph with respect to an $x$-monotone Jordan curve $K$. If $G$ is $K_{s,s}$-free, then $G$ has $O(sn)$ edges.
\end{theorem}

\begin{proof}
Let $P\subset K$ be the vertex set of $G$. If $K$ is an $x$-monotone Jordan curve, let us denote by $x_\ell, x_r$ (one of) its leftmost and rightmost points. Then, two internally disjoint arcs of $K$ join $x_\ell$ and $x_r$ -- let us denote them by $K_\ell$ and $K_{u}$, standing for the \textit{lower} and the \textit{upper} part of $K$, respectively. Then $K_{\ell}$ and $K_u$ are $x$-monotone curves, i.e. curves such that every vertical line intersects them in at most one point. Let $P_{\ell}=P\cap K_{\ell}$ and $P_u=P\cap K_u$. By Lemma \ref{lemma:visibility1}, we have $e(G[P_{\ell},P_u])=O(sn)$. We prove that $e(G[P_{\ell}])=O(sn)$ holds as well, which, by symmetry, implies $e(G[P_{u}])=O(sn)$ and finishes the proof.

Let $G'=G[P_{\ell}]$, and order the elements of $P_{\ell}$ from left to right, this gives an ordering of $G'$. Let $\mathcal{M}$ be the family of ordered bipartite graphs $M$ of the following form: the vertex classes of $M$ are $\{1,2,3\}$ and $\{4,5,6\}$, the pairs $\{1,6\}$, $\{2,4\}$, $\{3,5\}$ are edges of $M$, but $\{2,5\}$ is not an edge.

\begin{claim}\label{claim:x monotone}
    $G'$ contains no bipartite-induced copy of a member of $\mathcal{M}$. 
\end{claim}

\begin{proof}
Assume that $p_1<p_2<p_3<q_1<q_2<q_3$ induces a bipartite-induced copy of some $M\in\mathcal{M}$ in $G'$. Then $p_1q_3$ is an edge. This means that the part of $K_{\ell}$ between the points $p_1$ and $q_3$ is below the segment $p_1q_3$, and $K_{u}$ is above this segment. Furthermore, $p_2q_1$ is also an edge of $G'$, so the part of $K_{\ell}$ between $p_2$ and $q_1$ is below $p_2q_1$. This is also true for $p_3$ and $q_2$. But then the segment $p_2q_2$ cannot intersect $K_{\ell}$ as it is above every point of $p_2q_1\cup p_3q_2$, and it also cannot intersect $K_u$ as it is below $p_1q_3$. Hence, $p_2q_2$ is an edge as well, contradiction.
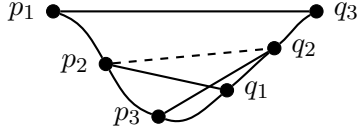
\begin{figure}[h]
    \centering
    \vspace{-0.2cm}
    \begin{tikzpicture}[scale=0.7]
        \node[vtx, label=left:$p_1$] (p1) at (0,3) {};
        \node[vtx, label=left:$p_2$] (p2) at (1,2) {};
        \node[vtx, label=left:$p_3$] (p3) at (2,1) {};

        \node[vtx, label=right:$q_3$] (q3) at (5,3) {};
        \node[vtx, label=right:$q_2$] (q2) at (4.2,2.3) {};
        \node[vtx, label=right:$q_1$] (q1) at (3.3,1.5) {};

        \draw[thick] (p1) to[out=-20, in=120] (p2) to[out=-60, in=160] (p3) to[out=-20, in=-140] (q1) to[out=40, in=-140] (q2) to[out=40, in=-150] (q3);

        \draw[thick] (p1) -- (q3);
        \draw[thick] (p2) -- (q1);
        \draw[thick] (p3) -- (q2);
        \draw[thick, dashed] (p2) -- (q2);
    \end{tikzpicture}
    \caption{Illustration of proof of Claim~\ref{claim:x monotone}.}
    \label{fig:enter-label}
\end{figure}
\end{proof}
\vspace{-0.3cm}

Let $M_0$ be the member of $\mathcal{M}$ in which the 5 unspecified pairs of vertices $\{i,j\}$ with $i\in \{1,2,3\}, j\in \{4,5,6\}$ are not edges. Then $M_0$ is a matching. Also, let $M_1$ be the member of $\mathcal{M}$ in which $\{1,4\}$, $\{2,5\}$, $\{3,6\}$ are not edges, and all other pairs $\{i,j\}$ with $i\in \{1,2,3\}, j\in \{4,5,6\}$ are edges. Moreover, let $M_2$ be the bipartite complement of $M_1$, that is, $\{1,4\},\{2,5\},\{3,6\}$ are edges, and the rest of the pairs are non-edges. See Figure \ref{fig:matching} for an illustration. Then $M_2$ is also a matching, and the complement of $G'$ contains no bipartite-induced copy of $M_2$. Let $\mathcal{F}$ be the family of graphs $H$ which have an ordering such that $H$ contains no bipartite-induced copy of $M_0$ and the complement of $H$ contains no bipartite-induced copy of $M_2$. Then $\mathcal{F}$ is a hereditary family of graphs and $G'\in\mathcal{F}$. Moreover, $\mathcal{F}$ is degree-bounded by Theorem \ref{thm:matching_degree_bounded}, and it has the density-EH property by Theorem \ref{thm:matching_EH}. Therefore, by Corollary \ref{thm:master}, every $n$-vertex $K_{s,s}$-free member of $\mathcal{F}$ has $O(sn)$ edges, so this holds for $G'$ as well. This finishes the proof.\end{proof}

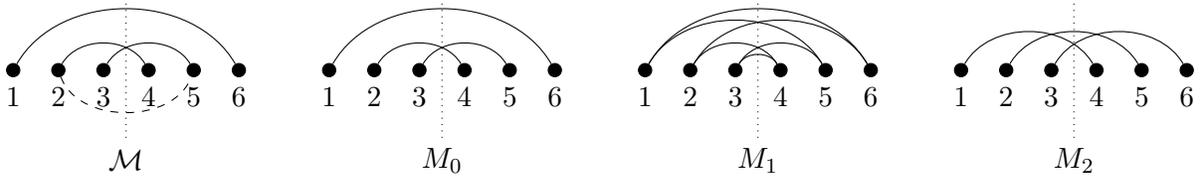
\begin{figure}[h]
\begin{center}\vspace{-0.3cm}
\begin{tikzpicture} [scale=0.6]
\node[vtx,label=below:$1$] (x1) at (-2.5,0) {};
\node[vtx,label=below:$2$] (x2) at (-1.5,0) {};
\node[vtx,label=below:$3$] (x3) at (-0.5,0) {};
\node[vtx,label=below:$4$] (x4) at (0.5,0) {};
\node[vtx,label=below:$5$] (x5) at (1.5,0) {};
\node[vtx,label=below:$6$] (x6) at (2.5,0) {};
\draw[dotted] (0,-1.5) -- (0,1.5); 
\draw (x1) to[out=60, in=120] (x6);
\draw (x2) to[out=60, in=120] (x4);
\draw (x3) to[out=60, in=120] (x5);
\draw[dashed] (x2) to[out=-70, in=-110] (x5);
\node[] at (0,-2) {$\mathcal{M}$};

\node[] (A) at (7,0) {};
\node[vtx,label=below:$1$] (x1) at ($(A)+(-2.5,0)$) {};
\node[vtx,label=below:$2$] (x2) at ($(A)+(-1.5,0)$) {};
\node[vtx,label=below:$3$] (x3) at ($(A)+(-0.5,0)$) {};
\node[vtx,label=below:$4$] (x4) at ($(A)+(0.5,0)$) {};
\node[vtx,label=below:$5$] (x5) at ($(A)+(1.5,0)$) {};
\node[vtx,label=below:$6$] (x6) at ($(A)+(2.5,0)$) {};
\draw[dotted] ($(A)+(0,-1.5)$) -- ($(A)+(0,1.5)$); 
\draw (x1) to[out=60, in=120] (x6);
\draw (x2) to[out=60, in=120] (x4);
\draw (x3) to[out=60, in=120] (x5);
\node[] at ($(A)+(0,-2)$) {$M_0$};

\node[] (B) at (14,0) {};
\node[vtx,label=below:$1$] (x1) at ($(B)+(-2.5,0)$) {};
\node[vtx,label=below:$2$] (x2) at ($(B)+(-1.5,0)$) {};
\node[vtx,label=below:$3$] (x3) at ($(B)+(-0.5,0)$) {};
\node[vtx,label=below:$4$] (x4) at ($(B)+(0.5,0)$) {};
\node[vtx,label=below:$5$] (x5) at ($(B)+(1.5,0)$) {};
\node[vtx,label=below:$6$] (x6) at ($(B)+(2.5,0)$) {};
\draw[dotted] ($(B)+(0,-1.5)$) -- ($(B)+(0,1.5)$); 
\draw (x1) to[out=60, in=120] (x6);
\draw (x1) to[out=60, in=120] (x5);
\draw (x2) to[out=60, in=120] (x4);
\draw (x2) to[out=60, in=120] (x6);
\draw (x3) to[out=60, in=120] (x4);
\draw (x3) to[out=60, in=120] (x5);
\node[] at ($(B)+(0,-2)$) {$M_1$};

\node[] (C) at (21,0) {};
\node[vtx,label=below:$1$] (x1) at ($(C)+(-2.5,0)$) {};
\node[vtx,label=below:$2$] (x2) at ($(C)+(-1.5,0)$) {};
\node[vtx,label=below:$3$] (x3) at ($(C)+(-0.5,0)$) {};
\node[vtx,label=below:$4$] (x4) at ($(C)+(0.5,0)$) {};
\node[vtx,label=below:$5$] (x5) at ($(C)+(1.5,0)$) {};
\node[vtx,label=below:$6$] (x6) at ($(C)+(2.5,0)$) {};
\draw[dotted] ($(C)+(0,-1.5)$) -- ($(C)+(0,1.5)$); 
\draw (x1) to[out=60, in=120] (x4);
\draw (x2) to[out=60, in=120] (x5);
\draw (x3) to[out=60, in=120] (x6);
\node[] at ($(C)+(0,-2)$) {$M_2$};
\end{tikzpicture}
\caption{
    The family $\mathcal{M}$, and ordered bipartite graphs $M_0, M_1, M_2$.
}\qedhere
\label{fig:matching}
\end{center}
\end{figure}

\vspace{-0.5cm}
\begin{theorem}
Let $G$ be an $n$-vertex visibility graph with respect to a star-shaped Jordan curve $K$. If $G$ is $K_{s,s}$-free, then $G$ has $O(sn)$ edges.
\end{theorem}

\begin{proof}
Let $P\subset K$ be the vertex set of $G$. Let $c\in K^*$ be a point such that every half-line starting at $c$ intersects $K$ in at most one point. Let $\ell$ be an arbitrary line through $c$, then $\ell$  intersects $K$ in exactly two points. Therefore, $\ell$ cuts $K$ into two parts, $K_1$ and $K_2$, let $P_i=K_i\cap P$ for $i=1,2$. By Lemma~\ref{lemma:visibility1}, we have $e(G[P_1,P_2])=O(sn)$. We prove that $e(G[P_{i}])=O(sn)$ holds as well for $i=1,2$, then we are done. By symmetry, it is enough to consider $P_1$.

Let $G_1=G[P_{1}]$, and order the elements of $P_{1}$ in clockwise direction starting from the one of the endpoints of $K_1$. This gives an ordering of $G_1$. Let $\mathcal{M}$ be the family of ordered bipartite graphs $M$ of the following form: the vertex classes of $M$ are $\{1,2\}$ and $\{3,4\}$, the pairs $\{1,3\}$ and $\{2,4\}$ are edges, but $\{1,4\}$ is not an edge.

\begin{claim}\label{claim:star-shaped}
    $G_1$ contains no bipartite-induced copy of a member of $\mathcal{M}$. 
\end{claim}

\begin{proof}
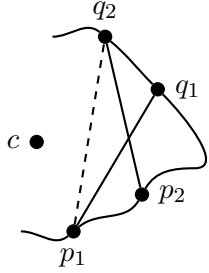
\begin{figure}[h]
    \centering
    \vspace{-0.2cm}
    \begin{tikzpicture}[scale=0.7]
        \node[vtx, label=left:$c$] (c) at (0, 0) {};
        \node[vtx, label=below:$p_1$] (p1) at (0.7,-1.7) {};
        \node[vtx, label=right:$p_2$] (p2) at (2,-1) {};
        \node[vtx, label=right:$q_1$] (q1) at (2.3,1) {};
        \node[vtx, label=above:$q_2$] (q2) at (1.3,2) {};
   
        \draw[thick] ($(p1)+(-1, 0)$) to[out=0, in=-140] (p1) to[out=40, in=-120] (p2) to[out=60, in=180] ($(p2)+(1, 0.5)$) to[out=0, in=-45] (q1) to[out=135, in=-40] (q2) to[out=140, in=0] ($(q2)+(-1, 0)$);
        
        \draw[thick] (p1) -- (q1);
        \draw[thick] (p2) -- (q2);
        \draw[thick, dashed] (p1) -- (q2);
    \end{tikzpicture}
    \vspace{-0.3cm}
    \caption{Illustration of proof of Claim~\ref{claim:star-shaped}.}
    \vspace{-0.3cm}
    \label{fig:star-shaped}
\end{figure}
Assume that $p_1<p_2<q_1<q_2$ induces a bipartite-induced copy of some $M\in\mathcal{M}$ in $G_1$. Then $p_1q_1$ and $p_2q_2$ are edges. Consider the 4 half-lines starting from $c$ and going through $p_1,p_2,q_1,q_2$, denote them by $\ell_1,\ell_2,\ell_3,\ell_4$. These half-lines follow a clockwise order. The intersection point of $\ell_2$ and the segment $p_1q_2$ must be before $p_2$, using that $K$ is star-shaped. Similarly, the intersection of $\ell_3$ and $p_2q_2$ is before $q_1$. This implies that $p_1q_1$ and $p_2q_2$ intersect in some point $x$, and the quadrilateral $cp_1xq_2$ is convex and is contained in $K^*$. But then $p_1q_2$ is contained in $K$, which means that $p_1q_2$ is an edge of $G_1$, contradiction.
\end{proof}

Let $M\in\mathcal{M}$ be the unique member with $\{2,3\}$ being a non-edge. Then $M$ is a matching, and the bipartite complement of $M$ is also a matching. Let $\mathcal{F}$ be the family of graphs $H$ which have an ordering such that  $H$ does not contain a bipartite-induced copy of $M$. Then $\mathcal{F}$ is a hereditary family of graphs and $G_1\in\mathcal{F}$. Moreover, $\mathcal{F}$ is degree-bounded by Theorem \ref{thm:matching_degree_bounded}, and it has the density-EH property by Theorem \ref{thm:matching_EH}. Therefore, by Theorem \ref{thm:master}, every $n$ vertex $K_{s,s}$-free member of $\mathcal{F}$ has $O(sn)$ edges, so this holds for $G'$ as well. This finishes the proof.
\end{proof}

\end{document}